\documentclass[reqno]{amsart}

\usepackage{amsmath,amsthm}
\usepackage{dsfont} 
\usepackage[english]{babel}
\usepackage{color}
\usepackage{amssymb}
\usepackage{mathrsfs}
\usepackage{graphicx}
\usepackage[font=footnotesize, labelfont=bf]{caption}
\usepackage{mathtools}
\usepackage[colorlinks=true, pdfstartview=FitV, linkcolor=blue, citecolor=blue, urlcolor=blue,pagebackref=false]{hyperref}
\usepackage{microtype}
\usepackage{enumitem}
\usepackage[normalem]{ulem}
\usepackage{soul}

\usepackage{float}

\newtheorem{theorem}{Theorem}[section]
\newtheorem{proposition}[theorem]{Proposition}
\newtheorem{lemma}[theorem]{Lemma}

\theoremstyle{definition}
\newtheorem{remark}[theorem]{Remark}
\newtheorem{example}[theorem]{Example}
\newtheorem{definition}[theorem]{Definition}

\newtheorem{assumption}{Assumption}
\numberwithin{equation}{section}

\definecolor{myblue}{RGB}{0,0,128}

\def\Xint#1{\mathchoice
	{\XXint\displaystyle\textstyle{#1}}%
	{\XXint\textstyle\scriptstyle{#1}}%
	{\XXint\scriptstyle\scriptscriptstyle{#1}}%
	{\XXint\scriptscriptstyle\scriptscriptstyle{#1}}%
	\!\int}
\def\XXint#1#2#3{{\setbox0=\hbox{$#1{#2#3}{\int}$ }
		\vcenter{\hbox{$#2#3$ }}\kern-.57\wd0}}

\def\dashint{\Xint-}

\newcommand{\dx}{\,\mathrm{d}x}
\newcommand{\dy}{\,\mathrm{d}y}

\newcommand{\e}{\varepsilon}
\newcommand{\dist}{{\rm{dist}}}
\renewcommand{\L}{\mathcal{L}}
\newcommand{\w}{\omega}

\newcommand{\R}{\mathbb{R}}
\newcommand{\Z}{\mathbb{Z}}
\newcommand{\N}{\mathbb{N}}
\newcommand{\Q}{\mathbb{Q}}

\newcommand{\F}{\mathcal{F}}

\renewcommand{\P}{\mathbb{P}}

\newcommand{\overbar}[1]{\mkern 1.5mu\overline{\mkern-1.5mu#1\mkern-1.5mu}\mkern 1.5mu}

\begin{document}

\author{Matthias Ruf}
\address[Matthias Ruf]{Section de math\'ematiques, Ecole Polytechnique F\'ed\'erale de Lausanne, Station 8, 1015 Lausanne, Switzerland}
\email{matthias.ruf@epfl.ch}

\author{Mathias Sch\"affner}
\address[Mathias Sch\"affner]{Institut f\"ur Mathematik, MLU Halle-Wittenberg, Theodor-Lieser-Stra\ss e 5, 06120 Halle (Saale), Germany}
\email{mathias.schaeffner@mathematik.uni-halle.de}

\title[New homogenization results for convex integral functionals]{New homogenization results for convex integral functionals and their Euler-Lagrange equations}

\begin{abstract}	
We study stochastic homogenization for convex integral functionals
$$
u\mapsto \int_D W(\w,\tfrac{x}\e,\nabla u)\dx,\quad\mbox{where}\quad u:D\subset \R^d\to\R^m,
$$
defined on Sobolev spaces. Assuming only stochastic integrability of the map $\w\mapsto W(\w,0,\xi)$, we prove homogenization results under two different sets of assumptions, namely
\begin{itemize}
\item[$\bullet_1$] $W$ satisfies superlinear growth quantified by the stochastic integrability of the Fenchel conjugate $W^*(\cdot,0,\xi)$ and a mild monotonicity condition that ensures that the functional does not increase too much by componentwise truncation of $u$,
\item[$\bullet_2$] $W$ is $p$-coercive in the sense $|\xi|^p\leq W(\w,x,\xi)$ for some $p>d-1$.
\end{itemize}
Condition $\bullet_2$ directly improves upon earlier results, where $p$-coercivity with $p>d$ is assumed and $\bullet_1$ provides an alternative condition under very weak coercivity assumptions and additional structure conditions on the integrand. We also study the corresponding Euler-Lagrange equations in the setting of Sobolev-Orlicz spaces. In particular, if $W(\w,x,\xi)$ is comparable to $W(\w,x,-\xi)$ in a suitable sense, we show that the homogenized integrand is differentiable.

%
\end{abstract}

\maketitle
{\small
	\noindent
	
	\noindent
}	


\section{Introduction}

We revisit the problem of stochastic homogenization of vectorial convex integral functionals. For a bounded Lipschitz domain $D\subset \R^d$, $d\geq2$, we consider integral functionals of the form 
\begin{equation}\label{intro:int}
F_\e(\w,\cdot,D):\, W^{1,1}(D)^m\to[0,+\infty],\qquad F_\e(\w,u,D)=\int_D W(\w,\tfrac{x}\e,\nabla u(x))\dx,
\end{equation}
where $W:\Omega\times\R^d\times \R^{m\times d}\to [0,+\infty)$ is a random integrand which is stationary in the spatial variable and convex in the last variable (see Section~\ref{s.preliminaries} for the precise setting). Homogenization of \eqref{intro:int} (for convex or nonconvex integrands) in terms of $\Gamma$-convergence is a classical problem in the calculus of variations, see for instance \cite{BD98,JKO} for textbook references. Assuming qualitative mixing in the form of ergodicity, Dal Maso and Modica proved in \cite{DMMoII} that in the scalar case $m=1$ the sequence of functionals \eqref{intro:int} $\Gamma$-converges almost surely towards a deterministic and autonomous integral functional
\begin{equation}\label{intro:inthom}
F_{\rm hom}(\cdot,D):\, W^{1,1}(D)\to[0,+\infty],\qquad F_{\rm hom}(u,D)=\int_D W_{\rm hom}(\nabla u(x))\dx,
\end{equation}
provided that $W$ satisfies standard $p$-growth, that is, there exist $1<p<\infty$ and $0<c_1,c_2<\infty$ such that $W$ is $p$-coerciv in the sense that
\begin{equation}\label{intro:pgrowth1}
 c_1|z|^p-c_2\leq W(\w,x,z)\quad\mbox{for all $z\in\R^{d}$ and a.e.\ $(\w,x)\in\Omega\times\R^d$},
\end{equation}
and satisfies $p$-growth in the form
\begin{equation}\label{intro:pgrowth2}
 W(\w,x,z)\leq c_2(|z|^p+1)\quad\mbox{for all $z\in\R^{d}$ and a.e.\ $(\w,x)\in\Omega\times\R^d$}.
\end{equation}
The result was extended to the vectorial (quasiconvex) case in \cite{MeMi}. By now there are many (classical and more recent) contributions on homogenization where the $p$-growth conditions \eqref{intro:pgrowth1} and \eqref{intro:pgrowth2} are relaxed in various ways: for instance nonstandard (e.g.\ $p(x)$, $(p,q)$ or  unbounded) growth conditions \cite{BD98,DG_unbounded,JKO,Mue87,zhikov} or degenerate $p$-growth (that is $c_1,c_2$ depend on $x$ and $\inf c_1=0$ and $\sup c_2=\infty$) \cite{D'OnZe,NSS17,RR21} (see also \cite{RZ} for the case $p=1$).

In this manuscript, we relax \eqref{intro:pgrowth1} and \eqref{intro:pgrowth2} in the way that, instead of \eqref{intro:pgrowth2}, we only assume that $W$ is locally bounded in the second variable, that is -- roughly speaking -- we assume 
\begin{equation}\label{intro:relaxedgrowth}
	\mathbb{E}[W(\cdot,x,\xi)]<+\infty.
\end{equation}
In the periodic setting, M\"uller proved in \cite{Mue87}, among other things, $\Gamma$-convergence of \eqref{intro:int} assuming the stronger boundedness condition
\begin{equation}\label{intro:pgrowth2b}
	{\rm ess sup}_{x\in\R^d}W(x,\xi)<+\infty\quad\mbox{for all $\xi\in\R^{m\times d}$}
\end{equation}
and $p$-coercivity with $p>d$, that is,
\begin{equation}\label{coerciv:pgeqd}
\exists\; p>d,\; c>0:\quad c|\xi|^p-\frac1c\leq W(x,\xi).
\end{equation}
This result was significantly extended \cite{AM11,DG_unbounded} to cover unbounded integrands and certain non-convex integrands with convex growth - still assuming \eqref{coerciv:pgeqd}. In the scalar case $m=1$, condition \eqref{coerciv:pgeqd} can be significantly relaxed to $p>1$, see \cite{DG_unbounded,JKO}. Note that condition \eqref{coerciv:pgeqd}  has to two effects: on the one hand it proves compactness of energy-bounded sequences (for this any $p>1$ suffices), but at the same time the Sobolev embedding turns energy-bounded sequences to compact sequences in $L^{\infty}$, which is crucial for adjusting boundary values of energy-bounded sequences in absence of the so-called fundamental estimate.

Here, we propose two ways for relaxing condition \eqref{coerciv:pgeqd} in the vectorial setting. In particular, in Theorems~\ref{thm.pureGamma}~and~\ref{thm:constrainedGamma}  below, we provide $\Gamma$-convergence results for \eqref{intro:int} with our without Dirichlet boundary conditions essentially (the precise statements can be found in Section~\ref{s.results}) assuming \eqref{intro:relaxedgrowth} and one of the following two conditions:
\begin{itemize}
\item[(a)] a 'mild monotonicity condition' which requires that for any matrix $\xi$ and any matrix $\widetilde{\xi}$ that is obtained from  $\xi$ by setting some (or equivalently one of the ) rows to zero we have
\begin{equation}\label{intro:A4}
	W(\w,x,\widetilde{\xi})\lesssim W(\w,x,\xi)
\end{equation}
(see Assumption~\ref{(A4)} below for the details), together with superlinear growth from below
$$
\lim_{|\xi|\to\infty}\frac{W(\w,x,\xi)}{|\xi|}=+\infty
$$
quantified by the stochastic integrability of the Fenchel conjugate of $\xi\mapsto W(\w,x,\xi)$ (see Assumption~\ref{(A3)}). In particular, we do not assume $p$-growth from below for some $p>1$ and thus improve even previous results in the scalar case (where the 'mild monotonicity condition' is always satisfied; see Remark~\ref{r.comments}~(iii)). Considering for instance the integrand $|\xi|^{p(\w,x)}$, it becomes clear that the integrability of the conjugate is a very weak condition as it allows for exponents $p$ arbitrarily close to $1$ (see Example \ref{ex:amples} for details)
\item[(b)] $p$-coercivity \eqref{intro:pgrowth1} for some $p$ with $p>d-1$. This improves the findings of \cite{Mue87} and (in parts) of \cite{DG_unbounded}, where corresponding statements are proved under the more restrictive assumption $p>d$. It also enlarges the range of admissible exponents considered in \cite[Chapter 15]{JKO}, where homogenization results were proven under $(p,q)$-growth conditions with $q<p^*$ (the critical Sobolev exponent associated to $p>1$)
\end{itemize}

Some comments are in order:
\begin{enumerate}[label=(\roman*)]
	\item To the best of our knowledge, the degenerate lower bound via the integrability of the Fenchel conjugate is the most general considered so far in the literature on homogenization of multiple integrals (at least when compactness in Sobolev spaces was established). For instance, letting $W(\w,x,\xi)=\lambda(\w,x)^p|\xi|^p$ for some weight $\lambda$, it reads that $\lambda(\cdot,x)^{-p/(p-1))}\in L^1(\Omega)$, which coincides with the assumption in \cite{RR21}, where it was shown that in general this integrability is necessary to have a non-degenerate value of the multi-cell formula with this class of integrands. While in the proof of the compactness statement in Lemma \ref{l.compactness} we use the convexity of $W$, this is not needed as one can replace $W$ by its biconjugate $W^{**}$ in the same proof when applying the Fenchel-Young inequality. Hence this assumption also yields compactness in the non-convex setting.
	\item While the condition \eqref{intro:A4} yields that componentwise truncation does not increase the energy too much, it does not imply that recovery sequences can assumed to be bounded (at least not using standard arguments). Moreover, truncation does only yield boundedness instead of compactness in $L^{\infty}$. It is for this reason that our approach is restricted to the convex case, where we can directly work with a corrector that has equi-integrable energy due to the strengthened version of the ergodic theorem proven in \cite{RR21}. 
	\item Condition \eqref{intro:A4} is tailor-made for componentwise truncation. Taking truncations of the full field $u$ of the form $\varphi(|u|_0)\tfrac{u}{|u|_0}$ (with a norm $|\cdot|_0$ and a suitable cut-off $\varphi$) leads to a more complicated condition on $W$ since the value of gradient of this map additionally depends on the value of $u$. In case of the Euclidean norm, a particular choice of $\varphi$ ensures that the norm of each partial derivative of the truncated map does not increase (cf. \cite[Section 4]{CDMSZ18}), so \eqref{intro:A4} for all matrices $\widetilde{\xi}$ with $|\widetilde{\xi}e_i|\leq |\xi e_i|$ for all $i\in\{1,\ldots,d\}$ would suffice to perform  non-componentwise truncations. They have the potential to reduce the $\Gamma$-convergence analysis to bounded sequences. We expect such a condition to pave the way to non-convex energies without condition \eqref{coerciv:pgeqd}. Note that in the scalar case $m=1$, we also obtain homogenization results in the non-convex setting by applying known relaxation results to the non-convex functional, see Remark~\ref{r.onconstrained}~(iii).
	\item The $p$-coercivity with exponent $p>d-1$ needs a fine choice of cut-off functions in combination with the compact embedding of $W^{1,p}(S_1)\subset L^\infty(S_1)$, where $S_1$ denotes the $d-1$-dimensional unit sphere. This idea has already been used for example in \cite{BC16,BC17} in context of div-curl lemmas and deterministic homogenization, and in \cite{BS21} in the context of regularity and stochastic homogenization of non-uniformly elliptic linear equations. 
\end{enumerate}
For the proof of our main $\Gamma$-convergence result we follow the strategy laid out in \cite{Mue87}: the lower bound, which does not require \eqref{intro:A4} or $p$-coercivity with $p>d-1$, is achieved via truncation of the energy. However, due to the degenerate lower bound, the truncation of the integrand is not straightforward, but needs to be done carefully (see Lemma \ref{l.truncateW}). Once the energies are suitably truncated, their $\Gamma$-convergence follows via standard arguments using blow-up and the multi-cell formula. In order to pass to the limit in the truncation parameter, we show that the multi-cell formula agrees with the single-cell formula given by a corrector on the probability space (see Lemma \ref{l.F_pot_formula}) as this formula passes easily to the limit (see Lemma \ref{l.klimit}). The new assumption \eqref{intro:A4} or the relaxed $p$-coercivity enter in the proof of the upper bound when we try to provide a recovery sequence for affine functions that agree with the affine function on the boundary. While under \eqref{intro:A4} we truncate peaks of the corrector and carefully analyze the error due to this truncation using the equi-integrability of the energy density of the corrector, the $p$-growth coercivity with $p>d-1$ makes use of the Sobolev embedding on spheres into $L^{\infty}$ that we quantify in a suitable way in Lemma \ref{L:optim} via the existence of ad hoc cut-off functions that allow us to control the energetic error when we impose boundary conditions on a non-constrained recovery sequence.

\smallskip

In addition to the $\Gamma$-convergence results of Theorems~\ref{thm.pureGamma}~and~\ref{thm:constrainedGamma}, we also consider homogenization of the Euler-Lagrange system corresponding to the functional \eqref{intro:int}. In order to formulate the latter problem in a convenient sense, i.e., without using merely the abstract notion of subdifferential of convex functionals, we need to investigate two issues: 
\begin{itemize}
	\item given that the integrand $W$ is differentiable with respect to the last variable, does a minimizer of the heterogeneous functional $u\mapsto F_{\e}(\w,u,D)+\{\text{boundary conditions}\}$ satisfy any PDE?
	\item is the homogenized integrand $W_{\rm hom}$ differentiable?
\end{itemize} 
For both points, the general growth conditions rule out deriving a PDE or the differentiability of $W_{\rm hom}$ by differentiating under the integral sign. Therefore, we rely on convex analysis. The subgradient for convex integral functionals is well-known on $L^p$-spaces. In order to capture the dependence on the gradient the standard way is to rely on the chain rule for subdifferentials. However, in general this only holds when the functional under examination has at least one continuity point. Hence, working on $L^p$-spaces is not feasible except for the choice $p=\infty$, which however does not necessarily coincide with the domain of the functional. As it turns out, the correct framework are generalized Orlicz spaces (cf. Section \ref{s.orlicz}). It should be noted that the integrand defining a generalized Orlicz space has to be even (otherwise the corresponding Luxemburg-norm fails to be a norm). Since on the one hand we need that the domain of our functional $F_{\e}(\w,\cdot,D)$ is contained in some generalized Sobolev-Orlicz space, while on the other hand the domain of the functional $F_{\e}(\w,\cdot,D)$ should have interior points on that space, the integrand $W(\w,x,\cdot)$ has to be comparable to an even function. For this reason, we are only able to prove the above two points under the additional assumption that
\begin{equation}\label{intro:almosteven}
	W(\w,x,\xi)\lesssim W(\w,x,-\xi),
\end{equation}
see Assumption \ref{a.2} for the detailed formulation. A possible approach to remove this assumption would be a theory of subdifferentials on so-called Orlicz-cones (see \cite[Section 2]{DeGr}, where such a theory is initiated in a very special setting). This is however beyond the scope of this work. Nevertheless, to the best of our knowledge this is the first time that the issue of global differentiability of $W_{\rm hom}$ and of the convergence of Euler-Lagrange equations (in a random setting) is settled without any polynomial growth from above, see the recent textbook \cite{CGSW21} for results on periodic homogenization in Orlicz spaces. The corresponding results are stated in Theorem \ref{thm:PDE}. 

\smallskip

The paper is structured as follows: in Section \ref{s.preliminaries} we first recall the basic notions from ergodic theory and state the properties of generalized Orlicz spaces that will be used in the paper. Then we formulate precisely our assumptions. In Section \ref{s.results} we present the main results of the paper, while we postpone the proofs to Section \ref{s.proofs}. In the appendix we prove a representation result for the convex envelope of radial functions, a very general measurability result for the optimal value of Dirichlet problems of integral functionals with jointly measurable integrand and extend an approximation-in-energy result of \cite{EkTe} that we need to treat the convergence of Dirichlet problems in the vectorial setting.

\section{Preliminaries and notation}\label{s.preliminaries}
\subsection{General notation}
We fix $d\geq 2$. Given a measurable set $S\subset\R^d$, we denote by $|S|$ its $d$-dimensional Lebesgue measure. For $x\in\R^d$ we denote by $|x|$ the Euclidean norm and $B_{\rho}(x)$ denotes the open ball with radius $\rho>0$ centered at $x$.
The real-valued $m\times d$-matrices are equipped with the operator-norm $|\cdot|$ induced from the Euclidean norm on $\R^d$, while we write $\langle\cdot,:\rangle$ for the Euclidean scalar product between two $m\times d$-matrices. Given a function $f:T\times\R^{m\times d}\to\R$, where $T$ is an arbitrary set, we denote by $f^*$ the Legendre-Fenchel conjugate of $f$ with respect to the last variable, that is,
\begin{equation*}
	f^*(t,\eta)=\sup\{\langle\eta,\xi\rangle-f(t,\xi):\,\xi\in\R^{m\times d}\}.
\end{equation*}
For a measurable set with positive measure, we define $\dashint_S=\frac{1}{|S|}\int_S$. We use standard notation for $L^p$-spaces and Sobolev spaces $W^{1,p}$. The Borel $\sigma$-algebra on $\R^d$ will be denoted by $\mathcal{B}^d$, while we use $\mathcal{L}^d$ for the $\sigma$-algebra of Lebesgue-measurable sets.
Throughout the paper, we use the continuum parameter $\e$, but statements like $\e\to 0$ stand for an arbitrary sequence $\e_n\to 0$. Finally, the letter $C$ stands for a generic positive constant that may change every time it appears.

\subsection{Stationarity and ergodicity}
Let $\Omega=(\Omega,\mathcal{F},\mathbb{P})$ be a complete probability space. Here below we recall some definitions from ergodic theory.
\begin{definition}[Measure-preserving group action]\label{def:group-action} A measure-preserving additive group action on $(\Omega,\F,\P)$ is a family $\{\tau_z\}_{z\in\R^d}$ of measurable mappings $\tau_z:\Omega\to\Omega$ satisfying the following properties:
\begin{enumerate}[label=(\arabic*)]
	\item\label{joint} (joint measurability) the map $(\w,z)\mapsto\tau_z(\w)$ is $\mathcal{F}\otimes\mathcal{L}^d-\mathcal{F}$-measurable;
	\item\label{inv} (invariance) $\P(\tau_z F)=\P(F)$, for every $F\in\F$ and every $z\in\R^d$;
	\item\label{group} (group property) $\tau_0=\rm id_\Omega$ and $\tau_{z_1+z_2}=\tau_{z_2}\circ\tau_{z_1}$ for every $z_1,z_2\in\R^d$.
\end{enumerate}
If, in addition, $\{\tau_z\}_{z\in\R^d}$ satisfies the implication
\begin{equation*}
	\mathbb{P}(\tau_zF\Delta F)=0\quad\forall\, z\in\R^d\implies \mathbb{P}(F)\in\{0,1\},
\end{equation*}
then it is called ergodic.
\end{definition}
\begin{remark}\label{r.stationary_extension}
	As noted in \cite[Lemma 7.1]{JKO}, the joint measurability of the group action implies that for every set $\Omega_0$ of full probability there exists a subset ${\Omega}_1\subset\Omega_0$ of full probability that is invariant under $\tau_z$ for a.e. $z\in\R^d$. In particular, if $\bar f:\Omega\to\R$ is a function that satisfies a given property almost surely, then the stationary extension $f:\Omega\times\R^d\to\R$ defined by $f(\w,x)=\bar f(\tau_x\w)$ satisfies the same property almost surely for a.e. $x\in\R^d$.
\end{remark}

We recall the following version of the ergodic theorem that is crucial for our setting (see \cite[Lemma 4.1]{RR21}).
\begin{lemma}\label{l.weakL1}
Let $g\in L^1(\Omega)$ and $\{\tau_z\}_{z\in\R^d}$ be a measure-preserving, ergodic group action. Then for a.e. $\w\in\Omega$ and every bounded, measurable set $B\subset\R^d$ the sequence of functions $x\mapsto g(\tau_{x/\e}\w)$ converges weakly in $L^1(B)$ as $\e\to 0$ to the constant function $\mathbb{E}[g]$.
\end{lemma} 
\subsection{Generalized Orlicz spaces}\label{s.orlicz}
We recall here the framework for generalized Orlicz spaces tailored to our setting. Let $(T,\mathcal{T},\mu)$ be a finite measure space. Given a jointly measurable function $\varphi:T\times\R^{m\times d}\to [0,+\infty)$ satisfying for a.e. $t\in T$ the properties
\begin{itemize}
 	\item [i)] $\varphi(t,0)=0$,
 	\item[(ii)] $\varphi(t,\cdot)$ is convex and even,
 	\item[(iii)] $\lim_{|\xi|\to +\infty}	\varphi(t,\xi)=+\infty$,
\end{itemize}
we define the generalized Orlicz space
$L^{\varphi}(T)^{m\times d}$ by 
\begin{equation*}
	L^{\varphi}(T)^{m\times d}=\left\{g:T\to\R^{m\times d}\text{ measurable }:\,\int_{T}\varphi(t,\beta g(t))\,\mathrm{d}\mu<+\infty\text{ for some }\beta>0\right\},
\end{equation*}
where we identify as usual functions that agree a.e. We equip this space with the Luxemburg norm
\begin{equation*}
	\|g\|_{\varphi}:=\inf\left\{\alpha>0:\,\int_T\varphi(t,\alpha^{-1} g(t))\,\mathrm{d}\mu\leq 1\right\},
\end{equation*}	
which then becomes a Banach space \cite[Theorem 2.4]{Koz}. We further assume the integrability conditions
\begin{equation}\label{eq:Orlicz_integrability}
	\sup_{|\zeta|\leq r}\varphi(\cdot,\zeta),\sup_{|\zeta|\leq r}\varphi^*(\cdot,\zeta)\in L^1(T)\quad\text{ for all }r>0,
\end{equation}
where we recall that $\varphi^*(t,\zeta)$ denotes the Legendre-Fenchel conjugate with respect to the last variable. Then $L^{\varphi}(T)^{m\times d}$ embeds continuously into $L^1(T)^{m\times d}$. Indeed, in this case the Fenchel-Young inequality and the definition of the Luxemburg-norm yield that
\begin{equation*}
	r\frac{\|g\|_{L^1(T)}}{\|g\|_{\varphi}}\leq\int_T\varphi\left(t,\frac{g(t)}{\|g\|_{\varphi}}\right)\,\mathrm{d}\mu+\int_T\sup_{|\zeta|\leq r}\varphi^*(t,\zeta)\,\mathrm{d}\mu\leq 1+\|\sup_{|\zeta|\leq r}\varphi^*(\cdot,\zeta)\|_{L^1(T)}<+\infty.
\end{equation*}
Denoting further by $(L^{\varphi}(T)^{m\times d})^*$ the dual space, \eqref{eq:Orlicz_integrability} allows us to apply \cite[Proposition 2.1 and Theorem 2.2]{Koz2} to characterize the dual space as follows: every $\ell\in (L^{\varphi}(T)^{m\times d})^*$ can be uniquely written  as a sum $\ell=\ell_a+\ell_s$, with $\ell_a\in L^{\varphi^*}(T)^{m\times d}$ (here the $*$ denotes the Legendre-Fenchel conjugate) and $\ell_s\in S^{\varphi}(T)$, where we set
\begin{align}\label{eq:dualspace}
	c_0(\mathcal{T})&:=\{(A_n)_{n \in\N}\subset\mathcal{T}:\,A_{n+1}\subset A_n\text{ for all }n\in\N,\,\mu\left(\bigcap_{n\in\N}A_n\right)=0\},\nonumber
	\\
	S^{\varphi}(T)&:=\{\ell\in (L^{\varphi}(T)^{m\times d})^*:\,\exists (A_n)_{n\in\N}\in c_0(\mathcal{\mathcal{T}}):\,\,\ell(h\chi_{T\setminus A_n})=0\quad\forall h\in L^{\varphi}(T)^{m\times d}\,\forall n\in\N\}.
\end{align}
Fur our analysis it will be crucial that \eqref{eq:Orlicz_integrability} further implies that for any element $\ell\in S^{\varphi}(T)$ it holds $\ell|_{L^{\infty}(T)^{m\times d}}=0$. To see this, note that \eqref{eq:Orlicz_integrability} implies that $\int_{T}\varphi(t,h(t))\,\mathrm{d}\mu<+\infty$ for all $h\in L^{\infty}(T)^{m\times d}$. Now let $(A_n)_{n\in\N}\in\mathcal{A}$ be a sequence as in the above definition for the element $\ell$. Since $T$ has finite measure, it follows that $\chi_{A_n}$ converges in measure to $0$. By linearity, for all $n\in\N$ we have that
\begin{equation*}
	\ell(h)=\ell(h\chi_{T\setminus A_n})+\ell(h\chi_{A_n})=\ell(h\chi_{A_n})
\end{equation*}
and so it suffices to show that $h\chi_{A_n}\to 0$ in $L^{\varphi}(T)^{m\times d}$. Given $\sigma>0$, the sequence $\varphi(\cdot,\sigma h\chi_{A_n})$ also converges in measure to $0$ and is uniformly bounded by the integrable function $\varphi(\cdot,\sigma h)$. Hence Vitali's convergence theorem yields
\begin{equation*}
	\lim_{n\to +\infty}\int_T\varphi(t,\sigma h(t)\chi_{A_n}(t))\,\mathrm{d}\mu=0,
\end{equation*}
so that $\limsup_{n\to +\infty}\|h\chi_{A_n}\|_{\varphi}\leq\sigma^{-1}$. Since $\sigma$ can be made arbitrarily large, we obtain the claimed convergence to $0$. 

Finally, we shall make use of the following representation formula for the subdifferential of convex integral functionals: let $f:T\times\R^{m\times d}\to \R$ be a jointly measurable function that is convex in its second variable. Assume that $g\in L^{\varphi}(T)^{m\times d}$ such that that $I_f(g)\in\R$, where
$I_f(g)=\int_T f(t,g(t))\,\mathrm{d}\mu$. Then the subdifferential of $I_f$ at $g$ is given by
\begin{equation*}
	\partial I_f(g)=\left\{\ell_a\in L^{\varphi^*}(T)^{m\times d}:\,\ell_a\in\partial_{\xi}f(\cdot,g(\cdot)) \text{ a.e.}\right\}+\left\{\ell_s\in S^{\varphi}(T):\,\ell_s(h-g)\leq 0\text{ for all }h\in \text{dom}(I_f)\right\};
\end{equation*}
cf. \cite[Theorem 3.1]{Koz2} which can be applied to due \eqref{eq:Orlicz_integrability}.

\subsection{Framework and assumptions}
Let $D\subset \R^d$ be an open, bounded set  with Lipschitz boundary and let $(\Omega,\mathcal{F},\mathbb{P})$ be a complete probability space equipped with a measure-preserving, ergodic group action $\{\tau_z\}_{z\in\R^d}$. For $\e>0$, we consider integral functionals defined on $L^1(D)^m$ with domain contained in $W^{1,1}(D)^m$, taking the form
\begin{equation*}
	F_{\e}(\w,u,D)=\int_D W(\w,\tfrac{x}{\e},\nabla u(x))\,\mathrm{d}x\in [0,+\infty]
\end{equation*}
with the integrand $W$ satisfying the following assumptions:
\begin{assumption}\label{a.1} There exists a $\mathcal{F}\otimes\mathcal{B}^{m\times d}$-measurable function $\overbar{W}:\Omega\times\R^{m\times d}\to [0,+\infty)$ such that 
\begin{enumerate}[label=(A\arabic*)]
	\item\label{(A1)} $W(\w,x,\xi):=\overbar{W}(\tau_x\w,\xi)$ (stationarity and joint measurability) ;
	
	\vspace*{2.5mm}
	
	\item\label{(A2)} for a.e. $\w\in\Omega$ the map $\xi\mapsto \overbar{W}(\w,\xi)$ is convex (thus also $\xi\mapsto W(\w,x,\xi)$ for a.e. $x\in\R^d$);
	
	\vspace*{2.5mm}
	
	\item\label{(A3)} for all $r>0$ 
	\begin{align*}
		&\w\mapsto\sup_{|\eta|\leq r}\overbar{W}(\w,\eta)\in L^1(\Omega);\quad(\text{local boundedness})
		\\
		&\w\mapsto\sup_{|\eta|\leq r}\overbar{W}^*(\w,\eta)\in L^1(\Omega), \quad(\text{inhomogeneous superlinearity})
	\end{align*}
	where $\overbar{W}^*$ denotes the Legendre-Fenchel transform of $\overbar{W}$ with respect to its last variable.
	\end{enumerate}
	\noindent Moreover, $\overbar{W}$ satisfies at least one of the following two conditions \ref{(A4)} or \ref{(A5)} below.
	\begin{enumerate}[label=(A\arabic*),resume]
	\item\label{(A4)} (mild monotoncity) there exists $C>1$ and a non-negative function $\overbar{\Lambda}\in L^1(\Omega)$ such that for a.e. $\w\in\Omega$ and all $\xi\in\R^{m\times d}$, all $\widetilde{\xi}\in\R^{m\times d}$ with $e_j^T(\xi-\widetilde{\xi})\in \{0,e_j^T\xi\}$ for all $1\leq j\leq m$ it holds that
	\begin{equation*}
	\overbar{W}(\w,\widetilde{\xi})\leq C\overbar{W}(\w,\xi)+\overbar{\Lambda}(\w)
	\end{equation*}
	and thus also
	\begin{equation*}
	W(\w,x,\widetilde{\xi})\leq CW(\w,x,\xi+\Lambda(\w,x)\text{ for a.e. }x\in\R^d
	\end{equation*}
	with $\Lambda(\w,x):=\overbar{\Lambda}(\tau_x\w)$.
	\item\label{(A5)} ($p>d-1$ coercivity) there exists $p>d-1$ such that
	$$
	\overbar{W}(\omega,\xi)\geq |\xi|^p\text{ and thus }W(\w,x,\xi)\geq |\xi|^p\text{ for a.e. }x\in\R^d.
	$$
\end{enumerate}
\end{assumption}

\begin{remark}\label{r.comments}
\begin{itemize}
	\item[(i)] Due to Remark \ref{r.stationary_extension}, the above Assumptions are indeed just assumptions on $\overbar{W}$. Note that the local suprema in \ref{(A3)} can be replaced by pointwise integrability of $\overbar{W}(\cdot,\xi)$ and 
		$\overbar{W}^*(\cdot,\xi)$ since convex functions on cubes attain their maximum at the finitely many corners and moreover $\overbar{W}\geq 0$, while $\overbar{W}^*$ can be bounded from below by $\overbar{W}^*(\w,\xi)\geq -\overbar{W}(\w,0)$. Another advantage of the definition of $W$ via stationary extension is that the function $W^*(\w,x,\xi)=\overbar{W}^*(\tau_x\w,\xi)$ remains $\mathcal{F}\otimes\L^d\otimes\mathcal{B}^{m\times d}$-measurable due to the completeness of the probability space. Indeed, following verbatim the proof of \cite[Proposition 6.43]{FoLe}, one can show that for any jointly measurable function $h:\Omega\times\R^{m\times d}\to\R$ the Fenchel-conjugate with respect to the second variable is still jointly measurable. However, note that completeness of $\Omega$ is essential for the proof when one only assumes joint measurability.
	\item[(ii)] The integrability condition on the conjugate $\overbar W^*$ in \ref{(A3)} implies that $\xi\mapsto \overbar W(\w,\xi)$ is superlinear at infinity. Indeed, by definition we know that for all $\xi,\eta\in\R^{m\times d}$ we have
	\begin{equation*}
	\overbar W(\w,\xi)+\overbar W^*(\w,\eta)\geq \langle\eta,\xi\rangle.
	\end{equation*}
	For $\xi\neq 0$ and $C>0$, choosing $\eta=C\xi/|\xi|$ we deduce that
	\begin{equation}\label{eq:quantitativelineargrowth}
	\overbar W(\w,\xi)\geq C|\xi|-\sup_{|\eta|\leq C}\overbar W^*(\w,\eta),
	\end{equation}
	so that by the arbitrariness of $C$ we obtain for a.e. $\w\in\Omega$
	\begin{equation*}
	\lim_{|\xi|\to +\infty}\frac{\overbar W(\w,\xi)}{|\xi|}=+\infty.
	\end{equation*} 
	 It will be useful to have a suitable radial lower bound for $\overbar W$. Define $\widehat{\ell}$ to be the the jointly measurable function $(\w,r)\mapsto \inf_{|\eta|=r}\overbar W(\w,\eta)$\footnote{It is a Carathéodory-function. Indeed, for a.e. $\w\in\Omega$ the convexity and finiteness of $\overbar{W}$ imply continuity with respect to $\xi$, which can be used to deduce continuity with respect to $r$, while measurability with respect to $\w$ can be shown as follows: for every $t>0$ the set $\{(\w,\xi)\in \Omega\times \{|\xi|=r\}:\,\overbar{W}(\w,\xi)<t\}$ is $\mathcal{F}\otimes \mathcal{B}^{m\times d}$-measurable, so that by the measurable projection theorem the projection onto $\Omega$ is measurable. This projection is exactly $\{w\in\Omega:\,\inf_{|\xi|=r}\overbar{W}(\w,\xi)<t\}$.} and consider the convex envelope (in $\R^{m\times d}$) of the map $\xi\mapsto \widehat{\ell}(\w,|\xi|)$. We then know from the above superlinearity (with $x=0$) and Lemma \ref{l.envelope} that   $\ell(\w,|\xi|)\leq \overbar W(\w,\xi)$ for some convex, monotone, superlinear function $\ell$. Moreover, due to \eqref{eq:quantitativelineargrowth} with $C=1$, without loss of generality the function $\ell$ satisfies the lower bound
	\begin{equation}\label{eq:ell_quant_linear}
	\ell(\w,|\xi|)\geq |\xi|-\overbar \Lambda(\w).
	\end{equation}
	\item[(iii)] The monotonicity assumption \ref{(A4)} is no restriction in the scalar case $m=1$ since in this case it reduces to $\overbar W(\w,0)\leq C\overbar W(\w,\xi)+\Lambda(\w)$, which follows from \ref{(A3)}. Moreover, \ref{(A4)} is also verified in case 
	\begin{equation}\label{eq:radialgrowth}
		\ell(\w,|\xi|)\leq \overbar W(\w,\xi)\leq C\ell(\w,|\xi|)+\overbar\Lambda(\w)
	\end{equation} 
	for some superlinear function $\ell(\w,\cdot):[0,+\infty)\to[0,+\infty)$. To see this,  observe that the convex envelope of $\xi\mapsto\ell(\w,|\xi|)$ is of the form $\xi\mapsto \ell_0(\w,|\xi|)$ with a monotone, convex function $\ell_0(\w,\cdot)$ (see Lemma \ref{l.envelope}). Then $\ell_0(\w,|\xi|)\leq \ell(\w,|\xi|)$ and due to the convexity of $\xi\mapsto \overbar W(\w,\xi)$ we have $\overbar W(\w,\xi) \leq C\ell_0(\w,|\xi|)+\overbar\Lambda(\w)$. Thus \ref{(A4)} is a consequence of the following estimate:
	\begin{equation*}
	\overbar{W}(\w,\widetilde{\xi})\leq C\ell_0(\w,|\widetilde{\xi}|)+\overbar\Lambda(\w)\leq C\ell_0(\w,|\xi|)+\overbar\Lambda(\w)\leq C\overbar W(\w,\xi)+\overbar\Lambda(\w),
	\end{equation*}
	where we used the monotonicity of $\ell_0$. Finally, \ref{(A4)} comes also for free if $\overbar{W}$ is even with respect to each row of $\xi$. Indeed, in this case it follows that $\overbar{W}(\w,\xi^{-})=\overbar W(\w,\xi)$, where $\xi^-$ is any matrix that is made of $\xi$ by multiplying some rows by $(-1)$. Since the matrix $\widetilde{\xi}$ in \ref{(A4)} is the midpoint of the line connecting $\xi$ and some $\xi^{-}$ as above, by convexity $\overbar W(\w,\widetilde{\xi})\leq \tfrac{1}{2}\overbar W(\w,\xi)+\tfrac{1}{2}\overbar W(\w,\xi^-)=\overbar W(\w,\xi)$.
\end{itemize}
\end{remark}
\begin{example}\label{ex:amples}
Here we give some well-known examples that satisfy Assumptions \ref{a.1}. In what follows we exploit that for fixed $p\in (1,+\infty)$ the Fenchel conjugate of the function $\tfrac{1}{p}|\xi|^p$ is given by the function $\tfrac{1}{q}|\xi|^q$, where $q=p/(p-1)$ denotes the conjugate exponent to $p$. We further assume that the probability space $\Omega$ and the stationary, ergodic group action $\{\tau_z\}_{z\in\R^d}$ are given.
\begin{itemize}
	\item[1)] Consider the integrand $\overbar W(\w,\xi)=\tfrac{1}{p(\w)}|\xi|^{p(\w)}$ with a random exponent $p:\Omega\to (1,+\infty)$. Denoting by $q(\cdot)=\tfrac{p(\cdot)}{p(\cdot)-1}$ its conjugate exponent, $\overbar W$ satisfies Assumption \ref{a.1} whenever for all $r>0$
	\begin{equation*}
		\frac{1}{p(\cdot)}r^{p(\cdot)}\in L^1(\Omega),\quad \quad \frac{1}{q(\cdot)}r^{q(\cdot)}\in L^1(\Omega),
	\end{equation*}
	which is equivalent (recall that $\exp(p)\geq p$ for all $p\geq 1$) to the moment generating functions of $p(\cdot)$ and $q(\cdot)$ being globally finite. In particular, one can construct examples (e.g. with subgaussian tails and a corresponding decay close to $1$) with ${\rm ess\,inf\,}_{\w}p(\w)=1$ and ${\rm ess\,sup\,}_{\w}p(\w)=+\infty$ that fall in the framework of our assumptions.
	\item [2)] Consider the double-phase integrand $\overbar W(\w,\xi)=|\xi|^p+a(\w)|\xi|^q$ with $1< p<+\infty$, $1\leq q<+\infty$ and $a\in L^1(\Omega)$ a non-negative function. Then Assumption \ref{a.1} is satisfied. If $p=1$ and $q>1$, then Assumption \ref{a.1} holds if in addition $a^{1-q}\in L^1(\Omega)$. Moreover, Assumption \ref{a.1} does not restrict to polynomial growth and is also satisfied for generalized double-phase integrands of the form $\overbar W(\w,\xi)=|\xi|^p+a(\w)\exp(\exp(|\xi|^q))$ with $1< p<+\infty$, $0< q<+\infty$ and $a\in L^1(\Omega)$ a non-negative function (obviously the double exponential can be replaced by any convex and continuous function). 
\end{itemize}
\end{example}

We also study the convergence of the associated Euler-Lagrange equations for $F_{\e}(\w,\cdot,D)$ under Dirichlet boundary conditions and external forces. To show that the homogenized operator is differentiable under the general growth conditions, we need to impose an additional structural assumption. 
\begin{assumption}\label{a.2}
In addition to Assumption \ref{a.1}, assume that for a.e. $\w\in\Omega$ the map $\xi\mapsto \overbar W(\w,\xi)$ is differentiable and almost even in the sense that there exists $C\geq 1$ such that for all $\xi\in\R^{m\times d}$ we have
\begin{equation}\label{eq:almosteven}
\overbar W(\w,-\xi)\geq \frac{1}{C}\overbar W(\w,\xi)-\overbar\Lambda(\w).
\end{equation}	
Note that the corresponding integrand $W(\w,x,\xi)$ then satisfies the same properties for a.e. $x\in\R^d$ with $\overbar{\Lambda}$ replaced by $\Lambda$.
\end{assumption}
\begin{remark}\label{r.Orliczspaces}
We need \eqref{eq:almosteven} to construct a generalized (Sobolev-)Orlicz space associated to the domain of $F_{\e}(\w,\cdot,D)$ or of $h\mapsto \mathbb{E}[\overbar W(\cdot,h)]$ as follows: the function $m(\w,\xi):= \min\{\overbar W(\w,\xi),\overbar W(\w,-\xi)\}$ is even and jointly measurable. So is its convex envelope ${\rm co}(m(\w,\xi))$ (cf. Remark \ref{r.comments} (i)) and finally also the non-negative function
\begin{equation}\label{eq:defphi}
	\overbar\varphi(\w,\xi)=C\max\{0,{\rm co}(m(\w,\xi))-\overbar W(\w,0)\},
\end{equation}
where $C$ is the constant given in \eqref{eq:almosteven}. Note that
\begin{equation*}
	\frac{1}{C}\overbar\varphi(\w,\xi)\leq {\rm co}(m(\w,\xi))+\overbar W(\w,0)\leq m(\w,\xi)+\overbar W(\w,0)\leq \overbar W(\w,\xi)+\overbar W(\w,0),
\end{equation*}
while \eqref{eq:almosteven} implies the lower bound
\begin{equation*}
	\frac{1}{C}\overbar\varphi(\w,\xi)\geq {\rm co}(m(\w,\xi))-\overbar W(\w,0)\geq \frac{1}{C}\overbar W(\w,\xi)-\overbar W(\w,0)-\overbar\Lambda(\w),
\end{equation*}
which up to increasing $\Lambda$ yields the two-sided estimate
\begin{equation}\label{eq:varphi_and_W}
	\overbar W(\w,\xi)-\overbar\Lambda(\w)\leq\overbar\varphi(\w,\xi)\leq C\overbar W(\w,\xi)+\overbar\Lambda(\w).
\end{equation}
Denoting by $\overbar\varphi^*$ the conjugate function of $\overbar\varphi$ (with respect to the last variable), we deduce that for all $r>0$
\begin{equation*}
	0\leq \overbar\varphi^*(\w,0)\leq \sup_{|\eta|\leq r}\overbar\varphi^*(\w,\eta)\leq \overbar\Lambda(\w)+\sup_{|\eta|\leq r}\overbar{W}^*(\w,\eta).
\end{equation*}
In particular, combined with \eqref{eq:varphi_and_W} and \ref{(A3)}, we find that $\bar\varphi$ satisfies all assumptions stated in Section \ref{s.orlicz} for the choice $(T,\mathcal{T},\mu)=(\Omega,\mathcal{F},\mathbb{P})$ (the superlinearity at $+\infty$ follows as in Remark \ref{r.comments} (ii)), so that we can define the generalized Orlciz space $L^{\overbar{\varphi}}(\Omega)^{m\times d}$, which enjoys all properties stated in Section \ref{s.orlicz}. On the physical space $(D,\mathcal{L}^d)$ with the Lebesgue-measure instead, due to Remark  \ref{r.stationary_extension} we can consider for a.e. $\w\in\Omega$ the (random) function $\varphi_{\e}(x,\xi)=\overbar{\varphi}(\tau_{x/\e}\w,\xi)$ to define the generalized Orlicz space $L^{\varphi_{\e}}_{\w}(D)^{m\times d}$, which satisfies again all properties stated in Section \ref{s.orlicz}.  

We can further define the corresponding generalized Sobolev-Orlicz space
\begin{equation*}
	W^{1,\varphi_{\e}}_{\w}(D)^m:=\{u\in W^{1,1}(D)^m:\,\nabla u\in L^{\varphi_{\e}}_{\w}(D)^{m\times d}\},
\end{equation*}
that becomes a Banach space for the norm
\begin{equation*}
	\|u\|_{W^{1,\varphi_{\e}}_{\w}}=\|u\|_{L^1(D)}+\|\nabla u\|_{\varphi_{\e},\w}
\end{equation*}
and which embeds continuously into $W^{1,1}(D)^m$. We also define the space $W^{1,\varphi_{\e}}_{0,\w}(D)^m$ as the subspace with vanishing $W^{1,1}(D)^m$-trace. Due to the continuous embedding this subspace is closed. As we shall prove, the homogenized integrand $W_{\rm hom}$ appearing in Theorem \ref{thm.pureGamma} satisfies the deterministic analogue of \eqref{eq:almosteven}, which allows us to define the associated Sobolev-Orlicz spaces $W^{1,\varphi_{\rm hom}}(D)^m$ and $W_0^{1,\varphi_{\rm hom}}(D)^m$ for the homogenized model. We need those spaces to properly formulate our results concerning the Euler-Lagrange equations. 
\end{remark}
\section{Main results}\label{s.results}

Note that due to the probabilistic nature our main results are only true for a.e. $\w\in\Omega$. At the beginning of Section \ref{s.proofs} we describe precisely which null sets we have to exclude.

We start our presentation of the main results by stating the $\Gamma$-convergence result of the unconstrained functionals.
\begin{theorem}\label{thm.pureGamma}
Let $W$ satisfy Assumption~\ref{a.1}. Then almost surely as $\e\to 0$, the functionals $F_{\e}(\w,\cdot,D)$ $\Gamma$-converge in $L^1(D)^m$ to the functional $F_{\rm hom}:L^1(D)^m\to [0,+\infty]$ defined on $W^{1,1}(D)^m$ by
\begin{equation*}
	F_{\rm hom}(u)=\int_D W_{\rm hom}(\nabla u(x))\dx\in [0,+\infty],
\end{equation*}	
where the integrand $W_{\rm hom}:\R^{m\times d}\to [0,+\infty)$ is convex. Moreover, the following is true: 
\begin{itemize}
\item Suppose $W$ satisfies \ref{(A4)}. Then $W_{\rm hom}$ is superlinear at infinity and there exists $C_0<+\infty$ such that for all $\xi\in\R^{m\times d}$ and all $\widetilde{\xi}\in\R^{m\times d}$ with $e_j^T(\xi-\widetilde{\xi})\in \{0,e_j^T\xi\}$ for all $1\leq j\leq m$ it holds that
\begin{equation*}
	W_{\rm hom}(\widetilde{\xi})\leq C_0\,(W_{\rm hom}(\xi)+1).
\end{equation*}
\item Suppose $W$ satisfies \ref{(A5)}. Then for all $\xi\in\R^{m\times d}$
\begin{equation*}
	W_{\rm hom}(\xi)\geq |\xi|^p,
\end{equation*}
where $p>d-1$ is the exponent in \ref{(A5)}.

\end{itemize}
\end{theorem}
\begin{remark}\label{rem:whom}
For an intrinsic formula defining $W_{\rm hom}$ see Lemma \ref{l.defW_hom}. It follows a posteriori from Theorem \ref{thm:constrainedGamma} that one can obtain $W_{\rm hom}(\xi)$ by the standard multi-cell formula
\begin{equation*}
	W_{\rm hom}(\xi):=\lim_{t\to +\infty}\inf\left\{\dashint_{(-t,t)^d}W(\w,x,\xi+\nabla u)\dx:\,u\in W^{1,1}_0((-t,t)^d,\R^m)\right\}.
\end{equation*} 
Indeed, by the change of variables $x\mapsto x/\e$ and Theorem \ref{thm:constrainedGamma} the above limit equals
\begin{equation*}
	\min_{u\in W^{1,1}_0((-1,1)^d,\R^m)}\dashint_{(-1,1)^d}W_{\rm hom}(\xi+\nabla u)\dx=W_{\rm hom}(\xi),
\end{equation*}
where the last equality follows from the convexity of $W_{\rm hom}$.

\end{remark}
%
%
%

Next we discuss the convergence of boundary value problems together with a varying forcing term added to the functionals. Given $g\in W^{1,\infty}(\R^d,\R^m)$ and $f_{\e}\in L^d(D)^m$, we define the constrained functional
\begin{equation}\label{eq:constrained}
	F_{\e,f_{\e},g}(\w,u,D)=
	\begin{cases}
		F_{\e}(\w,u,D)-\int_D f_{\e}(x)\cdot u(x)\dx &\mbox{if $u\in g+W^{1,1}_0(D)^m$,}
		\\
		\\
		+\infty &\mbox{otherwise on $L^1(D)^m$.}
	\end{cases}
\end{equation}
Due to the Sobolev embedding the integral involving $f_{\e}$ is finite for $u\in W^{1,1}(D)^m$.
\begin{theorem}\label{thm:constrainedGamma}
Let $W$ satisfy Assumption \ref{a.1}. Assume that $g\in W^{1,\infty}(\R^d)^m$ and that $f_{\e}\in L^d(D)^m$ is such that $f_{\e}\rightharpoonup f$ in $L^d(D)^m$ as $\e\to 0$. Then almost surely, as $\e\to 0$, the functionals $u\mapsto F_{\e,f_{\e},g}(\w,u,D)$ $\Gamma$-converge in $L^1(D)^m$ to the deterministic integral functional $F_{{\rm hom},f,g}:L^1(D)^m\to [0,+\infty]$ defined on $g+W^{1,1}_0(D)^m$ by
\begin{equation*}
	F_{{\rm hom},f,g}(u)=\int_D W_{\rm hom}(\nabla u(x))\dx-\int_Df(x)\cdot u(x)\dx\in \R\cup\{+\infty\}
\end{equation*}
and $+\infty$ otherwise. The integrand $W_{\rm hom}$ is given by Theorem \ref{thm.pureGamma}. Moreover, any sequence $u_{\e}$  such that 
\begin{equation}\label{thm:constrainedGamma:coer}
	\limsup_{\e\to 0}F_{\e,f_{\e},g}(\w,u_{\e},D)<+\infty
\end{equation}
is weakly relatively compact in $W^{1,1}(D)^m$ and strongly relatively compact in $L^{d/(d-1)}(D)^m$.

If \ref{(A5)} is satisfied, the above result is also valid when $f_\e \rightharpoonup f$ in $L^q(D)^m$ for some $q\geq1$ with $\frac1q<1-\frac1p+\frac1d$, and sequences $u_\e$ satisfying \eqref{thm:constrainedGamma:coer} are weakly relatively compact in $W^{1,p}(D)^m$, where $p>d-1$ is the exponent in \ref{(A5)}. 

\end{theorem}
\begin{remark}\label{r.onconstrained}
\begin{itemize}
	\item[(i)] The condition $g\in W^{1,\infty}(\R^d,\R^m)$ can be weakened to Lipschitz-continuity on $\partial\Omega$. Then one can redefine $g$ on $\R^d\setminus\partial\Omega$ using Kirszbraun's extension theorem and the definition of the functional $F_{\e,f_{\e},g}$ is not affected.

\item[(ii)] By the fundamental property of $\Gamma$-convergence, Theorem \ref{thm:constrainedGamma} and the boundedness of $F_{\e,f_{\e},g}(\w,g,D)$  as $\e\to 0$, imply that up to subsequences the minimizers of $F_{\e,f_{\e},g}(\w,\cdot)$ converge to minimizers of $F_{{\rm hom},f,g}$. In particular, when $W$ is strictly convex in the last variable, then one can argue verbatim as in \cite[Propisition 4.14]{RR21} to conclude that also $W_{\rm hom}$ is strictly convex. In this case, the minimizers $u_{\e}$ at the $\e$-level and $u_0$ of the limit functional are unique and $u_{\e}\to u_0$ as $\e\to 0$ weakly in $W^{1,1}(D)^m$ and strongly in $L^{d/(d-1)}(D)^m$.
\item[(iii)] In the scalar case $m=1$ non-convexity does not play a major role whenever the assumptions ensure a relaxation result via convexification of the integrand. In particular, this holds when the non-convex integrand $h$ is stationary and ergodic (in the sense of Assumption \ref{(A1)}) and the corresponding integrand $\overbar{h}:\Omega\times\R^{d}\to [0,+\infty)$ is jointly measurable, upper semicontinuous in the second variable for a.e. $\w\in\Omega$ and satisfies the following strengthened growth assumptions:
	\begin{itemize}
		\item[(1)] $\overbar h^*(\cdot,z)\in L^1(\Omega)$ for all $z\in\R^{d}$;
		\item[(2)] $\overbar h(\w,z)\leq \Phi(z)+\overbar\Lambda(\w)$ for a.e. $\w\in\Omega$ and all $z\in\R^d$, where $\Phi$ is a finite, convex function.
	\end{itemize}
	Under these assumptions, one can  apply \cite[Corollary 3.12]{MaSb} to deduce that the weak $W^{1,1}(D)$-relaxation of the non-convex functional with integrand $ h(\w,\tfrac{x}{\e},z)$ is given by 
	\begin{equation*}
		\int_D h^{**}(\w,\tfrac{x}{\e},\nabla u)\dx\quad\text{ for all }u\in W^{1,\infty}(D).
	\end{equation*}
	Due to the growth condition (2), one can extend this representation by well-known approximation results (cf. \cite[Lemma 3.6]{Mue87} and \cite[Chapter X, Proposition 2.10]{EkTe}) to all functions $u\in W^{1,1}(D)$ (the restriction to $D$ being smooth and strongly-starshaped in \cite{Mue87} is not necessary since the left-hand side functional is lower semicontinuous with respect to weak convergence, so recovery sequences on balls are recovery sequences on all open subsets $A$ with $|\partial A|=0$). Using (1), one can show that the weak $W^{1,1}(D)$-relaxation agrees with the strong $L^1$-relaxation (cf. the proof of Lemma \ref{l.compactness}, where one has to use the Fenchel-Young inequality for the two functions $h^*$ and $h^{**}$ instead), so that by general $\Gamma$-convergence theory the $\Gamma$-limit of the non-convex energy agrees with the $\Gamma$-limit of the convexified functionals. Due (1) and (2), the function $h^{**}$ satisfies Assumption \ref{a.1}, so the results from the convex case transfer to the non-convex one. Concerning our $\Gamma$-convergence result with boundary conditions, the only difference when relaxing the heterogeneous functional comes from the extension of the relaxation from Lipschitz-functions to general functions because on $W^{1,1}_0(D)$ we have the integrand $\overbar{h}^{**}(\w,\tfrac{x}{\e},\nabla g(x)+\xi)$, where $g\in W^{1,\infty}$ is the boundary datum. By convexity this integrand can be bounded via
	\begin{equation*}
		\Phi(\nabla g(x)+\xi)+\overbar\Lambda(\tau_{x/\e}\w)\leq \sup_{|\eta|\leq 2\|\nabla g\|_{L^{\infty}}}\Phi(\eta)+\Phi(2\xi)+\overbar{\Lambda}(\tau_{x/\e}\w).
	\end{equation*}
	Then we can due the approximation using \cite[Chapter X, Propositions 2.6 \& 2.10]{EkTe} and therefore can again assume that the integrand is convex. It should be noted that we need (2) only for the convexified integrand $\overbar{h}^{**}$, while for the relaxation result in \cite{MaSb} a bound of the form $\overbar{h}(\w,z)\leq \overbar W(\w,z)$ with $\overbar{W}$ is a real-valued function that is convex in $z$ for a.e. $\w\in\Omega$ and satisfies the integrability condition $\mathbb{E}[\overbar W(\cdot,z)]<+\infty$ for all $z\in\R^d$ (cf. \ref{(A3)} and Remark \ref{r.comments} (i)) suffices. However, we were not able to find a relaxation result only assuming the last integrability condition (translated to the physical space) instead of (2) and which holds on the whole domain of the integral functional.
\end{itemize}
\end{remark}
Our final result concerns the Euler-Lagrange equations of the functionals $F_{\e,f_{\e},g}$ and $F_{{\rm hom},f,g}$. In particular, we address the differentiability of the function $W_{\rm hom}$. Here we have to rely on the stronger Assumption \ref{a.2} to be able to work in (Sobolev-)Orlicz spaces.
\begin{theorem}\label{thm:PDE}
Let $W$ satisfy Assumption \ref{a.2} and let $g,f_{\e}$ and $f$ be as in Theorem \ref{thm:constrainedGamma}. Then the following statements hold true.
\begin{itemize}
	\item[i)] Almost surely there exists a function $u_{\e}\in g+W_0^{1,1}(D)^m$ such that
	\begin{equation*}
		\int_D \partial_{\xi}W(\w,\tfrac{x}{\e},\nabla u_{\e}(x))\nabla\phi(x)- f_{\e}(x)\cdot\phi(x)\dx\begin{cases}
			=0 &\mbox{if $\phi\in W_0^{1,\infty}(D)^m$,}
			\\
			\geq 0 &\mbox{if $\phi\in W_0^{1,1}(D)^m$ and $F_{\e}(\w,u+\phi,D)<+\infty$.}
		\end{cases}
	\end{equation*}
	The above (in)equality is equivalent to $u_{\e}$ minimizing $F_{\e,f_{\e},g}(\w,\cdot,D)$.
	\item[ii)] The function $W_{\rm hom}$ is continuously differentiable.
	\item[iii)] There exists a function $u_0\in g+W_0^{1,1}(D)^m$ such that
	\begin{equation*}
		\int_D \partial_{\xi}W_{\rm hom}(\nabla u_{0}(x))\nabla\phi(x)- f(x)\cdot\phi(x)\dx\begin{cases}
			=0 &\mbox{if $\phi\in W_0^{1,\infty}(D)^m$,}
			\\
			\geq 0 &\mbox{if $\phi\in W_0^{1,1}(D)^m$ and $F_{\rm hom}(u+\phi)<+\infty$.}
		\end{cases}
	\end{equation*}
	The above (in)equality is equivalent to $u_{0}$ minimizing $F_{{\rm hom},f,g}$.
	\item[iv)] If there exists $s>1$ such that $F_{\e}(\w,su_{\e},D)<+\infty$ respectively $F_{\rm hom}(su_0)<+\infty$, then 
	\begin{equation*}
		\int_D \partial_{\xi}W(\w,\tfrac{x}{\e},\nabla u_{\e}(x))\nabla\phi(x)-f_{\e}(x)\cdot\phi(x)\dx
			=0\quad \mbox{for all $\phi\in W_{0,\w}^{1,\varphi_{\e}}(D)^m$,}
	\end{equation*}
	respectively
	\begin{equation*}
		\int_D \partial_{\xi}W_{\rm hom}(\nabla u_{0}(x))\nabla\phi(x)-f(x)\cdot\phi(x)\dx
		=0\quad \mbox{for all $\phi\in W_{0}^{1,\varphi_{\rm hom}}(D)^m$,}
	\end{equation*}
	where the spaces $W_{0,\w}^{1,\varphi_{\e}}(D)^m$ and $W_{0}^{1,\varphi_{\rm hom}}(D)^m$ are the Sobolev-Orlicz spaces associated to $W$ and $W_{\rm hom}$ (cf. Remark \ref{r.Orliczspaces}).
	\item[v)] If $W$ is strictly convex in its last variable, then the solutions $u_{\e}$ and $u_0$ are unique and almost surely, as $\e\to 0$, the random solutions $u_{\e}=u_{\e}(\w)$ converge to $u_0$ weakly in $W^{1,1}(D)^m$ and strongly in $L^{d/(d-1)}(D)^m$.
\end{itemize}
\end{theorem}
\begin{remark}
	\begin{itemize}
\item[a)] We have the inclusions $W_0^{1,\infty}(D)^m\subset W^{1,\varphi_{\e}}_{0,\w}\cap W^{1,\varphi_{\rm hom}}_0(D)^m$ and  $\{F_{\e}(\w,\cdot,D)<+\infty\}\subset W_{\w}^{1,\varphi_{\e}}(D)^m$ and $\{F_{\rm hom}<+\infty\}\subset W^{1,\varphi_{\rm hom}}(D)^m$. Since the Sobolev-Orlicz spaces are vector spaces, this yields that the equations in iv) imply equality in i) and iii). Moreover, the equations in i) and iii) can be extended by approximation to $\varphi\in W^{1,1}_0(D)^m$, whose gradient can be approximated weakly$^*$ in $L^{\varphi_{\e}}_{\w}(D)^{m\times d}$ or $L^{\varphi_{\rm hom}}(D)^{m\times d}$, respectively, where both spaces are are regarded as subspaces of the dual space of $L^{\varphi^*_{\e}}_{\e}(D)^{m\times d}$ or $L^{\varphi^*_{\rm hom}}(D)^{m\times d}$, respectively. 

\item[b)] While the points i) and iii) imply that $u_{\e}$ and $u_0$ are distributional solutions of the PDEs $-{\rm div}(\partial_{\xi}W(\w,\tfrac{\cdot}{\e},\nabla u))=f_{\e}$ and $-{\rm div}(\partial_{\xi}W(\nabla u))=f$ respectively, they are no weak solutions in the corresponding Sobolev-Orlicz space. This problem is strongly related to the lack of density of smooth functions in Sobolev-Orlicz spaces. As a byproduct of our proof the solutions that minimize the energy satisfy 
\begin{equation*}
	\partial_{\xi}W(\w,\tfrac{\cdot}{\e},\nabla_{\e})\in L^{\varphi^*_{\e}}_{\w}(D)^{m\times d},\quad\quad \partial_{\xi}W_{\rm hom}(\nabla u_0)\in L^{\varphi^*_{\rm hom}}(D)^{m\times d},
\end{equation*}
so that in general the weak formulation of the PDE would make sense in duality as in iv), but we are not able to prove it.

\item[c)]  Concerning the integrands in Example \ref{ex:amples}, the condition in iv) is satisfied for $p(\cdot)$-Laplacians with essentially bounded exponent or double phase functionals $W(\w,x,\xi)=|\xi|^p+a(\w,\tfrac{x}{\e})|\xi|^q$ with no additional restrictions on the exponents. In a more abstract form, it suffices to have an estimate of the form $\overbar W(\w,s\xi)\leq C\overbar W(\w,\xi)+\overbar\Lambda(\w)$ for some $s>1$. In this case one can use the formula for $W_{\rm hom}$ given in Lemma \ref{l.defW_hom} to show that $W_{\rm hom}(s\xi)\leq C (W_{\rm hom}(\xi)+1)$.

\end{itemize}
\end{remark}
\section{Proofs}\label{s.proofs}
Before we start with the different proofs leading to our main results, let us comment on the null sets of $\Omega$ that we need to exclude: besides excluding the set of zero measure, where the properties of $\overbar{W}$ (or $W$) in Assumption \ref{a.1} or \ref{a.2} fail,
\begin{itemize}
	\item we will frequently apply the ergodic theorem in the form of Lemma \ref{l.weakL1} to the random field $W(\w,\tfrac{x}{\e},\xi)=\overbar W(\tau_{\frac{x}{\e}}\w,\xi)$. A priori, the null set where convergence may fail could depend on $\xi$, so let us briefly explain why this is not the case. Considering an element $\w\in\Omega$, where the ergodic theorem holds for all rational matrices $\xi\in\Q^{m\times d}$, we know that for any bounded, measurable set $E\subset\R^d$ it holds that
	\begin{equation*}
		\lim_{\e\to 0}\int_E W(\w,\tfrac{x}{\e},\xi)\dx=\mathbb{E}[\overbar W(\cdot,\xi)]|E|.
	\end{equation*} 
	To extend this property to irrational matrices $\xi_0$, note that the sequence of maps 
	\begin{equation*}
		\xi\mapsto \int_E W(\w,\tfrac{x}{\e},\xi)\dx 
	\end{equation*}
	is still convex and by assumption it is bounded on rational matrices. Since we can write $\R^{m\times d}$ as the countable union of cubes with rational vertices, this implies that the sequence is locally equibounded and by \cite[Theorem 4.36]{FoLe} it is locally equi-Lipschitz. Hence it converges pointwise for all $\xi\in\R^{m\times d}$ and the limit is given by $\mathbb{E}[\overbar W(\cdot,\xi)]$ as this function is the continuous extension of the limit for rational matrices.
	\item we will also apply Lemma \ref{l.weakL1} to the variables $\sup_{|\eta|\leq r}W(\w,\tfrac{x}{\e},\eta)$ or $\sup_{|\zeta|\leq r}W^*(\w,\tfrac{x}{\e},\zeta)$. These terms only appear in bounds, so that we can tacitly restrict $r$ to positive, rational numbers. Moreover, we will use Lemma \ref{l.weakL1} also for the map $\Lambda(\w,\tfrac{x}{\e})$.
	\item we further exclude the null sets where Lemma \ref{l.defW_hom} fails for rational matrices $\xi\in\Q^{m\times d}$ or where Lemma \ref{l.existence_f_hom} for $W$ or the countably many approximations $W_k$ given by Lemma \ref{l.truncateW}.
	\item finally, we apply Lemma \ref{l.weakL1} to the maps $W(\w,\tfrac{x}{\e},\xi+\overbar h_{\xi}(\tau_{x/\e}\w))$ for rational $\xi\in\Q^{m\times d}$, where $\overbar h_{\xi}$ is given by Lemma \ref{l.defW_hom}.
\end{itemize}
If not stated explicitly otherwise, we shall always assume that we have an element $\w$ of the set of full measure such that the above properties hold.

\subsection{Preliminary results: compactness, correctors and the multi-cell formula}
We first show how the bound on the conjugate in \ref{(A3)} yields weak $L^1$-compactness for the gradients of functions with equibounded energy. 
\begin{lemma}\label{l.compactness}
Suppose that $W$ satisfies \ref{(A1)}, \ref{(A2)} and \ref{(A3)}. Let $(u_{\e})_{\e>0}\subset W^{1,1}(D)^m$ be such that
\begin{equation*}
	\sup_{\e\in (0,1)}F_{\e}(\w,u_{\e},D)<+\infty.
\end{equation*}
Then, as $\e\to 0$, the gradients $\nabla u_{\e}$ are relatively weakly compact in $L^{1}(D)^{m\times d}$. If moreover $u_{\e}$ is bounded in $L^1(D)^m$, then, up to subsequences, there exists $u\in W^{1,1}(D)^m$ such that $u_{\e}\rightharpoonup u$ weakly in $W^{1,1}(D)^m$. If $W$ satisfies in addition \ref{(A5)}, the above statement is also true with $L^{1}(D)^{m\times d}$, $L^1(D)^m$ and $W^{1,1}(D)^m$ replaced by $L^{p}(D)^{m\times d}$, $L^p(D)^m$ and $W^{1,p}(D)^m$, respectively.
\end{lemma}
\begin{proof}
Let $A\subset D$ be a measurable set and $v\in L^{\infty}(D)^{m\times d}$ satisfy $\|v\|_{L^{\infty}(D)}\leq 1$. For any $r\geq 1$ the Fenchel-Young inequality and the convexity of $W$ in the last variable yield that
\begin{align*}
	\int_A \langle \nabla u_{\e},v\rangle\dx&\leq\int_A W(\w,\tfrac{x}{\e},\tfrac{1}{r}\nabla u_{\e})\dx+\int_A W^*(\w,\tfrac{x}{\e},rv)\dx
	\\
	&\leq \frac{1}{r}\int_A W(\w,\tfrac{x}{\e},\nabla u_{\e})\dx+\frac{r-1}{r}\int_A W(\w,\tfrac{x}{\e},0)\dx+\int_A\sup_{|\eta|\leq r}W^*(\w,\tfrac{x}{\e},\eta)\dx.
\end{align*}
Taking the supremum over all such $v$'s, the nonnegativity of $W$ and the global energy bound imply that
\begin{equation}\label{eq:A_bound}
	\int_A|\nabla u_{\e}|\dx\leq \frac{C}{r}+\int_A W(\w,\tfrac{x}{\e},0)\dx+\int_A\sup_{|\eta|\leq r}W^*(\w,\tfrac{x}{\e},\eta)\dx.
\end{equation}
Note that $W^*$ inherits the stationary of $W$ and so does the function $\sup_{|\eta|\leq r}W^*(\cdot,\cdot,\eta)$. Combining \ref{(A3)} and Lemma \ref{l.weakL1}, the functions in the last two integrals in \eqref{eq:A_bound} are equiintegrable. Hence we deduce that
\begin{equation*}
\lim_{|A|\to 0} \sup_{|\e|\ll 1}\int_{A}|\nabla u_{\e}|\dx\leq \frac{C}{r}.
\end{equation*}
Letting $r\to +\infty$, it follows that also $\nabla u_{\e}$ is equiintegrable as $\e\to 0$. Since $D$ has finite measure, it remains to show that $\nabla u_{\e}$ is bounded in $L^1(D)^{m\times d}$. To see this, choose $A=D$ and $r=1$ in \eqref{eq:A_bound} and use again Lemma \ref{l.weakL1} to conclude that the functions in the last two integrals in \eqref{eq:A_bound} are also equi-bounded in $L^1(D)$. The last assertion is a standard result for bounded sequences in $W^{1,1}$ once the gradients are weakly compact.

The claim for $W$ satisfying \ref{(A5)} is simpler and follows from 
$$
\sup_{\e\in (0,1)}\int_D|\nabla u_\e|^p\dx\leq \sup_{\e\in(0,1)}F_{\e}(\w,u_{\e},D)<+\infty
$$
and standard results for Sobolev spaces with exponent $p>1$.
\end{proof}

Next, we adapt the construction of correctors in \cite{DG_unbounded,JKO} to the superlinear setting without any polynomial growth of order $p>1$ from below. Define the set 
\begin{align}\label{eq:def_F_pot}
	F_{\rm pot}^1:=\{\overbar h\in L^1(\Omega)^d:\,&\mathbb{E}[\overbar h]=0\text{ and for a.e. }\w\in\Omega \text{ the function }h(x):= \overbar h(\tau_x\w)\in L^1_{\rm loc}(\R^d)^d\text{ satisfies }\nonumber
	\\
	&\;\partial_i h_j-\partial_j h_i=0 \text{ on }\R^d
	\text{ for all $1\leq i,j\leq d$ in the sense of distributions}\}. 
\end{align}
Even though $d\neq 3$ in general, we refer to the property $\partial_i h_j-\partial_j h_i=0$ as being curl-free. The following result is \cite[Lemma 4.11]{RR21}.
\begin{lemma}\label{l.F_pot}
	The space $F_{\rm pot}$ is a closed subspace of $L^1(\Omega)^d$. Moreover, given $\overbar h\in F_{\rm pot}^1$, there exists a map $\varphi:\Omega\to W^{1,1}_{\rm loc}(\R^d)$ such that $\nabla\varphi(\w,x)=\overbar h(\tau_x\w)$ almost surely as maps in $L^1_{\rm loc}(\R^d)^d$ and such that for every bounded set $B\subset\R^d$ the maps $\w\mapsto \varphi(\w,\cdot)$ and $\w\mapsto \nabla\varphi(\w,\cdot)$ are measurable from $\Omega$ to $L^1(B)$ and to $L^1(B)^d$, respectively.
\end{lemma}
The above result allows us to introduce a corrector by solving a minimization problem on the probability space as explained in the lemma below.
\begin{lemma}\label{l.defW_hom}
Suppose $W$ satisfies \ref{(A1)}, \ref{(A2)} and \ref{(A3)}. 
Let $\xi\in\R^{m\times d}$. Then there exists a function $\overbar h_{\xi}\in (F^1_{\rm pot})^m$ such that
\begin{equation*}
W_{\rm hom}(\xi):=\mathbb{E}[\overbar W(\cdot,\xi+\overbar h_{\xi})]=\inf_{\overbar h\in (F^1_{\rm pot})^m}\mathbb{E}[\overbar W(\cdot,\xi+\overbar h)].
\end{equation*}
We call $\phi_{\xi}:\Omega\to W^{1,1}_{\rm loc}(\R^d)^m$ given by Lemma \ref{l.F_pot} applied to the components of $\overbar h_{\xi}$ the corrector associated to the direction $\xi$. We assume in addition that $\dashint_{B_1}\phi_{\xi}(\w,x)\dx=0$. Then, as $\e\to 0$, almost surely for any bounded, open set $A\subset\R^d$ it holds that
\begin{align*}
	\e\phi_{\xi}(\w,\cdot/\e)\rightharpoonup 0\quad  \text{ in }W^{1,1}(A)^m.
\end{align*}
The function $\xi\mapsto W_{\rm hom}(\xi)$ is convex, finite and superlinear at infinity. Moreover, the following is true:
\begin{enumerate}
\item[(i)] Suppose that $W$ satisfies in addition \ref{(A4)}. Then, there exists $C_0<+\infty$ such that for all $\xi\in\R^{m\times d}$ and all $\widetilde{\xi}\in\R^{m\times d}$ with $e_j^T(\xi-\widetilde{\xi})\in \{0,e_j^T\xi\}$ for all $1\leq j\leq m$ it holds that
\begin{equation*}
	W_{\rm hom}(\widetilde{\xi})\leq C_0(W_{\rm hom}(\xi)+1).
\end{equation*}
\item[(ii)] Suppose that $W$ satisfies in addition \ref{(A5)}. Then for all $\xi\in\R^{m\times d}$
\begin{equation*}
	W_{\rm hom}(\xi)\geq |\xi|^p,
\end{equation*}
and almost surely, as $\e\to 0$, it holds that $\e\phi_{\xi}(\w,\cdot/\e)\rightharpoonup 0$ in $W^{1,p}(A)^m$, where $p>d-1$ is the exponent in \ref{(A5)}.
\end{enumerate} 
\end{lemma}

\begin{proof}
The existence of minimizers in the weakly closed set $(F^1_{\rm pot})^m$ (cf. Lemma \ref{l.F_pot}) for the functional $\overbar h\mapsto \mathbb{E}[\overbar W(\cdot,\xi+\overbar h)]$ follows from the direct method of the calculus of variations. Indeed, the convexity of $W$ in the last variable turns the functional weakly lower semicontinuous for the $L^1(\Omega)$-topology, while the relative weak compactness of minimizing sequences can be shown using \ref{(A3)} in the form of $\w\mapsto\sup_{|\eta|\leq r}\overbar W^*(\w,\eta)\in L^1(\Omega)$ for all $r>0$ as in the proof of Lemma \ref{l.compactness}, replacing the oscillating term $\tfrac{x}{\e}$ by $0$ and the physical space by the probability space. Next, note that the constraint $\dashint_{B_1}\phi_{\xi}(\w,x)\dx=0$ does not affect the measurability property stated in Lemma \ref{l.F_pot} as this integral term is a measurable function of $\w$. We continue by showing the weak convergence to zero, dropping the dependence on $\xi$ for the moment. For a.e. $\w\in\Omega$ we have $\nabla \phi(\w,\cdot/\e)=\overbar h(\tau_{\cdot/\e}\w)$ and $\overbar h\in L^1(\Omega)$. Hence the ergodic theorem in the form of Lemma \ref{l.weakL1} implies that for any bounded set $B\subset\R^d$ we have
\begin{equation}\label{eq:weakconvergencegradients}
\nabla \phi(\w,\cdot/\e)\rightharpoonup \mathbb{E}[\overbar h]=0 \quad\text{ in }L^1(B)^{m\times d},
\end{equation}
where we used that $\overbar h\in (F_{\rm pot}^1)^m$ for the last equality. We will show that
\begin{equation}\label{eq:sublinear}
\lim_{\e\to 0}\dashint_{B_1}\e\phi(\w,x/\e)\dx=0,
\end{equation}
which yields the claim by Poincar\'e's inequality considering a ball $B$ such that $B_1\cup A\subset B$. By a density argument one can show that for $r\geq 1$ 
\begin{equation*}
\dashint_{B_1}\frac{1}{r}\phi(\w,ry)\dy=\dashint_{B_1}\frac{1}{r}(\phi(\w,ry)-\phi(\w,y))\dy=\frac{1}{r}\dashint_{B_1}\int_{1}^{r}\nabla\phi(\w,ty)y\,\mathrm{d}t\dy=\frac{1}{r}\int_{1}^{r}\dashint_{B_1}\nabla\phi(\w,ty)y\dy\,\mathrm{d}t.
\end{equation*}
By approximation with continuous functions one can show that the map $t\mapsto\dashint_{B_1}\nabla\phi(\w,ty)y\dy$ is continuous on $(0,+\infty)$  and by \eqref{eq:weakconvergencegradients} it vanishes at infinity. Hence the right-hand side term in the above equality vanishes as $r\to +\infty$. This yields \eqref{eq:sublinear}. 

The function $W_{\rm hom}$ is convex by the convexity of $\overbar W$ in the last variable. It is finite since $0\in (F_{\rm pot}^1)^m$ is admissible in the minimization problem defining $W_{\rm hom}$, so that due to the integrability condition \ref{(A3)} we have
\begin{equation*}
	W_{\rm hom}(\xi)\leq \mathbb{E}[\overbar W(\cdot,\xi)]<+\infty.
\end{equation*}
The superlinearity follows from the lower bound in \eqref{eq:quantitativelineargrowth}. Indeed, for any $C>0$ we have
\begin{equation*}
	C|\xi|=C|\mathbb{E}[\xi+\overbar h_{\xi}]|\leq \mathbb{E}[C|\xi+\overbar h_{\xi}|]\leq \mathbb{E}[\overbar W(\cdot,\xi+\overbar h_{\xi})+\sup_{|\eta|\leq C}\overbar W^*(\cdot,\eta)]=W_{\rm hom}(\xi)+\mathbb{E}[\sup_{|\eta|\leq C}\overbar W^*(\cdot,\eta)].
\end{equation*}
By Assumption \ref{(A3)}, the last expectation is finite for any fixed $C>0$, so that dividing the above inequality by $|\xi|$ and letting $|\xi|\to +\infty$, we infer that
\begin{equation*}
	C\leq \liminf_{|\xi|\to +\infty}\frac{W_{\rm hom}(\xi)}{|\xi|}.
\end{equation*}
Superlinearity at infinity follows from the arbitrariness of $C>0$.

In order to prove the assertion in (i), fix $\xi,\widetilde{\xi}\in\R^{m\times d}$ as in the statement and let $\overbar h_{\xi}\in (F_{\rm pot}^1)^m$ be such that $W_{\rm hom}(\xi)=\mathbb{E}[W(\cdot,0,\xi+\overbar h_{\xi})]$. We define $\widetilde h_{\xi}\in(F_{\rm pot}^1)^m$ via $\widetilde h_\xi^j=0$ if $e_j^T\widetilde \xi=0$ and $\widetilde h_\xi^j=\overbar h_\xi^j$ if $e_j^T\widetilde \xi=e_j^T \xi$. Then, applying \ref{(A4)} pointwise to the two matrix-valued functions $\xi+\overbar h_{\xi}$ and $\widetilde{\xi}+\widetilde h_{\xi}$, we deduce that
\begin{equation*}
W_{\rm hom}(\widetilde{\xi})\leq \mathbb{E}[W(\cdot,0,\widetilde{\xi}+\widetilde h_{\xi})]\leq C\mathbb{E}[W(\cdot,0,\xi+\overbar h_{\xi})]+\mathbb{E}[\overbar\Lambda]=C_0(W_{\rm hom}(\xi)+1),
\end{equation*}
with $C_0=\max\{C,\mathbb{E}[\overbar\Lambda]\}<+\infty$.

The additional statements in (ii) are well-known: the $p$-growth from below implies that $\overbar h_\xi$ satisfies in addition $\overbar h_\xi\in (L^p(\Omega;\R^d))^m$ and from this we deduce the weak convergence in $W^{1,p}(A,\R^m)$ for $\e \phi_\xi(\omega,\cdot/\e)$. The coercivity for $W_{\rm hom}$ follows by $\mathbb E[\overbar h_\xi]=0$, \ref{(A5)} and Jensen's inequality as 
$$
|\xi|^p=|\mathbb E[\xi+\overbar h_\xi]|^p\leq \mathbb E[|\xi+\overbar h_\xi|^p]\leq W_{\rm hom}(\xi).
$$
\end{proof}

As a final result of this section, we show the almost sure existence of the limit in an asymptotic minimization formula in the physical space described in the following lemma. 

\begin{lemma}\label{l.existence_f_hom}
Suppose that $W$ satisfies \ref{(A1)}, \ref{(A2)} and \ref{(A3)}.
%
For a bounded open set $O\subset\R^d$ and $\xi\in\R^{m\times d}$, we define
\begin{equation*}
	\mu_{\xi}(\w,O)=\inf\left\{F_1(\w,u,O):\,u-\xi x\in W^{1,1}_{0}(O,\R^m)\right\}.
\end{equation*}
There exists a convex function $\mu_{\rm hom}:\R^{m\times d}\to[0,+\infty)$ and a set $\Omega'\subset\Omega$ with $\mathbb P[\Omega']=1$ such that the following is true: for every $\omega\in\Omega'$ and every cube $Q=x+(-\eta,\eta)^d\subset\R^d$ and $\xi\in\R^{m\times d}$ it holds that	\begin{equation*}
		\mu_{\rm hom}(\xi)=\lim_{t\to +\infty}\frac{1}{|tQ|}\mu_{\xi}(\w,tQ).
	\end{equation*}
\end{lemma}
\begin{proof}
We apply the subadditive ergodic theorem. According to Lemma \ref{l.measurable} the function $\w\mapsto \mu_{\xi}(\w,O)$ is measurable. To show its integrability, we test the affine function $u(x)=\xi x$ as a candidate in the infimum problem. Since $F_1$ is nonnegative, we obtain
	\begin{equation}\label{eq:pointwisebound}
		0\leq \mu_{\xi}(\w,O)\leq \int_O W(\w,x,\xi)\,\dx.
	\end{equation}
Tonelli's theorem and stationarity of $W$ yield that
\begin{equation}\label{eq:mu_bound}
	\mathbb{E}\left[\mu_{\xi}(\cdot,O)\right]\leq \int_O\mathbb{E}[W(\cdot,x,\xi)]\,\mathrm{d}x=\mathbb{E}[\overbar W(\cdot,\xi)]|O|.
\end{equation}
Hence $\mu_{\xi}(\cdot,O)\in L^1(\Omega)$. By stationarity of $W$ also  $\mu_{\xi}$ is $\tau$-stationary in the sense that
\begin{equation}\label{eq:mu_stationary}
	\mu_{\xi}(\tau_z\w,O)=\mu_{\xi}(\w,O+z)\quad\text{ for all }\w\in\Omega.
\end{equation}
Finally, if $(U_j)_{j=1}^n\subset\R^d$ are bounded open sets with
	\begin{equation*}
		\bigcup_{j=1}^n U_j\subset O, \quad U_j\cap U_k=\emptyset\text{ for all }1\leq j<k\leq n, \quad |O\setminus \bigcup_{j=1}^n U_j|=0,
	\end{equation*} 
	and for every $1\leq j\leq n$ we consider a map $v_j\in \xi x+W_0^{1,1}(U_j,\R^m)$, then the function $v=\sum_{j=1}^n v_j\chi_{U_j}$ belongs to $\xi x+W^{1,1}_0(O,\R^m)$ and therefore
	\begin{equation*}
		\mu_{\xi}(\w,O)\leq F_1(\w,v,O)=\sum_{j=1}^n F_1(\w,v_j,U_j).
	\end{equation*}
	Minimizing the right-hand side with respect to the variables $v_j$, we deduce subadditivity in the form of
	\begin{equation}\label{eq:mu_subadditive}
		\mu_{\xi}(\w,O)\leq \sum_{j=1}^N\mu_{\xi}(\w,U_j).
	\end{equation}
	By the subadditive ergodic theorem (see \cite[Theorem 2.7]{AkKr}), a.s. there exists the a~priori random limit
	\begin{equation}\label{eq:integer_convergence}
		\mu_{0}(\w,\xi):=\lim_{\substack{n\to +\infty\\ n\in\N}}\frac{1}{|nQ|}\mu_{\xi}(\w,nQ)
	\end{equation}
	for all cubes of the form $Q=z+(-k,k)^d$ with integer vertices $k\in\N$ and $z\in\Z^d$. 
	
	The extension to arbitrary sequences $t\to +\infty$ and general cubes $Q=x+(-\eta,\eta)^d$ with $x\in\R^d$ and $\eta>0$ follows by approximation as in \cite[Lemma 4.3]{RR21}, relying on the fact that $\overbar W(\cdot,\xi)\in L^1(\Omega)$, which allows us to apply the additive ergodic theorem in the form of Lemma \ref{l.weakL1} to the error terms that are due to integrating $W(\w,x,\xi)$ over sets with small measure (relatively to the scale $t^d$). Similarly, one can prove that $\mu_0$ is invariant under every group action $\tau_{z}$, so by ergodicity it is deterministic. We call this value $\mu_{\rm hom}(\xi)$.

	To fix the issue that the exceptional set where convergence fails may depend on $\xi$, fix $\xi_0,\xi\in\R^{m\times d}$. For comparing $\mu_{\xi}$ and $\mu_{\xi_0}$, we consider cubes of different size. For a cube $Q=x+(-\eta,\eta)^d$ and $s>0$ set $Q(s)=x+(-s\eta,s\eta)^d$ and fix $\delta>0$. There exists a smooth cut-off function $\varphi=\varphi_{\delta,t}\in C_c^{\infty}(\R^d,[0,1])$ such that 
	\begin{equation*}
		\varphi\equiv 1\text{ on }tQ,\quad\quad\varphi\equiv 0\text{ on }\R^d\setminus tQ(1+\delta/2),\quad\quad\|\nabla\varphi\|_{L^{\infty}(\R^d)}\leq \frac{C_Q}{\delta t}.
	\end{equation*}
	Given $v\in\xi x+W^{1,1}_0(tQ,\R^m)$, extend it to $\R^d$ setting $v(x)=\xi x$ on $\R^d\setminus tQ$ and define $\widetilde{v}\in \xi_0 x+W_0^{1,1}(tQ(1+\delta),\R^m)$ as
	\begin{equation*}
		\widetilde{v}(x)=\varphi(x) v(x)+(1-\varphi(x))\xi_0 x.
	\end{equation*}
	From the properties of $\varphi$, we infer that
	\begin{equation*}
		\mu_{\xi_0}(\w,tQ(1+\delta))\leq F_1(\w,\widetilde{v},tQ(1+\delta))
		\leq\int_{tQ}W(\w,x,\nabla v(x))\dx+\int_{tQ(1+\delta)\setminus tQ}W(\w,x,\nabla \widetilde{v}(x))\dx.
	\end{equation*}
	For $x\in tQ(1+\delta)\setminus tQ$ the product rule yields
	\begin{equation*}
	\nabla \widetilde{v}(x)=\frac{1}{2}2\nabla\varphi(x)\otimes(\xi-\xi_0)x+\frac{1}{2} \varphi(x)2\xi+\frac{1}{2}(1-\varphi(x))2\xi_0.
	\end{equation*}
	Since $\tfrac{1}{2}+\tfrac{1}{2}\varphi(x)+\tfrac{1}{2}(1-\varphi(x))=1$, the convexity of $W$ in the last variable allows us to bound the error term by
	\begin{align*}
	\int_{tQ(1+\delta)\setminus tQ}W(\w,x,\nabla \widetilde{v}(x))\dx&\leq \int_{tQ(1+\delta)\setminus tQ} W\left(\w,x,2\nabla\varphi(x)\otimes(\xi-\xi_0)x\right)\dx\\
	&\quad+\int_{tQ(1+\delta)\setminus tQ}W\left(\w,x,2\xi\right)+W\left(\w,x,2\xi_0\right)\dx.
	\end{align*}
	Since $|x|\leq C_Q(1+\delta)t$ on $tQ(1+\delta)$, for $\delta\leq 1$ we deduce that
	\begin{equation*}
	\left|2\nabla\varphi\otimes (\xi-\xi_0)x\right|\leq \frac{C|\xi-\xi_0|}{\delta},
	\end{equation*}
	where $C$ is independent of $\delta$ and $t$. We assume from now on that $|\xi-\xi_0|\leq \delta$.  Passing to the infimum over $v$, we deduce that 
	\begin{equation*}
		\mu_{\xi_0}(\w,tQ(1+\delta))\leq  \mu_{\xi}(\w,tQ)+\int_{tQ(1+\delta)\setminus tQ}\sup_{|\eta|\leq C}W(\w,x,\eta)+W\left(\w,x,2\xi\right)+W\left(\w,x,2\xi_0\right)\dx.
	\end{equation*}
	Assume further that $\xi_0\in\mathbb{Q}^{m\times d}$ and consider $\w$ in the set of full probability such that the limit of $t\mapsto |tQ|^{-1}\mu_{\xi_0}(\w,tQ)$ at $+\infty$ exists for all rational matrices $\xi_0$ and all cubes. The additive ergodic theorem \ref{l.weakL1} applied to the last integral then yields that
	\begin{equation}\label{eq:liminf_est}
		\mu_{\rm hom}(\xi_0)\leq \liminf_{t\to +\infty}\frac{1}{|tQ|}\mu_{\xi}(\w,tQ)+\mathbb{E}\left[\sup_{|\eta|\leq C}\overbar W(\cdot,\eta)+\overbar W\left(\cdot,2\xi\right)+\overbar W\left(\cdot,2\xi_0\right)\right]\left(1-\frac{1}{(1+\delta)^d}\right).
	\end{equation}
	From an analogous construction with the cubes $tQ$ and $tQ(1-\delta)$ we further infer that
	\begin{equation*}
		\mu_{\xi}(\w,tQ)\leq \mu_{\xi_0}(\w,tQ(1-\delta))+\int_{tQ\setminus tQ(1-\delta)}\sup_{|\eta|\leq C}W(\w,x,\eta)+W\left(\w,x,2\xi\right)+W\left(\w,x,2\xi_0\right)\dx,
	\end{equation*}
	which again by the ergodic theorem implies that
	\begin{equation}\label{eq:limsup_est}
		\limsup_{t\to +\infty}\frac{1}{|tQ|}\mu_{\xi}(\w,tQ)\leq \mu_{\rm hom}(\xi_0) +\mathbb{E}\left[\sup_{|\eta|\leq C}\overbar W(\cdot,\eta)+\overbar W\left(\cdot,2\xi\right)+\overbar W\left(\cdot,2\xi_0\right)\right]\left(1-(1-\delta)^d\right).
	\end{equation}
	Combining the two estimates \eqref{eq:liminf_est} and \eqref{eq:limsup_est} yields
	\begin{align*}
		\limsup_{t\to +\infty}\frac{1}{|tQ|}\mu_{\xi}(\w,tQ)&\leq \liminf_{t\to +\infty}\frac{1}{|tQ|}\mu_{\xi}(\w,tQ)
		\\
		&\quad+\mathbb{E}\left[\sup_{|\eta|\leq r_0}\overbar W(\cdot,\eta)+\overbar W\left(\cdot,2\xi\right)+\overbar W\left(\cdot,2\xi_0\right)\right]\left(2-(1-\delta)^d-\frac{1}{(1+\delta)^d}\right).	
	\end{align*}
	Considering a sequence of rational matrices that converges to $\xi$ and then letting $\delta\to 0$, we deduce that 
	\begin{equation*}
		\limsup_{t\to +\infty}\frac{1}{|tQ|}\mu_{\xi}(\w,tQ)\leq \liminf_{t\to +\infty}\frac{1}{|tQ|}\mu_{\xi}(\w,tQ),
	\end{equation*}
	so that the limit exists and hence the convergence holds for a uniform (with respect to $\xi$ and $Q$) set of full probability. The convexity of $\mu_{\rm hom}$ is a consequence of the convexity of the map $\xi\mapsto\mu_{\xi}(\w,O)$, which itself follows from the convexity of the map $\xi\mapsto W(\w,x,\xi)$. 
\end{proof}

\subsection{Proof of the \texorpdfstring{$\Gamma$}{}-liminf inequality by truncation of \texorpdfstring{$W$}{}}\label{sec:liminf}

For the $\Gamma$-liminf inequality, we approximate $W$ from below by integrands with polynomial growth. This technique was already implemented for functionals satisfying $p$ growth from below with $p>1$, see e.g.\ \cite{Mue87,DG_unbounded}. Here, we generalize this method to cover the case of merely superlinear growth condition as in Assumption~\ref{a.1}, where in contrast to the case with $p$-growth from below, the approximation only satisfies so-called non-standard or $p-q$-growth conditions. 

\begin{lemma}\label{l.truncateW}
Assume that $W$ satisfies \ref{(A1)}, \ref{(A2)} and \ref{(A3)}. Then there exists an increasing sequence $W_k:\Omega\times\R^d\times\R^{m\times d}\to [0,+\infty)$ that still satisfies \ref{(A1)}, \ref{(A2)} and \ref{(A3)} and in addition
\begin{itemize}
	\item[a)] $W(\w,x,\xi)=\sup_{k\in\N}W_k(\w,x,\xi)$ for a.e. $x\in\R^d$ and all $\xi\in\R^{m\times d}$;
	\item[b)] $W_k(\w,x,\xi)\leq C_k(1+|\xi|^q)$ for some $1<q<1^*$ (not depending on $k$) and $C_k\in (0,+\infty).$
\end{itemize}
\end{lemma}
\begin{proof}
Due to Remark \ref{r.stationary_extension}, it is enough to perform the construction at the level of $\overbar{W}$, so that \ref{(A1)} comes for free. We first define a suitable lower bound for $\overbar W$. Set $\widetilde{\ell}(\w,\xi)={\rm co}(\min\{\ell(\w,|\xi|),q^{-1}|\xi|^q\})$, where ${\rm co}$ denotes the convex envelope and $\ell$ is the convex, monotone, superlinear function given by Remark \ref{r.comments} (ii). Then $\widetilde{\ell}$ is jointly measurable\footnote{Since all functions are real-valued and bounded below by $0$, the convex envelope is given by the biconjugate, which preserves joint measurability as explained in Remark \ref{r.comments} (ii).} and $\widetilde{\ell}(\w,\xi)\leq W(\w,\xi)$. Next, since $W\geq 0$, following the proof of \cite[Theorem 6.36]{FoLe} there exist measurable and bounded functions $a_i:\Omega\to \R$ and $b_i:\Omega\to\R^m$ such that for all $\w\in\Omega$ with $\overbar{W}(\w,\cdot)$ being convex we have
\begin{equation*}
\overbar W(\w,\xi)=\sup_{i\in\N}\{a_i(\w)+b_i(\w)\cdot\xi\}.	
\end{equation*}
 We define an approximation of $W$ by setting
\begin{equation}
	\widetilde{W}_k(\w,\xi)=\max_{i\leq k}\{a_i(\w)+b_i(\w)\cdot\xi\}.	
\end{equation}
Then $\widetilde{W}_k$ is jointly measurable, increasing in $k$, convex in $\xi$ and satisfies $W_k\uparrow W$ pointwise in $\xi$ for a typical element $\w\in\Omega$ (cf. the beginning of Section \ref{s.proofs}). Finally, we define 
\begin{equation*}
	\overbar W_k(\w,\xi)=\max\{\widetilde{W}_k(\w,\xi),\widetilde{\ell}(\w,\xi)\}.
\end{equation*}
Then $\overbar W_k$ is still jointly measurable, convex in $\xi$ (as required in \ref{(A2)}), increasing in $k$, and due to the lower bound $\widetilde{\ell}(\w,\xi)\leq \overbar W(\w,\xi)$ it also follows that $\overbar W_k\uparrow \overbar W$. Moreover, since the individual functions $a_i,b_i$ are bounded, we have that
\begin{equation*}
	\overbar W_k(\w,\xi)\leq \widetilde{W}_k(\w,\xi)+|\xi|^q\leq C_k(1+|\xi|^q),
\end{equation*}
which also implies the upper bound in \ref{(A3)}. Next, we show the bound on the conjugate function in \ref{(A3)}. To bound it from below, note that by definition of the conjugate
\begin{equation*}
	\sup_{|\eta|\leq r}\overbar W_k^*(\cdot,\eta)\geq - \overbar W_k(\cdot,0)\geq -\overbar W(\cdot,0)\in  L^1(\Omega).
\end{equation*}
For the upper bound, we will use several times that the convex envelope for finite, convex functions $f:\R^{m\times d}\to\R$ that are bounded below by an affine function (here zero) can be characterized by the biconjugate function, i.e., $f^{**}={\rm co}(f)$ (cf. \cite[Remark 4.93 (iii)]{FoLe}) and that $f^{***}=f^*$ (see \cite[Proposition 4.88]{FoLe}), which in combination yields that 
\begin{equation}\label{eq:eraseenvelope}
	{\rm co}(f)^*=f^*.
\end{equation}
Since $f\leq g$ implies $f^*\geq g^*$, it suffices to show that 
\begin{equation*}
\sup_{|\eta|\leq r}(\widetilde{\ell})^{*}(\cdot,\eta)\in L^1(\Omega).
\end{equation*}
According to \eqref{eq:eraseenvelope} and the definition of $\widetilde{\ell}$, we know that
\begin{align*}
(\widetilde{\ell})^{*}(\w,\eta)&=\sup_{\xi\in\R^{m\times d}}\left(\langle \xi,\eta\rangle-\min\{\ell(\w,\xi),q^{-1}|\xi|^q\}\right)=\sup_{\xi\in\R^{m\times d}}\max\{\langle\xi,\eta\rangle-\ell(\w,0,\xi),\langle \xi,\eta\rangle-q^{-1}|\xi|^q\}
\\
&= \max\{\ell^*(\w,\eta),(\tfrac{1}{q}|\cdot|^q)^*(\eta)\}\leq \ell^*(\w,\eta)+\tfrac{1}{q'}|\eta|^{q'},
\end{align*}
where $q'$ is the conjugate exponent to $q$. The function $\eta\mapsto\tfrac{1}{q'}|\eta|^{q'}$ is deterministic and locally bounded, and therefore it suffices to show that
\begin{equation*}
	\sup_{|\eta|\leq r}\ell^*(\w,\eta)\in L^1(\Omega).
\end{equation*}
Recall the construction of $\ell$ in Remark \ref{r.comments} (ii) as the convex envelope of the radial minimum of $W$. According to \eqref{eq:eraseenvelope} we have that
\begin{equation*}
	\ell^*(\w,\eta)=\sup_{\xi\in\R^{m\times d}}\left(\langle\xi,\eta\rangle-\inf_{|z|=|\xi|}\overbar W(\w,z)\right)=\sup_{\xi\in\R^{m\times d}}\sup_{|z|=|\xi|}\left(\langle\xi,\eta\rangle-\overbar W(\w,z)\right).
\end{equation*} 
Taking the supremum over $|\eta|\leq r$ and using the commutativity of suprema, we deduce that
\begin{align*}
\sup_{|\eta|\leq r}\ell^*(\w,\eta)&=\sup_{\xi\in\R^{m\times d}}\sup_{|z|=|\xi|}\sup_{|\eta|\leq r}\left(\langle\xi,\eta\rangle-\overbar W(\w,z)\right) = \sup_{\xi\in\R^{m\times d}}\sup_{|z|=|\xi|}\left(r|\xi|-\overbar W(\w,z)\right)
\\
&= \sup_{z\in\R^{m\times d}}(r|z|-\overbar W(\w,z))=\sup_{z\in\R^{m\times d}}\sup_{|\eta|\leq r}(\langle \eta,z\rangle-\overbar W(\w,z))
\\
&=\sup_{|\eta|\leq r}\sup_{z\in\R^{m\times d}}(\langle \eta,z\rangle-\overbar W(\w,z))=\sup_{|\eta|\leq r}\overbar W^*(\w,\eta),
\end{align*}
which concludes the proof of \ref{(A3)}. 
\end{proof}
In the next proposition we prove the $\Gamma$-liminf inequality for the truncated energies defined via the integrand $W_k$ given by the previous lemma.
\begin{proposition}\label{p.lb_k}
Assume that $W$ satisfies \ref{(A1)}, \ref{(A2)} and \ref{(A3)}. Let $D\subset\R^d$ be an bounded, open set, $u\in W^{1,1}(D)^m$ and $(u_\e)_\e\subset L^1(D)^m$ satisfying $u_{\e}\to u$ in $L^1(D)^m$ as $\e\to 0$. For any $k\in\N$ let $\mu_{{\rm hom},k}(\xi)$ be the function given by Lemma \ref{l.existence_f_hom} applied to the integrand $W_k$ given by Lemma \ref{l.truncateW}. Then
\begin{equation*}
	\int_D \mu_{{\rm hom},k}(\nabla u(x))\dx\leq\liminf_{\e\to 0}\int_D W_k(\w,\tfrac{x}{\e},\nabla u_{\e}(x))\dx.	
\end{equation*}
\end{proposition}
\begin{proof}
To reduce notation, we drop $k$ from the notation and assume in addition that $W$ satisfies the growth condition
\begin{equation}\label{eq:polygrowth}
W(\w,x,\xi)\leq C(|\xi|^q+1)	
\end{equation}
for some $1<q<1^*$ and a.e. $x\in\R^d$. Without loss of generality, we assume that the limit inferior in the statement is finite and, passing to a non-relabeled subsequence, it is actually a limit. Following the classical blow-up method, define the absolutely continuous Radon-measure $m_{\e}$ on $D$ by its action on Borel sets $B\subset D$ via
\begin{equation*}
	m_{\e}(B)=\int_B W(\w,\tfrac{x}{\e},\nabla u_{\e}(x))\dx.
\end{equation*}
By our assumption, the sequence of measures $m_{\e}$ is equibounded, so that (up to passing to a further non-relabeled subsequence) $m_{\e}\overset{\star}{\rightharpoonup} m$ for some nonnegative finite Radon measure $m$ (possibly depending on $\w$). Using  Lebesgue's decomposition theorem, we can write $m= \widetilde{f}(x)\L^d+ \nu$, with $\nu$ a nonnegative measure $\nu$ that is singular with respect to the Lebesgue measure. Since $D$ is open, the weak$^*$ convergence implies that
\begin{equation*}
	\liminf_{\e\to 0}F_{\e}(\w,u_{\e},D)=\liminf_{\e\to 0}m_{\e}(D)\geq m(D)\geq \int_{D}\widetilde{f}(x)\dx.
\end{equation*} 
Hence it suffices to show that $\widetilde{f}(x_0)\geq \mu_{\rm hom}(\nabla u(x_0))$ for a.e. $x_0\in D$. The Besicovitch differentiation theorem \cite[Theorem 1.153]{FoLe} and Portmanteau's theorem imply that for a.e. $x_0\in D$ we have 
\begin{equation}\label{eq:besicovitch}
	\widetilde{f}(x_0)=\lim_{r\to 0}\frac{m(\overline{Q_{r}(x_0)})}{r^d}\geq \limsup_{r\to 0}\limsup_{\e\to 0}\frac{m_{\e}(\overline{Q_{r}(x_0)})}{r^d}\geq \limsup_{r\to 0}\limsup_{\e\to 0}\frac{m_{\e}(Q_{r}(x_0))}{r^d}.
\end{equation}
Therefore it suffices to show that for a.e. $x_0\in D$ we have (along the chosen subsequence in $\e$)
	\begin{equation*}
		\limsup_{r\to 0}\limsup_{\e\to 0}\dashint_{Q_r(x_0)}W(\w,\tfrac{x}{\e},\nabla u_{\e}(x))\dx\geq \mu_{\rm hom}(\nabla u(x_0)).
	\end{equation*}
	We claim further that
	\begin{equation}\label{eq:radialapprox}
		\limsup_{\eta\uparrow 1}\limsup_{r\to 0}\limsup_{\e\to 0}\dashint_{Q_r(x_0)}W(\w,\tfrac{x}{\e},\eta\nabla u_{\e}(x))\dx\leq \limsup_{r\to 0}\limsup_{\e\to 0}\dashint_{Q_r(x_0)}W(\w,\tfrac{x}{\e},\nabla u_{\e}(x))\dx,
	\end{equation}
	so that it suffices to show that
	\begin{equation}\label{eq:blowup}
		\limsup_{\eta\uparrow 1}\limsup_{r\to 0}\limsup_{\e\to 0}\dashint_{Q_r(x_0)}W(\w,\tfrac{x}{\e},\eta\nabla u_{\e}(x))\dx\geq \mu_{\rm hom}(\nabla u(x_0)).
	\end{equation}
	To verify \eqref{eq:radialapprox}, note that due to the convexity of the map $\xi\mapsto W(\w,y,\xi)$, for $\eta\in (0,1)$ it holds that
	\begin{equation*}
	W(\w,\tfrac{x}{\e},\eta\nabla u_{\e}(x))\leq	\eta W(\w,\tfrac{x}{\e},\nabla u_{\e}(x))+(1-\eta)W(\w,\tfrac{x}{\e},0).
	\end{equation*}
	Since $\eta\leq 1$, integrating this inequality over $Q_r(x_0)$ and taking the average, we obtain that
	\begin{equation*}
	\dashint_{Q_r(x_0)}W(\w,\tfrac{x}{\e},\eta\nabla u_{\e}(x))\dx\leq \dashint_{Q_r(x_0)}W(\w,\tfrac{x}{\e},\nabla u_{\e}(x))\dx+(1-\eta)\dashint_{Q_r(x_0)}W(\w,\tfrac{x}{\e},0)\dx
	\end{equation*}
	and thus it suffices to show that the last integral vanishes. As $\e\to 0$, we can apply the ergodic theorem and obtain
	\begin{equation*}
	\lim_{\eta\uparrow 1}\lim_{r\to 0}\lim_{\e\to 0}(1-\eta)\dashint_{Q_r(x_0)}W(\w,\tfrac{x}{\e},0)=\lim_{\eta\uparrow 1}(1-\eta)\mathbb{E}[\overbar W(\cdot,0)]=0.
	\end{equation*}
	To show \eqref{eq:blowup}, we fix $x_0$ to be a Lebesgue point of $u$ and $\nabla u$, and such that \eqref{eq:besicovitch} holds true. Using the Besicovitch derivation theorem on convex, open sets (see \cite[Remark 1.154 (ii)]{FoLe}), we may impose additionally that
	\begin{equation}\label{eq:besicovitch_open}
	\widetilde{f}(x_0)=\lim_{r\to 0}\frac{m(Q_{r}(x_0))}{r^d}.
	\end{equation}
	For such $x_0$ we define the linearization of $u$ at $x_0$ by $L_{u,x_0}(x)=u(x_0)+\nabla u(x_0)(x-x_0)$.
	
	In what follows, we drop the dependence on $x_0$ from cubes and tacitly assume that they are centered at $x_0$. We modify $u_{\e}$ close to $\partial Q_r$ such that the modification attains the boundary value $L_{u,x_0}$: for $0<\eta<1$ we pick a cut-off function $\varphi_{\eta}\in C_c^{\infty}(Q_r,[0,1])$ such that $\varphi_{\eta}(x)=1$ on $Q_{\eta r}$, which can be chosen such that $\|\nabla\varphi_{\eta}\|_{\infty}\leq \tfrac{C}{(1-\eta)r}$. Define then the function $u_{\e,\eta}:D\to\R^m$ by
	\begin{equation*}
		  u_{\e,\eta}=\eta\varphi_{\eta}u_{\e}+\eta(1-\varphi_{\eta})L_{u,x_0}.
	\end{equation*}
	Since $u_{\e,\eta}=\eta u_{\e}$ on $Q_{\eta r}$, we can estimate the energy of $u_{\e,\eta}$ on $Q_r$ by
	\begin{equation}\label{eq:split}
		F_{\e}(\w,u_{\e,\eta},Q_r)\leq F_{\e}(\w,\eta u_{\e},Q_r)+F_{\e}(\w,u_{\e,\eta},Q_{r}\setminus Q_{\eta r}).
	\end{equation}
	We argue that the last term is asymptotically negligible relative to $r^d$. To reduce notation, we set $S_{\eta}^r=Q_{r}\setminus Q_{\eta r}$. 
	Since $u_{\e}\in W^{1,1}(D)^m$ due to the global energy bound, the product rule yields that
	\begin{equation*}
		\nabla u_{\e,\eta}=\eta\varphi_{\eta}\nabla u_{\e}+\eta(1-\varphi_{\eta})\nabla u(x_0)+(1-\eta)\frac{\eta\nabla\varphi_{\eta}\otimes( u_{\e}-L_{u,x_0})}{1-\eta},
	\end{equation*}
	so that $0\leq \eta,\varphi_{\eta}\leq 1$ and the convexity of $\xi\mapsto W(\w,y,\xi)$ imply the estimate
	\begin{align}\label{eq:preparetoaverage}
	F_{\e}(\w,u_{\e,\eta},Q_{r}\setminus Q_{\eta r})&\leq F_{\e}(\w,u_{\e},Q_{r}\setminus Q_{\eta r})+F_{\e}(\w,L_{u,x_0},Q_{r}\setminus Q_{\eta r})\nonumber
	\\
	&\quad+\int_{Q_{r}\setminus Q_{\eta r}}W\left(\w,\tfrac{x}{\e},\frac{\eta\nabla\varphi_{\eta}(x)\otimes( u_{\e}(x)-L_{u,x_0}(x))}{1-\eta}\right)\dx.
	\end{align}
	Let us first estimate the last term, using the polynomial bound \eqref{eq:polygrowth}. Inserting the uniform bound on $\nabla\varphi_{\eta}$, we find that
	\begin{align*}
		&\int_{Q_{r}\setminus Q_{\eta r}}W\left(\w,\tfrac{x}{\e},\frac{\eta\nabla\varphi_{\eta}(x)\otimes( u_{\e}(x)-L_{u,x_0}(x))}{1-\eta}\right)\dx
		\\
		&\leq C\int_{Q_{r}\setminus Q_{\eta r}}(1-\eta)^{-2q}r^{-q}|u_{\e}(x)-L_{u,x_0}(x)|^q+1\dx.
	\end{align*}
	Inserting this estimate into \eqref{eq:preparetoaverage} and the resulting bound into \eqref{eq:split}, we infer that
	\begin{align}\label{eq:splitall}
	\frac{1}{r^d}F_{\e}(\w,u_{\e,\eta},Q_r)&\leq \frac{1}{r^d}F_{\e}(\w,\eta u_{\e},Q_r)+\frac{1}{r^d}F_{\e}(\w,u_{\e},Q_r\setminus Q_{\eta r})+\frac{1}{r^d}F_{\e}(\w,L_{u,x_0},Q_r\setminus Q_{\eta r})\nonumber
	\\
	&\quad +C(1-\eta)^{-2q}\dashint_{Q_r}\left|\frac{u_{\e}(x)-L_{u,x_0}(x)}{r}\right|^q\dx+C\frac{1}{r^d}|Q_r\setminus Q_{\eta r}|.
	\end{align} 
	We now let $\e\to 0$. To estimate the left-hand side term from below, note that $u_{\e,\eta}=\eta L_{u,x_0}$ on $\partial Q_r$, so that by a change of variables we have $F_{\e}(\w,u_{\e,\eta},Q_r)\geq \mu_{\eta\nabla u(x_0)}(\w,Q_r/\e)\e^{d}$. In particular, from Lemma \ref{l.existence_f_hom} we deduce that
	\begin{equation}\label{eq:lowerbound}
		\mu_{\rm hom}(\eta\nabla u(x_0))\leq\liminf_{\e\to 0}\frac{1}{r^d}F_{\e}(\w,u_{\e,\eta},Q_r).
	\end{equation}
	To treat the right-hand side terms in \eqref{eq:splitall}, we note that the second term is $r^{-d}m_{\e}(Q_r\setminus Q_{\eta r})$, so that we can estimate it using Portmanteau's theorem and switching to the closure of $Q_r\setminus Q_{\eta r}$. For the third term we can apply the ergodic theorem to the integrand $W(\w,\tfrac{x}{\e},\nabla u(x_0))$. In order to pass to the limit in the fourth term, we note that $u_{\e}$ is bounded in $W^{1,1}(D)^m$ and hence strongly converging in $L^q(D)$ due to the Sobolev embedding. In total, we obtain that
	\begin{align*}
	\mu_{\rm hom}(\eta\nabla u(x_0))&\leq\limsup_{\e\to 0}\frac{1}{r^d}F_{\e}(\w,\eta u_{\e},Q_r)+\frac{1}{r^d}m(\overline{Q_r}\setminus Q_{\eta r})+\mathbb{E}[\overbar W(\cdot,\nabla u(x_0))](1-\eta^d)\\
	&\quad +C(1-\eta)^{-2q}\dashint_{Q_r}\left|\frac{u(x)-L_{u,x_0}(x)}{r}\right|^q\dx+C(1-\eta^d).
	\end{align*}
	Next, we let $r\to 0$. On the one hand, by \eqref{eq:besicovitch} and \eqref{eq:besicovitch_open} we have that
	\begin{equation*}
	\lim_{r\to 0}\frac{1}{r^d}m(\overline{Q_r}\setminus Q_{\eta r})=\lim_{r\to 0}\frac{1}{r^d}m(\overline{Q_r})-\lim_{r\to 0}\frac{\eta^d}{(\eta r)^d}m(Q_{\eta r})=\widetilde{f}(x_0)-\eta^d\widetilde{f}(x_0).
	\end{equation*}
	On the other hand, the $L^{1^*}$-differentiability of $W^{1,1}$-functions (cf. \cite[Theorem 2, p. 230]{EvGa}) implies that
	\begin{equation*}
		\lim_{r\to 0}\dashint_{Q_r}\left|\frac{u(x)-L_{u,x_0}(x)}{r}\right|^q\dx=0.	
	\end{equation*}
	Gathering these two pieces of information, we obtain that
	\begin{equation*}
		\mu_{\rm hom}(\eta \nabla u(x_0))\leq \limsup_{r\to 0}\limsup_{\e\to 0}\frac{1}{r^d}F_{\e}(\w,\eta u_{\e},Q_r)+(1-\eta^d)\Big(\widetilde{f}(x_0)+\mathbb{E}[\overbar W(\cdot,\nabla u(x_0))]+C\Big).
	\end{equation*}
	Finally, letting $\eta\to 1$, the continuity of $\mu_{\rm hom}$ (which follows from convexity) yields
	\begin{equation*}
		\mu_{\rm hom}(\nabla u(x_0))\leq \limsup_{\eta\to 1}\limsup_{r\to 0}\limsup_{\e\to 0}\frac{1}{r^d}F_{\e}(\w,\eta u_{\e},Q_r),
	\end{equation*}
	which coincides with \eqref{eq:blowup} and thus we conclude the proof.
\end{proof}
We next need to prove that $W_{{\rm hom},k}= \mu_{{\rm hom},k}$ because the formula for $W_{{\rm hom},k}$ allows us to pass to the limit in $k$, while this is not obvious for the multi-cell formula defining $\mu_{{\rm hom},k}$. 

\begin{lemma}\label{l.F_pot_formula}
Assume that $W$ satisfies \ref{(A1)}, \ref{(A2)} and \ref{(A3)}. For $k\in\N$ let $W_k$ be the integrand given by Lemma \ref{l.truncateW} and denote by $W_{{\rm hom},k}$ and $\mu_{{\rm hom},k}$ the functions given by the Lemmata \ref{l.defW_hom} and \ref{l.existence_f_hom}, respectively. Then for all $\xi\in\R^{m\times d}$ it holds that
	\begin{equation}\label{eq:multi=single}
		W_{{\rm hom},k}(\xi)=\mu_{{\rm hom},k}(\xi).
	\end{equation}
\end{lemma}
\begin{proof} 
	Again we drop $k$ from the notation and just assume that \eqref{eq:polygrowth} holds true. Fix a cube $Q\subset D$. For every $\e>0$ and $t\in (0,1)$ consider a function $u_{\e,t}\in W^{1,1}(D)^m$ with $\dashint_Q u_{\e,t}\dx=0$ and $\dashint_Q\nabla u_{\e,t}\dx=0$, satisfying
	\begin{equation*}
		\frac{1}{|Q|}F_{\e}(\w,u_{\e,t},Q)=\min\left\{\dashint_Q W(\w,\tfrac{x}{\e},t\xi+\nabla u(x))\dx:\,\dashint_Q\nabla u\dx=0\right\}, 
	\end{equation*}
	where the minimum exists due to convexity and the superlinear growth of $W$.
	Since $u=0$ is admissible in the above minimization problem and has uniformly bounded energy with respect to $\e\to 0$, it follows from the superlinear growth of $W$ and the Poincar\'e inequality that up to a subsequence (not relabeled) we have $u_{\e,t}\to u_t$ as $\e\downarrow 0$ in $L^1(Q)^m$ for some $u_t\in W^{1,1}(Q)^m$ with $\dashint_Q\nabla u_t\dx=0$. Due to Proposition \ref{p.lb_k} with $Q=D$ and Jensen's inequality we have that
	\begin{align}\label{eq:zero_average_gradient}
		\mu_{\rm hom}(t\xi)&=\mu_{\rm hom}\left(\dashint_Qt\xi+\nabla u_t(x)\dx\right)\leq \dashint_Q \mu_{\rm hom}(t\xi+\nabla u_t(x))\dx\nonumber
		\\
		&\leq \liminf_{\e\to 0}\min\left\{\dashint_Q W(\w,\tfrac{x}{\e},t\xi+\nabla u(x))\dx:\,\dashint_Q \nabla u\dx=0\right\}.
	\end{align}
	Next, let $\phi_{\xi}:\Omega\to W^{1,1}_{\rm loc}(\R^d,\R^m)$ be the function given by Lemma \ref{l.defW_hom}. Define the $W^{1,1}_{\rm loc}(\R^d,\R^m)$-valued random field $v_{\e,t}$ by
	\begin{equation*}
		v_{\e,t}(x)=t\e\phi_{\xi}(\w)(\tfrac{x}{\e})-\dashint_{Q}t\nabla\phi_{\xi}(\w)(\tfrac{y}{\e})x\,\mathrm{d}y.
	\end{equation*}
	Then $\int_Q\nabla v_{t,\e}\dx=0$ and therefore almost surely\footnote{Here we prove the equality of two deterministic quantities. Therefore we can exclude a null set depending on any fixed $\xi$.}
	\begin{align}\label{eq:corrector_good}
		\inf\left\{\dashint_Q W(\w,\tfrac{x}{\e},t\xi+\nabla u(x))\dx:\,\dashint_Q\nabla u\dx=0\right\}&\leq \dashint_QW(\w,\tfrac{x}{\e},t\xi+\nabla v_{t,\e}(x))\dx\nonumber
		\\
		&=\dashint_Q W\left(\w,\tfrac{x}{\e},t\xi+t\overbar h_{\xi}(\tau_{x/\e}\w)-t\dashint_Q\overbar h_{\xi}(\tau_{y/\e}\w)\dy\right)\dx,
	\end{align}
	where $\overbar h_{\xi}$ is given by Lemma \ref{l.defW_hom}. Let us set $\overbar H_{\xi,\e}(\w)=\dashint_Q \overbar h_{\xi}(\tau_{y/\e}\w)\dy$. Due to convexity we can bound the last integral by
	\begin{align*}
	\dashint_Q W\left(\w,\tfrac{x}{\e},t\xi+t\overbar h_{\xi}(\tau_{x/\e}\w)-t\overbar H_{\xi,\e}(\w)\right)\dx	&\leq t\dashint_Q W\left(\w,\tfrac{x}{\e},\xi+\overbar h_{\xi}(\tau_{x/\e}\w)\right)\dx
	\\
	&\quad+(1-t)\dashint_Q W\left(\w,\tfrac{x}{\e},-\tfrac{t}{1-t}\overbar H_{\xi,\e}(\w)\right)\dx.
	\end{align*}
	The ergodic theorem \ref{l.weakL1} implies that 
	\begin{align}
		\lim_{\e\to 0}\overbar H_{\xi,\e}(\w)=\lim_{\e\to 0}\dashint_Q\overbar h_{\xi}(\tau_{y/\e}\w)\dy&=\mathbb{E}[\overbar h_{\xi}]=0,\label{eq:ergodic_H}
		\\
		\lim_{\e\to 0}\dashint_Q W(\w,\tfrac{x}{\e},\xi+\overbar h_{\xi}(\tau_{x/\e}\w))\dx&=\mathbb{E}[\overbar W(\cdot,\xi+h_{\xi})]=W_{\rm hom}(\xi).\label{eq:ergodic_energy}
	\end{align}
	Due to \eqref{eq:ergodic_H} we may assume for $\e>0$ sufficiently small that $|\overbar H_{\xi,\e}(\w)|\leq (1-t)/t$. Inserting the above convexity estimate into \eqref{eq:corrector_good}, by the ergodic theorem we find that
	\begin{equation*}
	\mu_{\rm hom}(t\xi)\leq W_{\rm hom}(\xi)+(1-t)\mathbb{E}\left[\sup_{|\eta|\leq 1}\overbar W(\cdot,\eta)\right].
	\end{equation*}
	Letting $t\uparrow 1$, we conclude the estimate $\mu_{\rm hom}(\xi)\leq W_{\rm hom}(\xi)$ due to the continuity of $\mu_{\rm hom}$. 
	
	We still need to show the reverse inequality. Here we can closely follow \cite[Lemma 4.13]{RR21}. Lemma \ref{l.measurable} implies that for every $\e>0$ there exists a measurable function $u_{\xi,\e}:\Omega\to W^{1,1}_0(Q/\e)^m$ such that almost surely
	\begin{equation*}
		F_1(\w,\xi x+u_{\xi,\e}(\w),Q/\e)=\mu_{\xi}(\w,Q/\e):=\inf\{F_1(\w,u,Q_1/\e):\,u\in \xi x+W^{1,1}_0(Q/\e)^m\}.
	\end{equation*}
	However, we switch to a jointly measurable map. From \cite[Lemma 16, p. 196]{DS} we deduce that there exist $\mathcal{F}\otimes\L^d$-measurable functions $v_{\xi,\e}:\Omega\times Q/\e\to\R^m$ and $b_{\xi,\e}:\Omega\times Q/\e\to\R^{m\times d}$ such that $v_{\xi,\e}(\w,\cdot)=u_{\xi,\e}(\w)$ and $b_{\xi,\e}(\w,\cdot)=\nabla u_{\xi,\e}(\w)$ for a.e. $\w\in\Omega$. In particular, for a.e. $\w\in\Omega$ we have $v_{\xi,\e}(\w,\cdot)\in W_0^{1,1}(Q/\e)^m$ and $\nabla v_{\xi,\e}(\w,\cdot)=b_{\xi,\e}(\w,\cdot)$. With a slight abuse of notation, we therefore write $b_{\xi,\e}=\nabla v_{\xi,\e}$. By Remark \ref{r.stationary_extension} we can assume that the set of $\w$, for which these properties hold, is invariant under the group action $\tau_{x}$ for almost all $x\in\R^d$. Finally, we extend $v_{\xi,\e}$ and $\nabla v_{\xi,\e}$ to $\Omega\times(\R^d\setminus Q/\e)$ by $0$. This extension is jointly measurable on $\Omega\times\R^d$. We now define $\overbar h_{\xi,\e}\in L^1(\Omega)^{m\times d}$ by
	\begin{equation*}
		h_{\xi,\e}(\w)=\dashint_{Q/\e}\nabla v_{\xi,\e}(\tau_{-y}\w,y)\dy.
	\end{equation*}
	Note that $\overbar h_{\xi,\e}$ is well defined and measurable due to the joint measurability of $\nabla v_{\xi,\e}$ and the joint measurability of the group action. To see that it is integrable, we can use Fubini's theorem and a change of variables in $\Omega$ to deduce that
	\begin{align*}
		\int_{\Omega}|\xi+\overbar h_{\xi,\e}(\w)|\,\mathrm{d}\mu&\leq \dashint_{Q/\e}\int_{\Omega}|\xi+\nabla v_{\xi,\e}(\tau_{-y}\w,y)|\,\mathrm{d}\mu\dy=\dashint_{Q/\e}\int_{\Omega}|\xi+\nabla v_{\xi,\e}(\w,y)|\,\mathrm{d}\mu\dy
		\\
		&=\int_{\Omega}\dashint_{Q/\e}|\xi+\nabla v_{\xi,\e}(\w,y)|\dy\,\mathrm{d}\mu\leq C\int_{\Omega}\dashint_{Q/\e}W(\w,y,\xi+\nabla v_{\xi,\e}(\w,y))\dy\,\mathrm{d}\mu+C,
	\end{align*}
	where we used the superlinearity \eqref{eq:quantitativelineargrowth} of $W$ for a.e. $x\in\R^d$ (cf. Remark \ref{r.stationary_extension}). The last term is finite, since for a.e. $\w\in\Omega$ the function $\nabla v_{\xi,\e}(\w,\cdot)$ is the gradient of an energy minimizer on $Q/\e$. We argue that $\overbar h_{\xi,\e}\in (F_{\rm pot}^1)^m$. Since for a.e. $\w\in\Omega$ the function $\nabla v_{\xi,\e}$ is the weak gradient of $u_{\xi,\e}(\w)\in W^{1,1}_0(Q/\e)^m$, it follows from Fubini's theorem and a change of variables in $\Omega$ that
	\begin{equation*}
		\int_{\Omega}\overbar h_{\xi,\e}(\w)\,\mathrm{d}\mu=\int_{\Omega}\underbrace{\dashint_{Q/\e}\nabla v_{\xi,\e}(\w,y)\dy}_{=0 \text{ almost surely}}\,\mathrm{d}\mu=0.
	\end{equation*}
	Hence, it suffices to show that all rows of $x\mapsto \overbar h_{\xi,\e}(\tau _x\w)$ satisfy the curl-free condition of Definition~\ref{eq:def_F_pot}. To this end, we derive a suitable formula for the distributional derivative of this map. Fix $\theta\in C^{\infty}_c(\R^d)$ and an index $1\leq j\leq d$. Since $\nabla v_{\xi,\e}(\w,\cdot)=0$ on $\R^d\setminus (Q/\e)$, we can write 
	\begin{align*}
		\int_{\R^d}	\overbar h_{\xi,\e}(\tau_x\w)\partial_j\theta(x)\dx&=\int_{\R^d}\int_{\R^d}\frac{\nabla v_{\xi,\e}(\tau_{x-y}\w,y)}{|Q/\e|}\partial_j\theta(x)\dy\dx=\int_{\R^d}\int_{\R^d}\frac{\nabla v_{\xi,\e}(\tau_{z}\w,x-z)}{|Q/\e|}\partial_j\theta(x)\,\mathrm{d}z\dx
		\\
		&=\int_{\R^d}\int_{\R^d}\frac{\nabla v_{\xi,\e}(\tau_{z}\w,y)}{|Q/\e|}\partial_j\theta(y+z)\dy\,\mathrm{d}z.
	\end{align*}
	In order to conclude that $\overbar h_{\xi,\e}\in (F_{\rm pot}^1)^m$, it suffices to note that for a.e. $\w\in\Omega$ and almost every $z\in\R^d$ the function $y\mapsto \nabla v_{\xi,\e}(\tau_z\w,y)$ is the gradient of the Sobolev function $u_{\xi,\e}(\tau_z\w)\in W^{1,1}_0(Q/\e)^m$, so that the curl-free conditions follows. Now we can conclude the proof. Since $\overbar h_{\xi,\e}\in (F_{\rm pot}^1)^m$, it follows from Jensen's inequality that
	\begin{align*}
		W_{\rm hom}(\xi)&\leq \mathbb{E}[\overbar W(\cdot,\xi+\overbar h_{\xi,\e})]\leq \int_{\Omega}\dashint_{Q/\e}\overbar W(\w,\xi+\nabla  v_{\xi,\e}(\tau_{-y}\w,y))\dy\,\mathrm{d}\mu
		\\
		&=\int_{\Omega}\dashint_{Q/\e}\overbar W(\tau_{y}\w,\xi+\nabla  v_{\xi,\e}(\w,y))\dy\,\mathrm{d}\mu=\int_{\Omega}\dashint_{Q/\e}W(\w,y,\xi+\nabla  v_{\xi,\e}(\w,y))\dy\,\mathrm{d}\mu
		\\
		&=\frac{1}{|Q/\e|}\mathbb{E}[\mu_{\xi}(\w,Q/\e)],
	\end{align*}
	where we used the stationarity of $W$, and that $\nabla v_{\xi,\e}(\w,\cdot)=\nabla u_{\xi,\e}(\w)$ almost surely in the last step. Passing to the limit as $\e\to 0$, we obtain from $L^1$-convergence in the subadditive ergodic theorem (see \cite[Theorem 2.3, p. 203]{Krengel}) the missing inequality
	\begin{equation*}
		W_{\rm hom}(\xi)\leq \mu_{\rm hom}(\xi).
	\end{equation*}
\end{proof}
Next, we prove that $\lim_{k\to +\infty}W_{{\rm hom},k}=W_{\rm hom}$, which allows us to prove the $\Gamma$-liminf inequality by truncation.
\begin{lemma}\label{l.klimit}
Assume that $W$ satisfies \ref{(A1)}, \ref{(A2)} and \ref{(A3)}. For $k\in\N$ let $W_k$ be the integrand given by Lemma \ref{l.truncateW} and denote by $W_{{\rm hom},k}$ and $W_{\rm hom}$ the functions given by Lemma \ref{l.defW_hom} applied to $W_k$ and $W$, respectively. Then for all $\xi\in\R^{m\times d}$ it holds that
	\begin{equation*}
		\lim_{k\to +\infty}W_{{\rm hom},k}(\xi)=W_{\rm hom}(\xi).
	\end{equation*}
\end{lemma}
\begin{proof}
Since $W_k\leq W$, we clearly have $W_{{\rm hom},k}(\xi)\leq W_{\rm hom}$. For the reverse inequality, note that due to monotonicity it suffices to prove the claim up to subsequences. Let $\overbar h_{\xi,k}\in (F_{\rm pot}^1)^m$ be a minimizer defining $W_{{\rm hom},k}(\xi)=\mathbb{E}[W_k(\cdot,0,\xi+\overbar h_{\xi,k})]$. Due to the uniform superlinearity of $W_k$ (recall the monotonicity in $k$), we know that (up to a subsequence) $\overbar h_{\xi,k}\rightharpoonup \widetilde{h}\in (F_{\rm pot}^1)^m$ in $L^1(\Omega)^{m\times d}$. Fix $h\in (F_{\rm pot}^1)^m$. Then due to monotone convergence and lower semicontinuity, for every $m\in\N$ we have
\begin{align*}
\mathbb{E}[W(\cdot,0,\xi+h)]&= \lim_{k\to +\infty}\mathbb{E}[W_k(\cdot,0,\xi+h)]\geq \lim_{k\to +\infty}\mathbb{E}[W_k(\cdot,0,\xi+\overbar h_{\xi,k})]
\\
&\geq \liminf_{k\to +\infty}\mathbb{E}[W_m(\cdot,0,\xi+\overbar h_{\xi,k})]\geq \mathbb{E}[W_m(\cdot,0,\xi+\widetilde{h})].
\end{align*}
Letting $m\to +\infty$, we find that $\mathbb{E}[W(\cdot,0,\xi+h)]\geq \mathbb{E}[W(\cdot,0,\xi+\widetilde{h})]$. Hence $\widetilde{h}$ is a minimizer and, as proven above 
\begin{equation*}
W_{\rm hom}(\xi)=\mathbb{E}[W(\cdot,0,\xi+\widetilde{h})]\leq \lim_{k\to +\infty}\mathbb{E}[W_k(\cdot,0,\xi+\overbar h_{\xi,k})]=\lim_{k\to +\infty}W_{{\rm hom},k}(\xi).	
\end{equation*}
\end{proof}
Now we are in a position to prove the $\Gamma$-liminf inequality for the original functionals.
\begin{proposition}\label{p.lb}
Assume that $W$ satisfies \ref{(A1)}, \ref{(A2)} and \ref{(A3)}. Let $D\subset\R^d$ be a bounded, open set, $u\in W^{1,1}(D)^m$ and let $(u_\e)_\e\subset L^1(D)^m$ be a sequence satisfying $u_{\e}\to u$ in $L^1(D)^m$ as $\e\to 0$. Then
	\begin{equation*}
		\int_D W_{\rm hom}(\nabla u(x))\dx\leq\liminf_{\e\to 0}\int_D W(\w,\tfrac{x}{\e},\nabla u_{\e}(x))\dx,	
	\end{equation*}
with $W_{\rm hom}$ defined in Lemma \ref{l.defW_hom}.
\end{proposition}
\begin{proof}
Since $W_k\leq W$ for all $k\in\N$, it follows from the liminf inequality proven in Proposition \ref{p.lb_k} that
\begin{eqnarray*}
	\liminf_{\e\to 0}\int_D W(\w,\tfrac{x}{\e},\nabla u_{\e}(x))\dx&\geq& 	\liminf_{\e\to 0}\int_D W_k(\w,\tfrac{x}{\e},\nabla u_{\e}(x))\dx\geq\int_D\mu_{{\rm hom},k}(\nabla u(x))\dx
	\\
	&\overset{\text{Lemma } \ref{l.F_pot_formula}}{=}&\int_D W_{{\rm hom},k}(\nabla u(x))\dx.
\end{eqnarray*}
Letting $k\to +\infty$, the claim follows from Lemma \ref{l.klimit} and the monotone convergence theorem.
\end{proof}

\subsection{Construction of a recovery sequence} In Proposition~\ref{p.lb} we proved the $\Gamma$-$\liminf$ inequality assuming only \ref{(A1)}, \ref{(A2)} and \ref{(A3)}. In the proof of the $\Gamma$-$\limsup$ inequality, see Proposition~\ref{p.recoverysequence} below, we need to assume in addition either the mild monotonicity condition \ref{(A4)} or the coercivity condition \ref{(A5)}. The main technical part in the proof of Proposition~\ref{p.recoverysequence} is to construct a recovery sequence for affine functions that satisfy prescribed affine boundary values. In order to attain the boundary condition, we introduce a cut-off that causes an additional error term. Assuming the monotonicity condition \ref{(A4)}, we  combine the cut-off with a truncation argument to control the additional error. Without a structure assumption such as \ref{(A4)}, truncation will not work and we rely on Sobolev embedding instead. Similar arguments were used e.g.\ in \cite{DG_unbounded,Mue87} assuming \ref{(A5)} with $p>d$, exploiting the compact embedding of $W^{1,p}$ into $L^\infty$. The improvement from $p>d$ to $p>d-1$ comes from suitably chosen cut-off functions in combination with the compact embedding of $W^{1,p}(S_1)\subset L^\infty(S_1)$, where $S_1$ denotes the $d-1$-dimensional unit sphere, provided $p>d-1$. The following lemma encodes the needed compactness property.
\begin{lemma}\label{L:optim}
Let $N\in\mathbb N$ and $p>d-1$. For every $\rho>0$ there exists $C_{\rho,N}<\infty$ (depending also on $d$ and $m$) such that the following is true: 

For any ball $B_R=B_R(x_0)$, any $u_1,\dots,u_N\in W^{1,p}(B_R)^m$ and $\delta\in(0,\frac12]$ there exists $\eta\in W_0^{1,\infty}(B_R)$ satisfying
\begin{equation}\label{L:optim:eta}
0\leq \eta\leq 1,\quad \eta=1\quad\mbox{in $B_{(1-\delta)R}$},\quad\|\nabla \eta\|_{L^\infty(B_R)}\leq \frac{2}{\delta R}
\end{equation}
and for all $i\in\{1,\ldots,N\}$
\begin{equation}\label{L:optim:claim}
\|\nabla \eta \otimes u_i\|_{L^\infty(B_R)}\leq\frac{\rho}{\delta}\biggl(\frac1{\delta R^d}\int_{B_R\setminus B_{(1-\delta) R}}|\nabla u_i|^p\dx\biggr)^\frac1p+\frac{C_{\rho,N}}{\delta R}\biggl(\frac1{\delta R^d}\int_{B_R\setminus B_{(1-\delta) R}}| u_i|^p\dx\biggr)^\frac1p.
\end{equation}
\end{lemma}

\begin{proof}

Without loss of generality, we suppose $x_0=0$. Set $S_1:=\{x\in \R^d\,:\,|x|=1\}$.

\smallskip

{\bf Step 1} We prove the statement for $u_1,\dots,u_N\in C^1(B_R)^m$. For $i\in\{1,\dots,N\}$ and $C:=4N$, we set
\begin{equation}\label{def:U}
U_i:=\biggl\{r\in[(1-\delta) R,R]\,:\,\int_{S_1}|\nabla u_i(rz)|^p\,d\mathcal H^{d-1}(z)\leq \frac{C}{\delta(1-\delta)^{d-1}R^d}\int_{B_R\setminus B_{(1-\delta) R}}|\nabla u_i|^p\dx\biggr\}.
\end{equation}
An elementary application of Fubinis Theorem and the definition of $U_i$ in the form
\begin{align*}
\int_{B_R \setminus B_{(1-\delta) R}}|\nabla u(x)|^p\dx=&\int_{(1-\delta) R}^Rr^{d-1}\int_{S_1}|\nabla u(r z)|^p\,d\mathcal H^{d-1}(z)\,dr\\
\geq&((1-\delta) R)^{d-1}\int_{((1-\delta) R,R)\setminus U_i}\int_{S_1}|\nabla u(rz)|^p\,d\mathcal H^{d-1}(z)\,dr\\
>&\frac{C(\delta R-|U_i|)}{\delta R}\int_{B_R \setminus B_{(1-\delta) R}}|\nabla u(x)|^p\dx
\end{align*}
imply $|U_i|\geq (1-\frac1C)\delta R$, or equivalently $|(1-\delta R,R)\setminus U_i|\leq \frac{\delta R}C$. An analogous computation yields that
\begin{equation}\label{def:U2}
V_i:=\biggl\{r\in[(1-\delta) R,R]\,:\,\int_{S_1}|u_i(rz)|^p\,d\mathcal H^{d-1}(z)\leq \frac{C}{\delta(1-\delta)^{d-1}R^d}\int_{B_R\setminus B_{(1-\delta) R}}|u_i|^p\dx\biggr\}
\end{equation}
satisfies $|V_i|\geq (1-\frac1C)\delta R$, or equivalently $|(1-\delta R,R)\setminus V_i|\leq \frac{\delta R}C$. Setting $U:=\bigcap_{i=1}^NU_i\cap V_i$, from the choice $C=4N$ we obtain
\begin{equation}\label{bound:lowerU}
|U|\geq \delta R-\frac{2N}C\delta R=\frac{\delta R}2. 
\end{equation}
Next, we define $\eta\in W^{1,\infty}(B_R;[0,1])$ by
$$
\eta(x)=\tilde \eta(|x|),\quad\mbox{where}\quad \tilde \eta(r)=\begin{cases}1&\mbox{if $r\in(0,(1-\delta) R)$,}\\\displaystyle \frac{1}{|U|}\int_{r}^R\chi_{U}(s)\,ds&\mbox{if $r\in ((1-\delta) R,R)$.}\end{cases}
$$
By definition, we have that $0\leq \eta\leq1$, $\eta=1$ in $B_{(1-\delta) R}$, $\eta\in W_0^{1,\infty}(B_R)$ and for $x=rz$ with $r\in[0,R]$ and $z\in S_1$
\begin{equation}
|\nabla \eta(rz)|=\begin{cases}0&\mbox{if $r\notin U$,}\\\frac1{|U|}&\mbox{if $r\in U$.}
\end{cases}
\end{equation}
Hence, recalling \eqref{bound:lowerU}, the map $\eta$ satisfies all the properties in \eqref{L:optim:eta}. 

Next, we use $p>d-1\geq1$ in the form that the embedding $W^{1,p}(S_1)^m\subset L^\infty(S_1)^m$ is compact. In particular, for every $\rho>0$ there exists $C_\rho$ such that for all $v\in C^1(S_1)^m$ it holds
\begin{align*}
\sup_{S_1} |v|\leq \rho \|D_\tau v\|_{L^p(S_1)}+C_\delta\|v\|_{L^p(S_1)}
\end{align*} 
where $D_\tau$ denotes the tangential derivative (see \cite[Lemma~5.1]{Lionsbook}). Applying the above estimate to $v_r\in C^1(S_1,\mathbb R^m)$ defined by $v_r(z):=u(r z)$ for all $z\in S^1$ with $u\in C_1(B_R,\R^m)$, we obtain with the chain rule
\begin{align}
\sup_{z\in S_1} |u(rz)|\leq \rho r \|\nabla u(r\cdot) \|_{L^p(S_1)}+C_\rho\|u(r\cdot)\|_{L^p(S_1)}.\label{compact:embedding1}
\end{align} 
Hence, for every $\rho>0$ there exists $C_\rho<\infty$ such that for all $x=rz$ with $r=|x|$ and $z=\frac{x}{|x|}$ 
\begin{align*}
|(\nabla \eta\otimes u_i)(x)|=&|(\nabla \eta\otimes u_i)(r z)|\leq\frac1{|U|}\sup_{r\in U}|u_i(rz)|\\
\leq&\frac{2}{\delta R}  \sup_{r\in U} \biggl(r\rho \biggl(\int_{S_1}|\nabla u_i(r z)|^p\,d\mathcal H^{d-1}(z)\biggr)^\frac1p+C_\rho \biggl(\int_{S_1}| u_i(r z)|^p\,d\mathcal H^{d-1}(z)\biggr)^\frac1p\biggr)\\
\leq&\frac{2}{\delta R}  \sup_{r\in U} \biggl(R\rho \biggl(\frac{2^{d+1}N}{\delta R^d}\int_{B_R\setminus B_{(1-\delta)R}}|\nabla u|^p\dx\biggr)^\frac1p+C_\rho \biggl(\frac{2^{d+1}N}{\delta R^d}\int_{B_R\setminus B_{(1-\delta) R}}|\nabla u|^p\dx\biggr)^\frac1p\biggr),
\end{align*}
where we use the definition of $U$ and $1-\delta\geq \frac12$ in the last inequality. The claimed estimate \eqref{L:optim:claim} follows by redefining the choice of $\rho>0$ (depending on $N$).

\smallskip

{\bf Step 2} Conclusion. Consider $u_1,\dots,u_N\in W^{1,p}(B_R)^m$. By standard density results, we find $(u_{i,j})_j\subset C^\infty(B_R)^m$ such that $u_{i,j}\to u_i$ in $W^{1,p}(B_R)^m$. By Step~2, we find for every $j\in\mathbb N$ a cut-off function $\eta_j\in W_0^{1,\infty}(B_R)$ satisfying 
\begin{align}\label{L:optim:eta1}
&0\leq \eta_j\leq 1,\quad \eta_j=1\quad\mbox{in $B_{(1-\delta)R}$},\quad\|\nabla \eta_j\|_{L^\infty(B_R)}\leq \frac{2}{\delta R},\\
&\|\nabla \eta_j \otimes u_{i,j}\|_{L^\infty(B_R)}\leq\frac{\rho}{\delta}\biggl(\frac1{\delta R^d}\int_{B_R\setminus B_{(1-\delta) R}}|\nabla u_{i,j}|^p\dx\biggr)^\frac1p+\frac{C_{\rho,N}}{\delta R}\biggl(\frac1{\delta R^d}\int_{B_R\setminus B_{(1-\delta) R}}| u_{i,j}|^p\dx\biggr)^\frac1p.\label{L:optim:claim1}
\end{align}
In view of the bounds in \eqref{L:optim:eta1} and the Banach-Alaoglu Theorem, there exists $\eta\in W_0^{1,\infty}(B_R)$ such that up to subsequences (not relabeled) $\eta_j\stackrel{\star}{\rightharpoonup}\eta$ in $W^{1,\infty}(B_R)$. Moreover, $\eta$ also satisfies the bounds in \eqref{L:optim:eta}. Since $\nabla \eta_j\stackrel{\star}\rightharpoonup \nabla \eta$ weakly$^*$ in $L^\infty(B_R;\R^d)$ and $u_{i,j}\to u_i$ (strongly) in $L^p(B_R)^m$, we deduce that $\nabla \eta_j\otimes u_{i,j}$ converges weakly in $L^p(B_R)^{m\times d}$ to $\nabla \eta\otimes u_i$ and by the boundedness of the right-hand side in \eqref{L:optim:claim1} also weakly$^*$ in $L^{\infty}(B_R)^{m\times d}$. Hence the claimed estimate \eqref{L:optim:claim} follows from \eqref{L:optim:claim1} and weak$^*$ lower-semicontinuity of the norm.

\end{proof}

Now we are in a position to state and prove the $\Gamma$-$\limsup$ inequality.  
\begin{proposition}\label{p.recoverysequence}
Let $W$ satisfy Assumption~\ref{a.1}. Let $D\subset\R^d$ be a bounded, open set with Lipschitz boundary and $u\in W^{1,1}(D)^m$. Then there exists sequence $(u_\e)_{\e>0}\subset W^{1,1}D)^m$ such that $u_\e\to  u$ in $L^1(D)$ and
\begin{equation}
\lim_{\e\downarrow0} F_\e(\omega,u_\e,D)=\int_D W_{\rm hom}(\nabla  u(x))\,dx.
\end{equation}
 
\end{proposition}

\begin{proof}

{\bf Step 1.} Local recovery sequence for affine functions with rational gradient.

We claim that for any bounded, open set $A\subset\R^d$, any affine function  $u:A\to\R^m$ with $\nabla u=\xi\in\mathbb{Q}^{m\times d}$ and $u_\e=u+\e \phi_{\xi}(\omega,\cdot/\e)$, where $\phi_{\xi}$ is given as in Lemma~\ref{l.defW_hom}, it holds that 
\begin{equation}\label{pf:recovery:lim1}
\lim_{\e\downarrow0}\|u_\e-u\|_{L^1(A)}=0\quad\mbox{and}\quad\lim_{\e\downarrow0}\int_AW(\omega,\tfrac{x}\e,\nabla u_\e(x))\dx=\int_AW_{\rm hom}(\nabla u(x))\dx.
\end{equation}
%
%
%

Indeed, the convergence $u_{\e}\to u$ in $L^1(A)^m$ is a direct consequence of Lemma~\ref{l.defW_hom}, while the ergodic theorem in the form of Lemma \ref{l.weakL1} yields that
\begin{align*}
\lim_{\e\downarrow0}\dashint_AW(\omega,\tfrac{x}\e,\nabla u_\e(x))\dx=&\lim_{\e\downarrow0}\dashint_{\frac1\e A}W(\omega,x,\xi+\nabla \phi_\xi(\omega,x))\dx
=\lim_{\e\downarrow0}\dashint_{\frac1\e A}W(\omega,x,\xi+ h_\xi(\tau_x\omega))\dx\\
=&\lim_{\e\downarrow0}\dashint_{\frac1\e A}\overbar W(\tau_x\omega,\xi+ h_\xi(\tau_x\omega))\dx
=W_{\rm hom}(\xi)=\dashint_AW_{\rm hom}(\nabla u)\dx.
\end{align*}

\smallskip

{\bf Step 2.} Recovery sequence with prescribed boundary values for affine functions -- Assumption \ref{(A4)}.

Let $A\subset\R^d$ be a bounded, open set and let $u$ be an affine function with $\nabla u=\xi\in\R^{m\times d}$. We claim that there exists a sequence $(v_\e)_\e\subset W^{1,1}(A)^m$ satisfying
\begin{equation}\label{eq:seqrec}
(v_\e)_\e\subset u+W_0^{1,1}(A)^m,\quad \lim_{\e\downarrow0}\biggl(\| v_\e-u\|_{L^1(A)}+\biggl|\int_A W(\omega,\tfrac{x}\e,\nabla v_\e)-W_{\rm hom}(\nabla u)\dx\biggr|\biggr)=0.
\end{equation}
Indeed, let $(u_j)_j$ be a sequence of affine functions satisfying $u_j\to u$ in $W^{1,\infty}(A)^m$ and $\nabla u_j\in \mathbb Q^{m\times d}$ for all $j\in\mathbb N$. In view of Step~1 there exists for every $j\in\mathbb N$ a sequence $(u_{j,\e})_\e\subset W^{1,1}(A)^m$ satisfying $u_{j,\e}\to u_j$ in $L^1(A)^m$ as $\e\downarrow0$ and \eqref{pf:recovery:lim1}. We glue $u_{j,\e}$ to $u$ at the boundary of $A$ and truncate peaks of $u_{j,\e}$ in $A$. More precisely, we consider for $\e,\delta,s>0$ and $j\in\N$ the function $v_{\e,\delta,j,s}\in W^{1,1}(A)^m$  given by
\begin{equation}\label{eq:ansatz}
 v_{\e,\delta,j,s}:=\eta T_s(u_{j,\e})+(1-\eta)u,
\end{equation} 
where $\eta=\eta_\delta \in C^1(A;[0,1])$ is a smooth cut-off function satisfying
\begin{equation}
\begin{cases}
{\eta}=0&\mbox{ on $\{x\in A\,:\,{\rm dist}(x,\partial A)<\delta\}$}\\\eta=1&\mbox{ on $\{x\in A\,:\,{\rm dist}(x,\partial A)>2\delta\}$}
\end{cases}
\end{equation}
and $T_s(u_{j,\e})\in L^\infty\cap W^{1,1}(A)^m$ is obtained from $u_{j,\e}$ by 'component-wise' truncation, that is,
\begin{equation}\label{eq:defTruncation}
T_s(u_{j,\e})\cdot e_k:=\max\{\min\{u_{j,\e}\cdot e_k,s\},-s\}\qquad\mbox{for $k\in\{1,\dots,m\}$.}
\end{equation} 
By the product rule, we obtain for every $t\in[0,1)$ 
\begin{align*}
t\nabla v_{\e,\delta,j,s}=&t(1-\eta)\nabla u+t\eta \nabla T_s(u_{j,\e})+(1-t)\frac{t}{1-t}\nabla \eta\otimes (T_s(u_{j,\e})-u).
\end{align*}
Hence, by convexity of $W$,
\begin{align*}
F_\e(\omega,tv_{\e,\delta,j,s},A)\leq& \int_{A}t(1-\eta)W(\omega,\tfrac{x}\e,\xi)\dx+t\int_{A}\eta W(\omega,\tfrac{x}\e,\nabla T_s(u_{j,\e}))\dx\\
&+(1-t)\int_{A} W(\omega,\tfrac{x}\e,\frac{t}{1-t}\nabla \eta\otimes (T_s(u_{j,\e})-u))\dx.
\end{align*}
Assumption \ref{(A3)} in combination with the ergodic theorem and the definition of $\eta$ imply that for every $t\in[0,1]$
\begin{equation}\label{P:limsup:I}
\limsup_{\delta\downarrow0}\limsup_{\e\downarrow0}\int_{A}t(1-\eta)W(\omega,\tfrac{x}\e,\xi)\dx\leq\lim_{\delta\to 0}|A\cap\{\dist(\cdot,\partial A)\leq 2\delta\}|\,\mathbb{E}[\overbar W(\cdot,\xi)]=0.
\end{equation}
Next, we show that for every $t\in[0,1]$ it holds that
\begin{equation}\label{P:limsup:II}
\limsup_{\delta\downarrow0}\limsup_{s\uparrow\infty}\limsup_{j\uparrow\infty}\limsup_{\e\downarrow0}t\int_{A}\eta W(\omega,\tfrac{x}\e,\nabla T_s(u_{j,\e}))\dx\leq |A|W_{\rm hom}(\xi).
\end{equation}
For this, we start with the decomposition
\begin{align*}
&\int_{A}t\eta W(\omega,\tfrac{x}\e,\nabla T_s(u_{j,\e}))\dx\\
\leq&\int_{A\cap \{|u_{j,\e}|_\infty<s\}} W(\omega,\tfrac{x}\e,\nabla u_{j,\e})\dx+\int_{A\cap \{|u_{j,\e}|_\infty\geq s\}} W(\omega,\tfrac{x}\e,\nabla T_s(u_{j,\e})))\dx.
\end{align*}
Since $W\geq0$, for all $s>0$ we have that
\begin{align}\label{pf:rec:IIa}
\limsup_{j\uparrow\infty}\limsup_{\e\downarrow0}\int_{A\cap \{|u_{j,\e}|_\infty<s\}} W(\omega,\tfrac{x}\e,\nabla u_{j,\e})\dx\leq& \limsup_{j\uparrow\infty}\limsup_{\e\downarrow0}\int_{A}W(\omega,\tfrac{x}\e,\nabla u_{j,\e})\dx\notag\\
=&\limsup_{j\uparrow\infty}|A|W_{\rm hom}(\xi_j)=|A|W_{\rm hom}(\xi),
\end{align}
where we used \eqref{pf:recovery:lim1} and the continuity of $W_{\rm hom}$. In order to estimate the term where truncation is active, we use the definition of $T_s$ in the form
$$
\forall k\in\{1,\dots,m\}:\quad e_k^T \nabla T_s(u_{j,\e})\in\{0,e_k^T\nabla u_{j,\e}\} \quad\mbox{a.e.}
$$
For $s\geq \|u\|_{L^\infty(D)}+2$ and $j$ sufficiently large such that it holds $\|u_j-u\|_{L^{\infty}(D)}<1$, we have
$$
\{|u_{j,\e}|_\infty \geq s\}\subset\{|u_{j,\e}-u_j|_\infty\geq 1\}.
$$
Hence, we obtain with help of \ref{(A4)} and $s$ and $j$ as above that
\begin{align*}
\int_{A\cap \{|u_{j,\e}|_\infty\geq s\}} W(\omega,\tfrac{x}\e,\nabla T_s(u_{j,\e})))\dx&\leq C\int_{A\cap \{|u_{j,\e}|_\infty\geq s\}}(W(\omega,\tfrac{x}\e,\nabla u_{j,\e})+\Lambda(\omega,\tfrac{x}\e))\dx\\
&\leq C\int_{A\cap \{|u_{j,\e}-u_j|_\infty\geq 1\}} (W(\omega,\tfrac{x}\e,\nabla u_{j,\e})+\Lambda(\omega,\tfrac{x}\e))\dx.
\end{align*}
We claim that
\begin{align}\label{pf:rec:IIb}
\limsup_{\e\downarrow0}\int_{A\cap \{|u_{j,\e}-u_j|_\infty\geq 1\}} (W(\omega,\tfrac{x}\e,\nabla u_{j,\e})+\Lambda(\omega,\tfrac{x}\e))\dx= 0,
\end{align}
which together with \eqref{pf:rec:IIa} yields \eqref{P:limsup:II}.
In order to verify \eqref{pf:rec:IIb}, recall that $\nabla u_{j,\e}=\xi_j+\nabla \phi_{\xi_j}(\tfrac{\cdot}\e)$, so that the ergodic theorem in the form of Lemma~\ref{l.weakL1} implies that $W(\omega,\tfrac{\cdot}\e,\nabla u_{j,\e})\rightharpoonup W_{\rm hom}(\xi_j)$ and $\Lambda(\omega,\tfrac{\cdot}\e)\rightharpoonup \mathbb E[\overbar\Lambda]$ weakly in $L^1(A)$ as $\e\downarrow0$. In particular, due to the Dunford-Pettis theorem (see e.g.\ \cite[Theorem 2.54]{FoLe}) the sequences $(W(\omega,\tfrac{\cdot}\e,\nabla u_{j,\e}))_\e$ and $(\Lambda(\omega,\tfrac{\cdot}\e))_\e$ are equi-integrable and thus \eqref{pf:rec:IIb} follows from the convergence $u_{j,\e} \to u_j$ in $L^1(A)^m$ which implies that $|A\cap\{|(u_{j,\e}-u_j|_\infty\geq1\}|\to0$. 

It remains to show
\begin{equation}\label{P:limsup:III}
\limsup_{t\uparrow1}\limsup_{\delta\downarrow0}\limsup_{s\uparrow\infty}\limsup_{j\uparrow\infty}\limsup_{\e\downarrow0}(1-t)\int_{A} W(\omega,\tfrac{x}\e,\frac{t}{1-t}\nabla \eta\otimes (T_s(u_{j,\e})-u))\dx\leq0.
\end{equation}
Using once more the convexity of $W$, we can bound the integral by
\begin{align}\label{eq:insert_u_j}
	\int_{A} W(\omega,\tfrac{x}\e,\frac{t}{1-t}\nabla \eta\otimes (T_s(u_{j,\e})-u))\dx&\leq \int_{A} W(\omega,\tfrac{x}\e,\frac{2t}{1-t}\nabla \eta\otimes (T_s(u_{j,\e})-u_j))\dx\nonumber
	\\
	&\quad +\int_{A} W(\omega,\tfrac{x}\e,\frac{2t}{1-t}\nabla \eta\otimes (u_j-u))\dx.
\end{align}
For the last term, recall that $u_j\to u$ in $L^{\infty}(A)^m$. Hence, given $s,\delta>0$ and $t\in [0,1)$, for $j$ large enough we have 
\begin{equation*}
	\left|\frac{2t}{1-t}\nabla \eta\otimes (u_j-u)\right|\leq 1\quad\text{ on }A,
\end{equation*}
so that for such $j$ 
\begin{equation*}
	0\leq W(\omega,\tfrac{x}\e,\frac{2t}{1-t}\nabla \eta\otimes (u_j-u))\leq \sup_{|\zeta|\leq 1}W(\w,\tfrac{x}{\e},\zeta).
\end{equation*}
Due to Assumption \ref{(A3)} we can apply the ergodic theorem to the function on the right-hand side and deduce that
\begin{equation}\label{eq:u-u_j}
	\limsup_{t\uparrow 1}\limsup_{\delta\downarrow 0}\limsup_{s\uparrow\infty}\limsup_{j\uparrow\infty}\limsup_{\e\downarrow0}(1-t)\int_{A} W(\omega,\tfrac{x}\e,\frac{2t}{1-t}\nabla \eta\otimes (u_j-u))\dx\leq 0.
\end{equation}
To treat the other right-hand side term in \eqref{eq:insert_u_j}, first recall that $u_{j,\e}\to u_j$ in $L^1(A)^m$. Using Egorov's theorem, for every sequence $\e\to 0$  we find a subsequence such that for any $\eta>0$ there exists a measurable set $A_{\eta}$ with measure less than $\eta$ and such that $u_{j,\e}\to u_j$ uniformly on $A\setminus A_{\eta}$. Consider $j$ large enough such that $\|u_j-u\|_{L^{\infty}(A)}<1$ and $s\geq\|u\|_{L^{\infty}(A)}+2$. Then $T_su_j=u_j$ a.e. on $A$. By construction, for any such $s$, and $\delta>0$ and $t\in[0,1)$ there exist $r>0$ such that
$$
\sup_{\e>0}\left|\frac{t}{1-t}\nabla \eta\otimes (T_s(u_{j,\e})-u_j)\right|\leq r.$$
Since $T_s(u_{j,\e})$ also converges uniformly to $T_s(u_j)=u_j$ on $A\setminus A_{\eta}$, for $\e$ small enough we can bound the integrand by
\begin{equation*}
	0\leq W(\omega,\tfrac{x}\e,\frac{t}{1-t}\nabla \eta\otimes (T_s(u_{j,\e})-u_j))\dx\leq \mathds{1}_{A_{\eta}}\sup_{|\zeta|\leq r}W(\w,\tfrac{x}{\e},\zeta)+\sup_{|\zeta|\leq 1}W(\w,\tfrac{x}{\e},\zeta).
\end{equation*}
The second right-hand side term can be treated as before, while for the first one the equi-integrability of $x\mapsto\sup_{|\zeta|\leq r}W(\w,\tfrac{x}{\e},\zeta)$ (cf. Assumption \ref{(A3)} and Lemma \ref{l.weakL1}) implies that
\begin{equation*}
	\lim_{\eta\to 0}\limsup_{\e\to 0}\int_{A_{\eta}}\sup_{|\zeta|\leq r}W(\w,\tfrac{x}{\e},\zeta)\dx=0.
\end{equation*}
Since the limit does not depend on the subsequence we picked for Egorov's theorem, we proved that 
\begin{equation*}
	\limsup_{t\uparrow 1}\limsup_{\delta\downarrow 0}\limsup_{s\uparrow\infty}\limsup_{j\uparrow\infty}\limsup_{\e\downarrow0}(1-t)\int_{A} W(\omega,\tfrac{x}\e,\frac{2t}{1-t}\nabla \eta\otimes (T_s(u_{j,\e}-u_j))\dx\leq 0,
\end{equation*}
so that in combination with \eqref{eq:u-u_j} and \eqref{eq:insert_u_j} we indeed obtain \eqref{P:limsup:III}.

Combining \eqref{P:limsup:I}, \eqref{P:limsup:II} and \eqref{P:limsup:III} with the definition of $v_{\e,\delta,j,s}$, we obtain
\begin{align*}
\limsup_{t\uparrow1}\limsup_{\delta\downarrow0}\limsup_{s\uparrow\infty}\limsup_{j\uparrow\infty}\limsup_{\e\downarrow0}\biggl(\|tv_{\e,\delta,j,s}-u\|_{L^1(A)}+F_\e(\omega,tv_{\e,\delta,j,s})-\int_A W_{\rm hom}(\nabla u)\,dx\biggr)\leq 0.
\end{align*}
Passing to a suitable diagonal sequence and using the $\liminf$-inequality of Proposition~\ref{p.lb} on the bounded, open set $A$, we obtain the desired recovery sequence satisfying \eqref{eq:seqrec}.

\smallskip

{\bf Step 3.} Recovery sequence with prescribed boundary values for affine functions -- Assumption \ref{(A5)}.

We show that for bounded, open set $A\subset\R^d$, every affine function $u$ with $\nabla u=\xi\in\R^{m\times d}$ there exists a sequence $(v_\e)_\e$ satisfying 
\begin{equation}\label{eq:seqrecp}
(v_\e)_\e\subset u+W_0^{1,p}(A)^m,\quad \lim_{\e\downarrow0}\biggl(\| v_\e-u\|_{L^p(A)}+\biggl|\int_A W(\omega,\tfrac{x}\e,\nabla v_\e)-W_{\rm hom}(\nabla u)\dx\biggr|\biggr)=0.
\end{equation}

We first consider the case when $A$ is a ball. Let $B=B_R(x_0)$ with $R>0$, $x_0\in\R^d$,consider a sequence $u_j$ of affine functions satisfying $\nabla u_j=\xi_j\in \mathbb Q^{m\times d}$ and $u_j\to u$ in $W^{1,\infty}(B)^m$. For $\delta\in(0,\frac12)$ and $j\in\mathbb N$, let $\eta=\eta_{\e,\delta,j}$ be as in Lemma~\ref{L:optim} with $N=1$ and $u_1=\e \phi_{\xi_j}(\frac\cdot\e)$. We set
\begin{align*}
v_{\e,\delta,j}:=\eta_{\e,\delta,j} (u_j+\e\phi_{\xi_j}(\tfrac{\cdot}\e))+(1-\eta_{\e,\delta,j})u.
\end{align*}
Clearly, we have
\begin{equation}\label{pf:limsupb:conv}
\limsup_{t\uparrow1}\limsup_{\delta\downarrow0}\limsup_{j\uparrow \infty}\limsup_{\e\downarrow0}\|tv_{\e,\delta,j}-u\|_{L^p(B)}=0
\end{equation}
and by convexity, we have 
\begin{align*}
F_\e(\omega,tv_{\e,\delta,j},B)\leq& \int_{B}t(1-\eta_{\e,\delta,j})W(\omega,\tfrac{x}\e,\xi)\dx+t\int_{B}\eta_{\e,\delta,j} W(\omega,\tfrac{x}\e,\xi_j+\nabla \phi_{\xi_j}(\tfrac{x}\e))\dx\\
&+(1-t)\int_{B} W(\omega,\tfrac{x}\e,\frac{t}{1-t}\nabla \eta_{\e,\delta,j}\otimes (\e\phi_{\xi_j}(\tfrac{x}\e)+u_j-u))\dx.
\end{align*}
Similarly to Step~2, we obtain
\begin{equation}\label{pf:limsupb:ii0}
\limsup_{t\uparrow1}\limsup_{\delta\downarrow0}\limsup_{j\uparrow\infty}\limsup_{\e\downarrow0}\int_{B}t(1-\eta_{\e,\delta,j})W(\omega,\tfrac{x}\e,\xi)=0
\end{equation}
(see \eqref{P:limsup:I}) and 
\begin{equation}\label{pf:limsupb:ii}
\limsup_{t\uparrow1}\limsup_{\delta\downarrow0}\limsup_{j\uparrow\infty}\limsup_{\e\downarrow0}\,t\int_{B}\eta_{\e,\delta,j} W(\omega,\tfrac{x}\e,\xi_j+\nabla \phi_{\xi_j}(\tfrac{x}\e))\dx\leq |B|W_{\rm hom}(\xi)=\int_B W_{\rm hom}(\nabla u)\dx.
\end{equation}
Hence, it remains to estimate the last term in the above bound. By Lemma~\ref{L:optim}, we find for every $\rho>0$ a constant $C_\rho<\infty$ depending only on $d,m,p$ and $\rho>0$ such that
\begin{align*}
\|\nabla \eta_{\e,\delta,j}\otimes \e\phi_{\xi_j}(\tfrac{\cdot}\e)\|_{L^\infty(B)}\leq&\frac{\rho}{\delta}\biggl(\frac1{\delta R^d}\int_{B}|\nabla \phi_{\xi_j}(\tfrac{x}\e)|^p\dx\biggr)^\frac1p+\frac{C_{\rho}}{\delta R}\biggl(\frac1{\delta R^d}\int_{B}|\e\phi_{\xi_j}(\tfrac{x}\e)|^p\dx\biggr)^\frac1p.
\end{align*}
By Lemma \ref{l.defW_hom} the last integral vanishes as $\e\to 0$, while the first right-hand side integral remains bounded as $\e\to 0$. From the arbitrariness of $\rho>0$ we thus infer that
\begin{align*}
\limsup_{\e\downarrow 0}\|\nabla \eta_{\e,\delta,j}\otimes \e\phi_{\xi_j}(\tfrac{\cdot}\e)\|_{L^\infty(B)}=0,
\end{align*}
which implies with help of the triangle inequality and \eqref{L:optim:eta} that for all $t\in[0,1)$, $\delta\in(0,\frac12]$ and $j\in\mathbb N$ it holds
\begin{align*}
\limsup_{\e\downarrow0}\biggl\|\frac{t}{1-t}\nabla \eta\otimes (\e\phi_{\xi_j}(\tfrac{\cdot}\e)+u_j-u)\biggr\|_{L^\infty(B)}\leq& \frac{2}{1-t}\frac1{\delta R}\|u_j-u\|_{L^\infty(B)}.
\end{align*}
In particular, for all $t\in(0,1)$, $\delta\in(0,\frac12]$ and $j$ sufficiently large (depending on $t$ and $\delta$) we have
\begin{align*}
	\limsup_{\e\to 0}\int_{B} W(\omega,\tfrac{x}\e,\frac{t}{1-t}\nabla \eta_{\e,\delta,j}\otimes (\e\phi_{\xi_j}(\tfrac{x}\e)+u_j-u))\dx&\leq \limsup_{\e\to 0}\int_B \sup_{|\zeta|\leq 1}W(\w,\tfrac{x}{\e},\zeta)\dx
	\\
	&=|D|\,\mathbb{E}[\sup_{|\zeta|\leq 1}\overbar W(\cdot,\zeta)],
\end{align*}
where we used Assumption \ref{(A3)}  and the ergodic theorem in the form of Lemma \ref{l.weakL1}. Thus
%
\begin{align}\label{pf:limsupb:iii3}
\limsup_{t\uparrow1}\limsup_{\delta\downarrow0}\limsup_{j\uparrow \infty}\limsup_{\e\downarrow0}(1-t)\int_BW(\omega,\tfrac{x}\e,\frac{t}{1-t}\nabla \eta_{\e,\delta,j}\otimes (\e\phi_{\xi_j}(\tfrac{x}\e)+u_j-u))\dx\leq 0.
\end{align}
Combining \eqref{pf:limsupb:conv}, \eqref{pf:limsupb:ii0}, \eqref{pf:limsupb:ii} and \eqref{pf:limsupb:iii3}, we obtain a diagonal sequence satisfying \eqref{eq:seqrecp} in the case $A=B$. 

Next, we remove the restriction on $A$ being a ball and consider a general bounded, open set $A\subset\R^d$ and an affine function $u$ with $\nabla u=\xi\in\R^{m\times d}$. By the Vitali covering theorem, we find a collection of disjoint balls $B^{(j)}\subset A$, $j\in\mathbb N$ such that $|A\setminus \cup_{i\in\mathbb N} B^{(i)}|=0$. Hence, for every $\nu>0$ there exists $N\in\mathbb N$ such that $|A\setminus (\cup_{i=1}^N B^{(i)})|<\nu$. The previous result for balls ensures that for all $i\in\{1,\dots,N\}$, we find a sequence $(v_\e^{(i)})_\e\in u+W_0^{1,p}(B^{(i)})^m$ such that
$$
\lim_{\e\downarrow}\|v_\e^{(i)}-u\|_{L^p(B^{(i)})}=0\quad\mbox{and}\quad \lim_{\e\downarrow0}F(\omega,v^{(i)}_\e,B^{(i)})=\int_{B^{(i)}}W_{\rm hom}(\nabla u)\dx.
$$
Setting
$$
v_\e:=u+\sum_{i=1}^N (v_\e^{(i)}-u),
$$
we  have that $v_\e\in u+W_0^{1,p}(A)^m$ and $v_\e\to u$ in $L^p(A)^m$. Moreover, by the ergodic theorem it holds that
\begin{align*}
\limsup_{\e\downarrow0}F(\omega,u_\e,A)\leq& \limsup_{\e\downarrow0}F(\omega,u_\e,A\setminus \cup_{i=1}^NB^{(i)})+\sum_{i=1}^N\limsup_{\e\downarrow0}F(\omega,u_\e,B^{(i)})\\
=& \limsup_{\e\downarrow0}\int_{A\setminus \cup_{i=1}^NB^{(i)}}W(\omega,\tfrac{x}\e,\xi)\dx+\sum_{i=1}^N\int_{B^{(i)}} W_{\rm hom}(\nabla u)\dx\\
\leq& \mathbb{E}[\overbar W(\cdot,\xi)]\,|A\setminus\cup_{i=1}^NB^{(i)}| +\int_AW_{\rm hom}(\nabla u)\dx,
\end{align*}
where we used in the last step $W_{\rm hom}\geq0$. Letting $N\uparrow +\infty$, we conclude the proof.

\smallskip

{\bf Step 4.} The general case. 

In this step, we pass from the limsup-inequality for affine functions to the limsup-inequality for general functions $u\in W^{1,1}(D)^m$. Let us first consider $u\in W^{1,\infty}(D)^m$ which is piecewise affine, i.e., there exist finitely many disjoint open sets $A_i\subset D$, $i=1,\dots,K$ such that $|D\setminus \cup_i A_i|=0$ and $ u|_{A_i}(x)=\xi_ix+b_i$ for some $\xi_i\in\R^{m\times d}$ and $b_i\in\R^m$. In view of Step~2 and Step~3, we find for every $A_i$ a sequence $(u_{i,\e})_\e\in u+W_0^{1,1}(A_i)^m$ satisfying either \eqref{eq:seqrec} or \eqref{eq:seqrecp}. Since all $u_{i,\e}$, $i=1,\dots,K$ coincide with $u$ at the boundary of $A_i$, we can glue them together and obtain a recovery sequence for $u$ on $D$.

\smallskip

The limsup-inequality for general $u\in W^{1,1}(D)^m$ follows as in \cite{DG_unbounded,Mue87}. By the 'locality of the recovery sequence' (see \cite[Corollary 3.3]{DG_unbounded}), we can assume that $D\subset\R^n$ is an open ball. Indeed, for a bounded, open set with Lipschitz boundary $D$ we find a ball $B$ such that $D\subset\subset B$. Assume that we have a recovery sequence $(u_\e)_\e$ satisfying $u_\e\to u$ in $L^1(B)^m$ and $F_\e(\omega,u_\e,B)\to \int_B W_{\rm hom}(\nabla u)\dx$ as $\e\downarrow0$. Then,
\begin{align*}
\lim_{\e\downarrow0}F_\e(\omega,u_\e,D)=&\lim_{\e\downarrow0}(F_\e(\omega,u_\e,B)-F_\e(\omega,u_\e,D\setminus B))
\leq \lim_{\e\downarrow0}F_\e(\omega,u_\e,B)-\liminf_{\e\downarrow0}F_\e(\omega,u_\e,D\setminus B)\\
\leq& \int_BW_{\rm hom}(\nabla u)\dx-\int_{B\setminus D}W_{\rm hom}(\nabla u)\dx=\int_D W_{\rm hom}(\nabla u)\dx,
\end{align*}
where we used in the second inequality that $u_\e$ is a recovery sequence on $B$ and the $\liminf$ inequality Proposition~\ref{p.lb}

Hence, it suffices to assume that $D\subset\R^n$ is an open ball and thus smooth and star-shaped. Fix $u\in W^{1,1}(D)^m$ such that $\int_D W_{\rm hom}(\nabla u)\dx<+\infty$. Since $W_{\rm hom}:\R^{m\times n}\to[0,+\infty)$ is convex (cf. Lemma \ref{l.defW_hom}), we can apply \cite[Lemma 3.6]{Mue87} and find a sequence $(u_j)_j$ of piecewise affine functions satisfying $u_j\to u$ in $W^{1,1}(D)^m$ and $\int_D W_{\rm hom}(\nabla u_j)\dx\to \int_D W_{\rm hom}(\nabla u)\dx$ as $j\to\infty$. Hence, by the previous argument, we find $(u_{j,\e})_\e$ with $u_{j,\e}\to u_j$ in $L^1(D)^m$ and $F_\e(\omega,u_{j,\e},D)\to\int_D W_{\rm hom}(\nabla u_j)\dx$ and thus
$$
\lim_{j\to\infty}\lim_{\e\downarrow0}\biggl(\| u_{j,\e}-u\|_{L^1(D)}+\biggl|\int_D W(\omega,\tfrac{x}\e,\nabla u_{\e,j})-W_{\rm hom}(\nabla u)\,dx\biggr|\biggr)=0.
$$
The claim follows again by a diagonal argument. 
\end{proof}

\subsection{Convergence with boundary value problems and external forces}
In this section we prove the $\Gamma$-convergence result under Dirichlet boundary conditions and with external forces, i.e., Theorem \ref{thm:constrainedGamma}.
\begin{proof}[Proof of Theorem \ref{thm:constrainedGamma}]
In a first step we consider the case $f_{\e}=f=0$. Weak $W^{1,1}$-compactness  (or $W^{1,p}$-compactness  assuming \ref{(A5)}) of energy bounded sequences follows from the unconstrained case (cf. Lemma \ref{l.compactness}) and the fact that the Dirichlet boundary condition allows us to apply a Poincar\'e inequality to deduce $L^1$-boundedness. The $\Gamma$-liminf inequality is also a consequence of the unconstrained case (cf. Lemma \ref{p.lb}) since the boundary condition is stable under weak convergence in $W^{1,1}(D)^m$. Hence it remains to show the $\Gamma$-limsup inequality, which requires more work. Fix $u\in W^{1,1}(D)^m$ such that $u=g$ on $\partial D$ in the sense of traces and $F_{\rm hom}(u)<+\infty$. 

Let $t\in (0,1)$. Then $u_t:=g+t(u-g)\in g+W^{1,1}_0(D)^m$ and by convexity we have 
\begin{align*}
	F_{\rm hom}\left(\frac{(1+t)}{2t}(u_t-g)\right)&=F_{\rm hom}\left(\frac{(1+t)}{2}(u-g)\right)=F_{\rm hom}\left(\frac{(1+t)}{2}u+\frac{(1-t)}{2}\frac{(1+t)}{(t-1)}g\right)
	\\
	&\leq \frac{(1+t)}{2}F_{\rm hom}(u)+\frac{(1-t)}{2}F_{\rm hom}\left(\frac{(1+t)}{(t-1)}g\right)<+\infty,
\end{align*}
and $F_{\rm hom}(u_t)\leq (1-t)F_{\rm hom}(g)+tF_{\rm hom}(u)<+\infty$. In particular, 
\begin{equation*}
	\lim_{t\uparrow 1}F_{\rm hom}(u_t)\leq F_{\rm hom}(u)
\end{equation*}
and since $u_t\to u$ in $L^1(D)^m$ when $t\uparrow 1$ and $\tfrac{1+t}{2t}>1$, a diagonal argument allows us to show the $\Gamma$-limsup inequality for functions $u$ such that additionally $F_{\rm hom}(s(u-g))<+\infty$ for some $s>1$. By Lemma \ref{l.EkTemApprox} and the properties of $W_{\rm hom}$ (cf. Lemma \ref{l.defW_hom}) there exists a sequence $v_{n}\in C_c^{\infty}(D)^m$ such that $v_n\to u-g$ in $W^{1,1}(D)^m$ and
\begin{equation}\label{eq:zerobc}
	\lim_{n\to +\infty}\int_D W_{\rm hom}(s\nabla v_n)\dx=\int_D W_{\rm hom}(s\nabla (u-g))\dx<+\infty.
\end{equation}
By choosing a suitable subsequence (not relabeled) we can assume in the following that $(v_n)_n$ satisfies in addition $v_n\to u-g$ and $\nabla v_n\to \nabla v-\nabla g$ a.e. in $D$. Next, we show 
$$
\lim_{n\to +\infty}\int_D W_{\rm hom}(\nabla (v_n+g))\dx=\int_D W_{\rm hom}(\nabla u)\dx.
$$
Indeed, by Fatous lemma and the non-negativity of and $W_{\rm hom}$, we have
$$\liminf_{n\to\infty}\int_D W_{\rm hom}(\nabla (v_n+g))\dx\geq\int_D W_{\rm hom}(\nabla u)\dx.$$
To show the corresponding inequality for the $\limsup$, we first observe that for all $\xi\in\R^{m\times d}$
\begin{equation}\label{est:tildeWforfatou}
W_{\rm hom}(\xi+\nabla g)=W_{\rm hom}(\tfrac{1}{s}s\xi+(1-\tfrac{1}{s})\tfrac{s}{s-1}\nabla g(x))\leq W_{\rm hom}(s\xi)+W_{\rm hom}(\tfrac{s}{s-1}\nabla g(x)).
\end{equation}
Hence, the desired inequality follows with help of estimate \eqref{est:tildeWforfatou}, Fatous lemma and \eqref{eq:zerobc}:
\begin{align*}
\limsup_{n\to\infty}\int_D W_{\rm hom}(\nabla (v_n+g))\dx\leq&-\liminf_{n\to\infty}\int_D W_{\rm hom}(s\nabla v_n)+W_{\rm hom}(\tfrac{s}{s-1}\nabla g(x))-W_{\rm hom}(\nabla (v_n+g))\dx\\
&+\lim_{n\to\infty}\int_D W_{\rm hom}(s\nabla v_n)+W_{\rm hom}(\tfrac{s}{s-1}\nabla g(x))\dx\\
\leq&\int_D W_{\rm hom}(\nabla u)\dx.
\end{align*}
Since $g$ is Lipschitz-continuous, we thus deduce that there exists a sequence $u_n\in {\rm Lip}(D)^m$ such that $u_n=g$ on $\partial D$ with $u_n\to u$ in $W^{1,1}(D)^m$ and
\begin{equation}\label{eq:energyconv}
\lim_{n\to +\infty}\int_DW_{\rm hom}(\nabla u_n)\dx=\int_D W_{\rm hom}(\nabla u)\dx.
\end{equation}
Thus, by a further diagonal argument, it suffices to show the upper bound for Lipschitz-functions $u\in {\rm Lip}(D)^m$ with $u=g$ on $\partial D$. Applying \cite[Proposition 2.9, Chapter X]{EkTe} to each component of $u$, we find a sequence $u_n\in {\rm Lip}(D)^m$ and an increasing sequence of open sets $D_n\subset D$ such that $|D_\setminus D_n|\to 0$, the function $u_n$ is piecewise affine on $D_n$, $u_n=g$ on $\partial D$, $u_n\to u$ uniformly on $D$, $\nabla u_n\to\nabla u$ a.e. on $D$, and
\begin{equation}\label{eq:Lipbound}
	\|\nabla u_n\|_{L^{\infty}(D)}\leq \| \nabla u\|_{L^{\infty}(D)}+ o(1).
\end{equation}
Then $u_n\to u$ in $L^1(D)^m$ and the dominated convergence theorem and the continuity of $W_{\rm hom}$ imply that
\begin{equation}\label{eq:reducetoEkTemApprox}
	\lim_{n\to +\infty}W_{\rm hom}(\nabla u_n)\dx=\int_DW_{\rm hom}(\nabla u)\dx.
\end{equation}
Since $u_n$ is piecewise affine in $D_n$, there exist disjoint open sets $D_n^1,\dots,D_n^{N_n}$ such that $D_n=\cup_{j} D_n^j$ and $\nabla u=\xi_n^j$ on $D_n^j$. By Step~2 and Step~3 of the proof of Proposition \ref{p.recoverysequence}, we find for every $j\in\{1,\dots,N_n\}$ a recovery sequence $u_{\e,n}^j\in u_n+W_0^{1,1}(D_n^j)^m$ satisfying
$$
u_{\e,n}^j\to u_n\quad\mbox{in $L^1(D_n^j)^m$ and}\quad \lim_{\e\downarrow0}F_\e (\omega,u_{\e,n}^j,D_n^j)=\int_{D_n^j}W_{\rm hom}(\nabla u_n)\dx.
$$
Since $u_{\e,n}^j\in u_n+W_0^{1,1}(D_n^j)^m$, we have that $u_{\e,n}:D\to\R^m$ defined by
$$
u_{\e,n}(x)=\begin{cases}u_{\e,n}^j(x)&\mbox{if $x\in D_n^j$}\\u_n(x)&\mbox{if $x\in D\setminus D_n$}\end{cases}\quad\mbox{satisfies}\quad u_{\e,n}\in W^{1,1}(D)^m,\quad u_{\e,n}=g\text{ on }\partial D.
$$
%
Finally, $u^j_{\e,n}$ being a recovery sequence on $D^j_n$, it holds that $u_{\e,n}\to u_n$ in $L^1(D)^m$ as $\e\to 0$ and
\begin{align}\label{eq:almostrecovery}
\limsup_{\e\to 0}F_{\e}(\w,u_{\e,n},D)&= \sum_{j=1}^{N_n}\lim_{\e\to 0}F_{\e}(\w,u_{\e,n}^j,D_n^j)+\limsup_{\e\to 0}F_{\e}(\w,u_n,D\setminus D_n)\nonumber
\\
&=\int_{D_n}W_{\rm hom}(\nabla u_n)\dx+\limsup_{\e\to 0}F_{\e}(\w,u_n,D\setminus D_n).
\end{align}
We estimate the last term as follows: the function $\|\nabla u_n\|_{L^{\infty}(D)}$ is uniformly bounded by \eqref{eq:Lipbound}, so that for some $r>0$ Assumption \ref{(A3)} and the ergodic theorem in the form of Lemma \ref{l.weakL1} imply that
\begin{equation*}
\limsup_{\e\to 0}F_{\e}(\w,u_n,D\setminus D_n)\leq\lim_{\e\to 0}\int_{D\setminus D_n}\sup_{|\zeta|\leq r}W(\w,\tfrac{x}{\e},\zeta)\dx=|D\setminus D_n|\, \mathbb{E}\left[\sup_{|\zeta|\leq r}W(\cdot,0,\zeta)\right].
\end{equation*}
Inserting this bound into \eqref{eq:almostrecovery}, we infer from \eqref{eq:energyconv} that
\begin{equation*}
\lim_{n\to +\infty}\limsup_{\e\to 0}F_{\e}(\w,u_{\e,n},D)\leq \int_D W_{\rm hom}(\nabla u)\dx.
\end{equation*}
Since $u_n\to u$ in $L^1(D)^m$, using another diagonal argument we conclude the proof without external forces.

Next, we consider the case of non-trivial forcing terms $f_\e$ and $f$. Here, we only prove the case that $f_{\e},f\in L^d(D)^m$ are such that $f_{\e}\rightharpoonup f$ in $L^d(D)^m$, the refined results if $W$ satisfies \ref{(A5)} are simpler and left to the reader. We first show relative compactness of energy bounded sequence. Due to H\"older's inequality, the Sobolev embedding in $W^{1,1}(D)^m$, and Poincar\'e's inequality in $W^{1,1}_0(D)^m$, for any admissible $u$ we deduce that 
\begin{align*}
	\int_D |f_{\e}\cdot u|\dx&\leq \|f_{\e}\|_{L^{d}(D)}\|u\|_{L^{d/(d-1)}(D)}\leq C_0 \|f_{\e}\|_{L^d(D)}\|u\|_{W^{1,1}(D)}\\
	&\leq C_0\|f_{\e}\|_{L^d(D)}(\|\nabla u-\nabla g\|_{L^1(D)}+\|g\|_{W^{1,1}(D)})
	\\
	&\leq C_0\|f_{\e}\|_{L^d(D)} \|\nabla u\|_{L^1(D)}+C_0\|f_{\e}\|_{L^d(D)}\|g\|_{W^{1,1}(D)},
\end{align*}
which combined with \eqref{eq:quantitativelineargrowth} for $C=C_0 \sup_{\e}\|f_{\e}\|_{L^d(D)}$ and Remark \ref{r.stationary_extension} shows that
\begin{align}\label{eq:compactnesswithforces}
	F_{\e}(\w,u,D)-\int_Df_{\e}\cdot u\dx&\geq \frac{1}{2}F_{\e}(\w,u,D)+C\int_D|\nabla u|\dx-\underbrace{\int_D \sup_{|\eta|\leq C}W^*(\w,\tfrac{x}{\e},\eta)\dx}_{=:a_{\e}(\w)}-\int_D |f_{\e}\cdot u|\dx\nonumber
	\\
	&\geq \frac{1}{2}F_{\e}(\w,u,D)-a_{\e}(\w)-C\|g\|_{W^{1,1}(D)}.
\end{align}
Due to the ergodic theorem the sequence $a_{\e}(\w)$ is bounded when $\e\to 0$. Hence boundedness of $F_{\e,f_{\e},g}(\w,u_{\e},D)$ implies that also $F_{\e}(\w,u_{\e},D)$ is bounded, so the weak relative compactness of $u_{\e}$ in $W^{1,1}(D)^m$ is a consequence of the case $f_{\e}=0$. Moreover, as shown in \cite[Theorem B.1]{RR21}, the weak convergence in $W^{1,1}(D)^m$ implies the strong convergence in $L^{d/(d-1)}(D)^m$. Hence along any sequence with equibounded energy and with $u_{\e}\to u$ in $L^1(D)^m$, the term $\int_D  f_{\e}\cdot u_{\e}\dx$ converges to $\int_D  f\cdot u\dx$. Thus also the $\Gamma$-convergence is a consequence of the case $f_{\e}=0$.
\end{proof}

\subsection{Differentiability of the homogenized integrand}\label{subs:diff}
In this section we prove that if $W$ satisfies the stronger Assumption \ref{a.2}, then $W_{\rm hom}$ is continuously differentiable. To this end, we rely on convex analysis on generalized Orlicz spaces (cf. Section \ref{s.orlicz} and Remark \ref{r.Orliczspaces}). Given $\overbar\varphi$ defined by \eqref{eq:defphi}, let $L^{\overbar\varphi}(\Omega)^{m\times d}$ be the associated generalized Orlicz space. Below we prove the announced properties of $W_{\rm hom}$.
\begin{proposition}\label{p.onWhom}
Let $W$ satisfy Assumption \ref{a.2}. Then $W_{\rm hom}$ is continuously differentiable with derivative
\begin{equation*}
	\partial_{\xi}W_{\rm hom}(\xi)=\mathbb{E}[\partial_{\xi}W(\cdot,0,\xi+\overbar h_{{\xi}})],
\end{equation*}	
where $\overbar h_{\xi}$ is given by Lemma \ref{l.defW_hom}.
\end{proposition}
\begin{proof}
It suffices to show that $W_{\rm hom}$ is differentiable with the claimed derivative. The continuity of the derivative follows from general convex analysis \cite[Theorem 4.65]{FoLe}. Recall that
\begin{equation}\label{eq:p.onWhom:0}
	W_{\rm hom}(\xi)=\min_{\overbar h\in (F_{\rm pot}^1)^m}G(\xi+\overbar h)\qquad\mbox{where}\qquad G(z):=\mathbb{E}[\overbar W(\cdot,z)].
\end{equation}
We already know that $W_{\rm hom}$ is a real-valued, convex function on $\R^{m\times d}$. In particular, its subdifferential is always non-empty and $W_{\rm hom}$ is differentiable at $\xi\in\R^{m\times d}$ if and only if the subdifferential at $\xi$ contains exactly one element. We aim to express the subdifferential via a suitable chain rule. To this end, define the set-valued mapping $F:\R^{m\times d}\rightrightarrows L^{\overbar\varphi}(\Omega)^{m\times d}$ by
\begin{equation*}
	F(\xi)=\Big(\xi+(F_{\rm pot}^1)^m\Big)\cap {\rm dom}(G) .
\end{equation*}
Due to the estimate $\overbar\varphi(\w,\cdot)\leq C\overbar W(\w,\cdot)+\overbar\Lambda(\w)$ (cf. \eqref{eq:varphi_and_W}), it follows indeed that ${\rm dom}(G)\subset L^{\overbar\varphi}(\Omega)^{m\times d}$. Since $G$ is convex, the graph of $F$ defined by ${\rm gph}(F)=\{(\xi,\overbar h)\in\R^{m\times d}\times L^{\overbar\varphi}(\Omega)^{m\times d} ,\,\overbar h\in F(\xi)\}$ is a convex subset of $\R^{m\times d}\times L^{\overbar\varphi}(\Omega)^{m\times d}$. Clearly, we can rewrite the first identity in \eqref{eq:p.onWhom:0} by
\begin{equation}\label{eq:p.onWhom:1}
W_{\rm hom}(\xi)=\min\{G(y):\,y\in F(\xi)\}.
\end{equation}
We aim to use the representation result for subdifferentials of optimal-value functions as in \eqref{eq:p.onWhom:1} given in \cite[Corollary 7.3]{MNRT}. For this it is left to check that $G$ is finite and continuous at a point in $F(\xi)$ with respect to convergence in $L^{\overbar\varphi}(\Omega)^{m\times d}$. We choose the constant function $\overbar\xi\in F(\xi)$ and let $\overbar h_n\to 0$ in $L^{\overbar\varphi}(\Omega)^{m\times d}$. Set $t_n =1-\|\overbar h_n\|_{{\overbar\varphi}}$. Then the convexity of $W$ and of $\overbar\varphi$ together with \eqref{eq:varphi_and_W} and $\overbar \varphi(\cdot,0)=0$ yield that for $n$ large enough such that $t_n\in(0,1)$ 
\begin{align*}
	G(\overbar\xi+\overbar h_n)&=\mathbb{E}[\overbar W(\cdot,\xi+\overbar h_n)]\leq t_n\mathbb{E}[\overbar W(\cdot,\xi/t_n)]+(1-t_n)\mathbb{E}[\overbar W(\cdot,\overbar h_n/(1-t_n))]
	\\
	&\leq \mathbb{E}[\overbar W(\cdot,\xi/ t_n)]+\|\overbar h_n\|_{{\overbar\varphi}}\underbrace{\mathbb{E}[\overbar\varphi(\cdot,\overbar h_n/\|\overbar h_n\|_{\overbar\varphi})]}_{\leq 1}+\|\overbar h_n\|_{{\overbar\varphi}}\mathbb{E}[\overbar\Lambda]
	\\
	&\leq \mathbb{E}[\overbar W(\cdot,\xi/t_n)]+\|\overbar h_n\|_{\overbar\varphi}(1+\mathbb{E}[\Lambda]).
 \end{align*} 
Hence, $\|\overbar h_n\|_{\overbar\varphi}\to0$ as $n\to\infty$ yields
$$
\limsup_{n\to\infty}G(\overbar\xi+\overbar h_n)\leq \limsup_{n\to\infty}\mathbb{E}[\overbar W(\cdot,\xi/t_n)]=G(\overbar\xi).
$$
The reverse inequality follows from Fatous's lemma since $\overbar h_n\to 0$ also in $L^1(\Omega)^{m\times d}$ by the continuous embedding. Now we are in a position to apply \cite[Corollary 7.3]{MNRT} to conclude that
\begin{equation}\label{eq:partialwhom0}
	\partial W_{\rm hom}(\xi)=\bigcup_{\ell\in\partial G(\xi+\overbar h_{\xi})}\Big\{\eta\in \R^{m\times d}:\langle \eta,z-\xi\rangle\leq \ell(\overbar h-(\xi+\overbar h_\xi))\quad\text{ for all }z\in\R^{m\times d},\,\overbar h\in F(z)\Big\},
\end{equation}
where in the above equation $\ell(\overbar h-(\xi+\overbar h_\xi))$ stands for the duality pairing and $\overbar h_\xi\in (F^1_{\rm pot})^m$ is a minimizer in \eqref{eq:p.onWhom:0}.
Our goal is to show that the sets above are all singletons containing the same element, which then concludes the proof that $W_{\rm hom}$ is differentiable with the claimed derivative. For this, we first observe that for every $z\in\R^{m\times d}$ and $t\in(0,1)$, convexity of $\overbar W$ and Assumption \ref{(A3)} imply 
\begin{equation*}
	G(z+t\overbar h_{{\xi}})\leq G\left(t\xi+t\overbar h_{\xi}+(1-t)\frac{-t\xi+z}{1-t}\right)\leq t W_{\rm hom}(\xi)+(1-t)\mathbb{E}\left[W\left(\cdot,\frac{-t\xi+z}{1-t}\right)\right]<+\infty,
\end{equation*}
so that $z+t\overbar h_{\xi} \in {\rm dom}(G)$ and thus $z+t\overbar h_\xi\in F(z)$. Hence, for  $\eta\in \partial W_{\rm hom}(\xi)$ we deduce from \eqref{eq:partialwhom0}
\begin{equation*}
	\langle\eta,z-\xi\rangle\leq \ell(z-\xi+(t-1)\overbar h_{\xi})=\ell(z-\xi)+(t-1)\ell(\overbar h_{\xi})\overset{t\to 1}{\to}\ell(z-\xi).
\end{equation*}
This holds for all $z\in\R^{m\times d}$ and therefore 
\begin{equation*}
	\langle\eta,z\rangle=\ell(z)\quad\text{ for all }z\in\R^{m\times d}.
\end{equation*}
We claim that the above expression does not depend on the element $\ell\in\partial G(\xi+\overbar h_{\xi})$.
To this end, let us recall the expression for the subdifferential of $G$ at $\xi+\overbar h_{\xi}$ in Section \ref{s.orlicz}: for a function $\overbar h\in L^{\overbar\varphi}(\Omega)^{m\times d}$ we have
\begin{align*}
	\partial G(\overbar h)&=\{\ell_a\in L^{\overbar\varphi^*}(\Omega)^{m\times d}:\,\ell_a(\w)=\partial_{\xi}W(\w,\overbar h(\w))\}
	\\ 
	&\quad+\{\ell_s\in S^{\overbar\varphi}(\Omega):\,\ell_s(\overbar v-\overbar h)\leq 0\quad\text{ for all } \overbar v\in {\rm dom}(G)\}.
\end{align*}
In particular, there is only one possibility for the component $\ell_a$. Moreover, as noted in Section \ref{s.orlicz} it holds that $\ell_s(z)=0$ for all constant functions $z\in\R^{m\times d}$ and all $\ell_s\in S^{\overbar\varphi}(\Omega)$. This in turn yields that
\begin{equation*}
	\langle \eta,z\rangle =\ell_a(z)=\mathbb{E}[\partial_{\xi}W(\cdot,\xi+\overbar h_{\xi})z]
\end{equation*}
for all $z\in\R^{m\times d}$. Hence $\partial W_{\rm hom}(\xi)$ is a singleton containing the claimed element and we conclude the proof.
\end{proof}

\subsection{Stochastic homogenization of the Euler-Lagrange equations}
In this last section we provide the arguments for our main result on the Euler-Lagrange equations. The notation relies heavily on the generalized Sobolev-Orlicz spaces introduced in Remark \ref{r.Orliczspaces} (see also Section \ref{s.orlicz}).
\begin{proof}[Proof of Theorem \ref{thm:PDE}]
i): We start showing existence of minimizers for the problem
\begin{equation}\label{eq:epsminproblem}
	\min_{u\in g+ W_0^{1,1}(D)^m}\int_DW(\w,\tfrac{x}{\e},\nabla u(x))-f_{\e}(x)\cdot u(x)\dx.
\end{equation}
Recall that in \eqref{eq:compactnesswithforces}, for all $u\in g+W_0^{1,1}(D)^m$ we proved the estimate
\begin{equation*}
	F_{\e}(\w,u,D)-\int_Df_{\e}\cdot u\dx\geq \frac{1}{2}F_{\e}(\w,u,D)-a_{\e}(\w)-C\|g\|_{W^{1,1}(D)},
\end{equation*}
where $a_{\e}(\w)$ is finite and converging as $\e\to 0$.
This implies the compactness of minimizing sequences since for fixed $\e>0$ the coercivity of $F_{\e}(\w,\cdot,D)$ with respect to weak convergence in $W^{1,1}$ can be proven as in Lemma \ref{l.compactness} and moreover
\begin{equation}\label{eq:global_energybound}
\sup_{0<\e<1}\left(F_{\e}(\w,g,D)-\int_D  f_{\e}\cdot g\rangle\dx\right)<+\infty
\end{equation}
due to the Lipschitz regularity of $g$ and Assumption \ref{(A3)}. Moreover,  $F_{\e}(\w,\cdot,D)$ is weakly lower semicontinuous due to convexity and strong lower semicontinuity, while the term involving $f_{\e}$ is continuous with respect to weak convergence in $W^{1,1}(D)^m$ due to the Sobolev embedding. We thus proved the existence of a minimizer. Next, let us show that minimizers are characterized by
\begin{equation*}
	\int_D \partial_{\xi}W(\w,\tfrac{x}{\e},\nabla u_{\e}(x))\nabla\phi(x)- f_{\e}(x)\cdot\phi(x)\dx\begin{cases}
		=0 &\mbox{if $\phi\in W_0^{1,\infty}(D)^m$,}
		\\
		\geq 0 &\mbox{if $\phi\in W_0^{1,1}(D)^m$ and $F_{\e}(\w,u+\phi,D)<+\infty$,}
	\end{cases}
\end{equation*}
which proves i). Assume first that $u_{\e}$ satisfies the above system. Then, due to convexity of $W$, for any $\phi\in W_0^{1,1}(D)^m$ with $F_{\e}(\w,u+\phi,D)<+\infty$ we have that
\begin{align*}
&\quad F_{\e,f_{\e},g}(\w,u+\phi,D)-F_{\e,f_{\e},g}(\w,u,D)
\\
&\geq F_{\e,f_{\e},g}(\w,u+\phi,D)-F_{\e,f_{\e},g}(\w,u,D)
-\int_D \partial_{\xi}W(\w,\tfrac{x}{\e},\nabla u_{\e}(x))\nabla\phi(x)-f_{\e}(x)\cdot\phi(x)\dx
\\
&\geq \int_D \underbrace{W(\w,\tfrac{x}{\e},\nabla u_{\e}+\nabla\phi)-W(\w,\tfrac{x}{\e},\nabla u_{\e})-\partial_{\xi}W(\w,\tfrac{x}{\e},\nabla u_{\e}(x))\nabla\phi(x)}_{\geq 0}\dx\geq 0,
\end{align*}
which shows minimality since $F_{\e}$ and $F_{\e,f_{\e},g}$ have the same domain on $g+W^{1,1}_0$. 

For the reverse implication, we omit the dependence on $\w$ to reduce notation. Define the two proper convex functionals $G:W^{1,\varphi_{\e}}_0(D)^m\to [0,+\infty]$ and $F:W^{1,\varphi_{\e}}_0(D)^m\to \R$ by
\begin{equation*}
	G(v)=\int_D W(\tfrac{x}{\e},\nabla g+\nabla v)\dx,\quad F(v)=-\int_D f_{\e}\cdot (v+g)\dx.
\end{equation*}
Note that $F$ is real-valued and continuous since $W^{1,\varphi_{\e}}_0(D)^m$ embeds into $W^{1,1}_0(D)^m$. Moreover, $0\in {\rm dom}(G)$, so that we can apply the sum rule for the subdifferential of $G+F$ \cite[Theorem 6.1]{MNRT} to obtain
\begin{equation*}
	\partial (G+F)(v)=\partial G(v)+\partial F(v)=\partial G(v)-f_{\e}\quad\text{ for all }v\in {\rm dom}(G).
\end{equation*}
To find a suitable formula for the subdifferential of $G$, we first write $G=G_0\circ\nabla$, where the gradient $\nabla: W^{1,\varphi_{\e}}_0(D)^m\to L^{\varphi_{\e}}(D)^{m\times d}$ is a bounded linear map, and $G_0:L^{\varphi_{\e}}(D)^{m\times d}\to [0,+\infty]$ is defined by
\begin{equation*} 
	G_0(h)=\int_D W(\tfrac{x}{\e},\nabla g+h)\dx.
\end{equation*}
Arguing as in the proof of Proposition \ref{p.onWhom}, the boundedness of $\nabla g$ implies that $G_0$ is finite and continuous in $0$. Hence the chain rule for subdifferentials implies that $\partial G(u)=\nabla^*\partial G_0(\nabla u)$, where $\nabla^*$ denotes the adjoint operator of the gradient map. In particular, for any $v,\phi\in W^{1,\varphi_{\e}}_0(D)^m$ it holds that $\partial G(v)\phi=\partial G_0(\nabla v)\nabla \phi$, while by the results in Section \ref{s.orlicz} the subdifferential of $G_0$ at a function $h\in L^{\varphi_{\e}}(D)^{m\times d}$ is given by
\begin{align*}
	\partial G_0(h)&=\{\ell_a\in L^{\varphi_{\e}^*}(D)^{m\times d}:\,\ell_a=\partial_{\xi}W(\tfrac{x}{\e},\nabla g(x)+h(x))\}
	\\
	&\quad+\{\ell_s\in S^{\varphi_{\e}}(D)^{m\times d}:\,\ell_s(\phi-h)\leq 0\quad\text{ for all } \phi\in {\rm dom}(G_0)\}.
\end{align*}
The function $u=g+v$ being a solution of the minimization problem \eqref{eq:epsminproblem} is equivalent to the fact that 
\begin{equation*}
	0\in \partial (G+F)(v)=\nabla^*\partial G_0(\nabla v)-f_{\e},
\end{equation*}
which is equivalent to the facts that $\partial_{\xi}W(\tfrac{\cdot}{\e},\nabla u)\in L^{\varphi^*_{\e}}(D)^{m\times d}$ and that there exists $\ell_s\in S^{\varphi_{\e}}(D)^{m\times d}$ with $\ell_s(\cdot-\nabla v)\leq 0$ on ${\rm dom}(G_0)$, which combined satisfy 
\begin{equation*}
	\int_D \partial_{\xi}W(\tfrac{x}{\e},\nabla u)\nabla \phi-f_{\e}\cdot\phi\dx+\ell_s(\nabla \phi)=0\quad\text{ for all }\phi\in W_0^{1,\varphi_{\e}}(D)^m.
\end{equation*} 
If $\phi\in W^{1,\varphi_{\e}}_0(D)^{m\times d}$ is such that $F_{\e}(\w,u+\phi,D)<+\infty$, then $\nabla v+\nabla\phi\in {\rm dom}(G_0)$ and therefore $\ell_s(\nabla\phi)\leq 0$, which then yields
\begin{equation}\label{eq:variationalinequality}
	\int_D \partial_{\xi}W(\tfrac{x}{\e},\nabla u)\nabla \phi-\langle f_{\e},\phi\rangle\dx\geq 0.
\end{equation}
Moreover, as noted in Section \ref{s.orlicz} we have that $\ell_s(\nabla \phi)=0$ whenever $\phi\in W^{1,\infty}_0(D)^m$. This shows that any minimizer satisfies the system in Theorem \ref{thm:PDE} i). 

\vspace*{2.5mm}

ii) The fact that Assumption \ref{a.2} implies $W_{\rm hom}\in C^1$ was shown in Proposition \ref{p.onWhom}

\vspace*{2.5mm}

iii) The proof for the homogenized equation is analogous once one notes that $W_{\rm hom}$ satisfies the same assumptions as $W$ on a deterministic level. Indeed, we already know that $W_{\rm hom}$ is differentiable (Proposition \ref{p.onWhom}), while Assumption \ref{a.2} implies that
\begin{equation*}
	W_{\rm hom}(-\xi)=\mathbb{E}[\overbar W(\cdot,-\xi+\overbar {h}_{-\xi})]\geq\frac{1}{C}\mathbb{E}[W(\cdot,\xi-\overbar h_{-\xi})]-\mathbb{E}[\overbar\Lambda]\geq \frac{1}{C}W_{\rm hom}(\xi)-C,
\end{equation*}
so that also $W_{\rm hom}$ is almost even, which allows us to define the associated Sobolev-Orlicz space. Moreover, Assumption \ref{(A3)} holds since $W_{\rm hom}$ is convex, finite and superlinear at $+\infty$ (the last property ensures that $W_{\rm hom}^*$ is also convex and finite), while \ref{(A4)} or \ref{(A5)} (which were not needed in the proof of i) anyway) were proven in Lemma \ref{l.defW_hom}.

\vspace*{2.5mm}

iv) Under the additional assumption that there exists $s>1$ such that $su_{\e}\in {\rm dom}(F_{\e}(\w,\cdot,D))$ with $u_{\e}=g+v$ we know that $s(\nabla v+\nabla g)-\nabla g\in {\rm dom}(G_0)$. Then for any $\phi\in W_0^{1,\varphi_{\e}}(D)^m$ and $\delta>0$, convexity implies that 
\begin{align*}
	G_0(\nabla v+\delta\nabla\phi)&= G_0\left(\frac{1}{s}\left(s(\nabla v+\nabla g)-\nabla g\right)+\left(1-\frac{1}{s}\right)\left(-\nabla g+\left(1-\frac{1}{s}\right)^{-1}\delta\nabla\phi\right)\right)
	\\
	&\leq \frac{1}{s}G_0(s (\nabla v+\nabla g)-\nabla g)+\left(1-\frac{1}{s}\right)G_0\left(-\nabla g+\left(1-\frac{1}{s}\right)^{-1}\delta\nabla \phi\right).
\end{align*} 
Since $G_0$ is continuous in $-\nabla g$ with $G_0(-\nabla g)=\int_DW(\tfrac{x}{\e},0)\dx$, for $\delta$ small enough the right-hand side is finite and we can therefore conclude that $\ell_s(\delta\nabla\phi)\leq 0$ ($\ell_s$ as in the subdifferential representation in i)), which yields that $\ell_s(\nabla\phi)\leq 0$. Since this also holds for $-\varphi$, we conclude that $\ell_s(\nabla\phi)=0$ and therefore
\begin{equation*}
	\int_D\partial_{\xi}W(\tfrac{x}{\e},\nabla u)\nabla\phi-\langle f_{\e},\phi\rangle\dx=0\quad\text{ for all }\phi\in W_0^{1,\varphi_{\e}}(D)^m.	
\end{equation*}
The proof for the homogenized functional is the same.

\vspace*{2.5mm}

v) The convergence assertion in the strict convex case is a consequence of i) and iii) combined with Remark \ref{r.onconstrained} (ii).

\end{proof}

\subsection*{Acknowledgments}
The authors thank Thomas Ruf for the idea to encode superlinearity at infinity in the growth of the Fenchel-conjugate.

\appendix

\section{Convex envelopes of coercive radial functions}
\begin{lemma}\label{l.envelope}
Let $\ell:[0,+\infty)\to \R$ be coercive\footnote{That is, its sublevel sets are precompact.} and bounded from below. Then the convex envelope of $\R^k\ni \xi\mapsto L(\xi):= \ell(|\xi|)$ is given by the formula
\begin{equation*}
	\xi\mapsto\begin{cases}
		{\rm co}(\ell_{\rm lsc})(|\xi|) &\mbox{ if $|\xi|\geq r_{\rm max}$,}\\
		\ell_{\rm lsc}(r_{\rm max}) &\mbox{ if $|\xi|<r_{\rm max}$},
	\end{cases}
\end{equation*}
where $r_{\rm max}:=\max\{r\in [0,+\infty):\,\ell_{\rm lsc}(r)=\min_{t\in [0,+\infty)}\ell_{\rm lsc}(t)\}$ is the largest minimizer of the lower semicontinuous envelope $\ell_{\rm lsc}$ of $\ell$ and ${\rm co}(\ell_{\rm lsc})$ denotes the convex envelope of the function $\ell_{\rm lsc}$. In particular, it is of the form $\xi\mapsto \widetilde{\ell}(|\xi|)$ for some convex, monotone function $\widetilde{\ell}$.
\end{lemma}
\begin{proof}
First, we observe that $L$ is bounded from below and ${\rm co}(L)$ is finite, hence continuous. Therefore, by \cite[Remark 4.93]{FoLe} we have ${\rm co}(L)={\rm co}(L_{\rm lsc})$. The lower semicontinuous envelope of $L$ is given by $L_{\rm lsc}(\xi)=\ell_{\rm lsc}(|\xi|)$. By \cite[Theorem 3.8]{DM} the function $\ell_{\rm lsc}$ is still coercive. Moreover, it is also bounded from below. Hence we can assume without loss of generality that $\ell$ is already lower semicontinuous. Note that $r_{\rm max}$ exists due to the lower semicontinuity and coercivity of $\ell$. Define the function $\widetilde{\ell}:[0,+\infty)\to \R$ by 
\begin{equation*}
	\widetilde{\ell}(r)=\begin{cases}
		{\rm co}(\ell)(r) &\mbox{ if $r\geq r_{\rm max}$,}\\
		\ell(r_{\rm max}) &\mbox{ if $r<r_{\rm max}$}.
	\end{cases}
\end{equation*}
We argue that $\widetilde{\ell}$ is monotone and convex. Since constant functions are convex, we know that $\ell(r_{\rm max})\leq {\rm co}(\ell)(r)\leq \ell(r)$ for all $r\in [0,+\infty]$, which implies that $\ell(r_{\rm max})={\rm co}(\ell)(r_{\max})$. We next fix $r_1<r_2$ with $r_1,r_2\in [r_{\rm max},+\infty)$ and write $r_1=tr_{\rm max}+(1-t)r_2$ for some $t\in [0,1]$. Then by convexity of ${\rm co}(\ell)$ we have that
\begin{align*}
{\rm co}(\ell)(r_1)&\leq t\,{\rm co}(\ell)(r_{\rm max})+(1-t)\,{\rm co}(\ell)(r_2)=t\,\ell(r_{\rm max})+(1-t)\,{\rm co}(\ell)(r_2)
\\
&\leq t\,{\rm co}(\ell)(r_2)+(1-t)\,{\rm co}(\ell)(r_2)={\rm co}(\ell)(r_2), 
\end{align*} 
which shows the monotonicity of $\widetilde{\ell}$. To prove convexity of $\widetilde{\ell}$, it suffices to consider the case when $r_1\leq r_{\rm max}<r_2$ and $t\in [0,1]$ (the other cases being obvious). Then by the monotonicity of $\widetilde{\ell}$ we have that
\begin{align*}
\widetilde{\ell}(tr_1+(1-t)r_2)&\leq \widetilde{\ell}(t r_{\rm max}+(1-t)r_2)={\rm co}(\ell)(tr_{\rm max}+(1-t)r_2)
\\
&\leq t\,\underbrace{{\rm co}(\ell)(r_{\rm max})}_{=\ell(r_{\rm max})}+(1-t)\,{\rm co}(\ell)(r_2)=t\widetilde{\ell}(r_1)+(1-t)\widetilde{\ell}(r_2).
\end{align*}
Due to monotonicity and convexity of $\widetilde{\ell}$, the map $\xi\mapsto \widetilde{\ell}(|\xi|)$ is convex on $\R^k$ and therefore ${\rm co}(L)(\xi)\geq \widetilde{\ell}(|\xi|)$. It remains to show the reverse inequality. To this end, note that for every unit vector $v\in\R^k$ the restriction of ${\rm co}(L)$ to $\R v$ is also convex and satisfies ${\rm co}(L)(rv)\leq L(rv)=\ell(r)$ for all $r\in [0,+\infty)$. Hence ${\rm co}(L)(rv)\leq {\rm co}(\ell)(r)$ for all $r\in [0,+\infty]$. As $v$ was arbitrary, we deduce that ${\rm co}(L)(\xi)\leq {\rm co}(\ell)(|\xi|)$ for all $\xi\in\R^k$. However, we need to improve this bound for $|\xi|<r_{\rm max}$. Let $\xi\in\R^k$ be such that $|\xi|<r_{\rm max}$ and consider the line through $0$ and $\xi$, which intersects $\partial B_{r_{\rm max}}$ in two points $\xi_1$ and $\xi_2$ (take any line if $\xi=0$). Then we can write $\xi=t\xi_1+(1-t)\xi_2$ for some $t\in [0,1]$ and hence
\begin{equation*}
{\rm co}(L)(\xi)\leq t\, {\rm co}(L)(\xi_1)+(1-t)\,{\rm co}(L)(\xi_2)\leq t \,{\rm co}(\ell)(r_{\rm max})+(1-t)\,{\rm co}(\ell)(r_{\rm max})=\ell(r_{\rm max})=\widetilde{\ell}(|\xi|),
\end{equation*}
which proves the claim.
\end{proof}

\section{Measurability}\label{app:2}

Here we prove the following general measurability result that in particular covers our integrand $W$ satisfying Assumption \ref{a.1}.
\begin{lemma}\label{l.measurable}
	Let $W:\Omega\times\R^d\times\R^{m\times d}\to [0,+\infty]$ be $\mathcal{F}\otimes\mathcal{L}^d\times\mathcal{B}^{m\times d}$-measurable, $\xi\in\R^{m\times d}$ and $O\subset\R^d$ be a bounded, open set. If the function
	\begin{equation*}
		\mu_{\xi}(\w,O)=\inf_{u\in W^{1,1}_0(O)^m}\int_OW(\w,x,\xi+\nabla u)\dx
	\end{equation*}
	is almost surely finite, then it is $\mathcal{F}$-measurable. Moreover, if for almost every $\w\in\Omega$ there exists a minimizer in $W^{1,1}_0(O)^m$, then there exists a measurable function $u:\Omega\to W^{1,1}_0(O)^m$ such that almost surely
	\begin{equation*}
		\mu_{\xi}(\w,O)=\int_O W(\w,x,\xi+\nabla u(\w))\dx.
	\end{equation*}
\end{lemma}
\begin{proof}
	Redefining $W(\w,x,\xi)=|\xi|^2$ on the set of $\w$, where $\mu_{\xi}(\w,O)$ is not finite (and when no minimizer exists in the second case), we can assume without loss of generality that all properties holds for all $\w\in\Omega$. Note that this is possible since the modified integrand is still jointly measurable due to the completeness of the probability space. We first prove that the functional $(\w,u)\mapsto \int_OW(\w,x,\xi+\nabla u)\dx$ is $\mathcal{F}\otimes\mathcal{B}(W_0^{1,1}(O)^m)$-measurable. By truncation, we can assume without loss of generality that $W$ is bounded. If $W$ is additionally continuous in the third variable, then the joint measurability is a consequence of Fubini's theorem (which shows measurability in $\w$) and continuity of the functional with respect to strong convergence in $W^{1,1}_0(O)^m$. Indeed, joint measurability then holds due to the separability of $W^{1,1}_0(O)^m$, which ensures joint measurability of Carath\'eodory-functions.
	
	To remove the continuity assumption on $W$, we use a Monotone Class Theorem for functions. To this end consider the classes of functions defined as
	\begin{align*}
		\mathcal{C}:=&\{h:\Omega\times \R^d\times \R^{m\times d}\to \R,\,h(\w,x,\xi)=h_1(\w)h_2(x)g(\xi),\,h_1,h_2,g\text{ bounded, }
		\\
		&\quad\quad h_1\text{ $\mathcal{F}$-measurable},\,h_2\text{ $\mathcal{L}^d$-measurable},\,g\text{ continuous}\}
	\end{align*}
	and 
	\begin{align*}
		\mathcal{R}:=&\{h:\Omega\times \R^d\times \R^{m\times d}\to \R,\,h\text{ bounded and $\F\otimes\mathcal{L}^d\otimes\mathcal{B}^{m\times d}$-measurable, }
		\\
		&\qquad \text{such that}\; (\w,u)\mapsto\int_Oh(\w,x,\xi+\nabla u)\dx \text{ is  $\F\otimes\mathcal{B}(W^{1,1}_0(O)^m)$-measurable}\}.
	\end{align*}
	Note that if $h\in \mathcal{C}$, then $h$ is $\F\otimes\mathcal{L}^d\otimes\mathcal{B}^{m\times d}$-measurable, thus the argument above shows that $\mathcal{C}\subset\mathcal{R}$. Moreover, $\mathcal{R}$ contains the constant functions, is a vector space of bounded functions, and is closed under uniformly bounded, increasing limits. Finally, the set $\mathcal{C}$ is closed under multiplication. Thus \cite[Chapter I, Theorem 21]{DeMe} ensures that $\mathcal{R}$ contains all bounded functions that are measurable with respect to the $\sigma$-algebra generated by $\mathcal{C}$. By definition of $\mathcal{C}$  this $\sigma$-algebra coincides with $\mathcal{F}\otimes \mathcal{L}^d\otimes\mathcal{B}^{m\times d}$. 
	
	Having the joint measurability of $(\w,u)\mapsto \int_OW(\w,x,\xi+\nabla u)\dx$ at hand, the measurability of the optimal value function  $\mu_{\xi}(\cdot,O)$ follows from the measurable projection theorem. Indeed, for every $t\in\R$ we know from joint measurability that
	\begin{equation}\label{c:ultima}
		\left\{(\w,\varphi)\in \Omega\times W^{1,1}_0(O)^m:\,\int_OW(\w,x,\xi+\nabla u)\dx<t\right\}\in\F\otimes\mathcal{B}(W_0^{1,1}(O)^m).
	\end{equation}
	By assumption $(\Omega, \F, \mathbb P)$ is a complete probability space. Since $W_0^{1,1}(O,\R^m)$ is a complete, separable, metric space, the projection theorem \cite[Theorem 1.136]{FoLe} yields the $\F$-measurability of the projection of \eqref{c:ultima} onto $\Omega$. Therefore we have
	\begin{equation*}
		\left\{\w\in \Omega:\,\inf_{u\in W_0^{1,1}(O)^m}\int_OW(\w,x,\xi+\nabla u)\dx=\mu_{\xi}(\w,O)<t\right\}\in\mathcal{F},
	\end{equation*} 
	which proves the $\F$-measurability of $\mu_{\xi}(\cdot,A)$. To show the existence of a measurable selection of minimizers, define the multi-valued map $\Gamma:\Omega\rightrightarrows W^{1,1}_0(O)^m$ by
	\begin{equation*}
		\Gamma(\w)=\left\{u\in W^{1,1}_0(O)^m:\,\int_OW(\w,x,\xi+\nabla u)\dx=\mu_{\xi}(\w,O)\right\}.
	\end{equation*}
	By the measurability of $\mu_{\xi}(\cdot,O)$ and the joint measurability of the functional, we see that the graph of the this multi-function is $\mathcal{F}\otimes\mathcal{B}(W_0^{1,1}(O)^m)$-measurable. Due to the completeness of $(\Omega,\mathcal{F},\mathbb{P})$ and the separability and completeness of $W^{1,1}_0(O)^m$,  we can now apply Aumann's measurable selection theorem \cite[Theorem 6.10]{FoLe} and conclude the proof.
\end{proof}

\section{Approximation results in the vectorial case}
In this section we extend \cite[Proposition 2.6, Chapter X]{EkTe} to the vectorial setting, a result that has already been used in the literature, but the proof in \cite{EkTe} is only valid for scalar functions. For (probably) technical reasons, we need to assume in addition the analogue of \ref{(A4)} in the homogeneous setting or the lower bound $W(\cdot)\geq |\cdot|^{p}$ for some $p>d-1$. Let us emphasize that in the literature the result was used under the stronger assumption $p>d$.
\begin{lemma}\label{l.EkTemApprox}
Let $W:\R^{m\times d}\to [0,+\infty)$ be convex and $D\subset\R^d$ be a bounded, open set with Lipschitz boundary. Assume that one of the following assumptions holds true:
\begin{itemize}
\item [(1)] there exists $C<+\infty$ such that for all $\xi\in\R^{m\times d}$ and all $\widetilde{\xi}\in\R^{m\times d}$ with $e_j^T(\xi-\widetilde{\xi})\in \{0,e_j^T\xi\}$ for all $1\leq j\leq m$ it holds that
\begin{equation}\label{eq:A4appendix}
W(\widetilde{\xi})\leq C(W(\xi)+1).
\end{equation}
\item [(2)]  there exists $p>d-1$ such that
\begin{equation}\label{eq:pgrowthappendix}
W(\xi)\geq |\xi|^p\qquad\mbox{for all $\xi\in\R^{m\times d}$.}
\end{equation}
\end{itemize}
Then for any $u\in W^{1,1}_0(D)^m$ there exists a sequence $u_n\in C^{\infty}_c(D)^m$ such that $u_n\to u$ strongly in $W^{1,1}(D)^m$ and
\begin{equation*}
\lim_{n\to +\infty}\int_D W(\nabla u_n(x))\dx=\int_DW(\nabla u(x))\dx.
\end{equation*}
\end{lemma}
\begin{proof}
It suffices to treat the case when 
\begin{equation*}
\int_D W(\nabla u)\dx<+\infty,
\end{equation*}
since otherwise the statement reduces to well-known density results in $W^{1,1}_0$. We only prove that it is sufficient to consider $u$ with compact support in $D$. Once, this is established the statement follows by a routine regularization by convolution (see \cite{EkTe} or \cite[Lemma~3.6]{Mue87}).


We first treat the more involved case \eqref{eq:pgrowthappendix} and explain the necessary modifications/simplifications for \eqref{eq:A4appendix} afterwards. Let $u\in W_0^{1,p}(D)^m$ be given. First extend $u$ to be zero outside $D$, so that $u\in W^{1,p}(\R^d)^m$. For every $x\in D$ we consider a ball $B_{r_x}(x)\subset\subset D$, while for $x\in\partial D$ the Lipschitz regularity of $\partial D$ implies that (up to an Euclidean motion) there exists a cylinder $C_{x}= B^{d-1}_{r'_x}(0)\times (-h_x,h_x)$ with $x\in C_x$ and 
\begin{equation}\label{eq:boundarydescription}
D\cap C_{x}=\{(y',y_d)\in C_x:\,y_d<\psi_{x}(y')\}
\end{equation}
for some Lipschitz-function $\psi_{x}:B^{d-1}_{r'_x}(0)\to(-h_x,h_x)$. Up to reducing $r_{x'}$, we may assume that 
\begin{equation}\label{eq:nottouching}
\psi_x(B^{d-1}_{r'_x}(0))\subset\subset (-h_x,h_x).
\end{equation}
Choose then $r_x<\min\{r'_x,h_x\}$ such that $B_{r_x}(x)\subset\subset C_x$ and such that the Lipschitz-constant $L_x$ of $\psi_x$ satisfies
\begin{equation}\label{eq:rx_small}
0<2L_xr_x\leq h_x+\inf_{|y'|\leq r'_x}\psi_x(y'),
\end{equation}
which is possible due to \eqref{eq:nottouching}. Due to the compactness of $\overline{D}$, we find a finite family of above balls $B_i=B_{r_{x_i}}(x_i)$ ($1\leq i\leq N$) that cover $\overline{D}$. Let us emphasize that these balls will be fixed throughout the rest of the proof, so we omit the dependence on the radii $r_{x_i}$ or the number $N$ of certain quantities. For interior points $x_i\in D$, we define $z_i=x_i$, while for points $x_i\in\partial D$ we choose $z_i\in\R^d$ such that in the local coordinates we have $z_i=(0,-h_{x_i})$ (i.e., at the bottom of the local graph representation). Now let $0<\rho_k<1$ be such that $\lim_k\rho_k=1$ and for any $1\leq i\leq N$ we define
\begin{equation*}
u_{k,i}(x)=\rho_ku(z_{i}+\tfrac{1}{\rho_k}(x-z_{i})).
\end{equation*}
Since $\rho_k\to 1$, it holds that $u_{k,i}\to u$ in $W^{1,p}(\R^d)^m$ as $k\to +\infty$. 
Next, let $(\phi_{i})_{i=0}^N$ be a smooth partition of unity subordinated to the cover $\{\R^d\setminus\overline{D},(B_{i})_{i=1}^N\}$ of $\R^d$ (note that we work on the 'manifold' $\R^d$ and therefore ${\rm supp}(\phi_{i})$ is compactly contained in $B_{i}$ and $\phi_{0}$ vanishes on $D$). We build an ad hoc Lipschitz partition of unity as follows: choose $\delta_0>0$ such that for each $1\leq i\leq N$ we have ${\rm supp}(\phi_i)\subset\subset B_{(1-\delta_0)r_{x_i}}(x_i)$ and then use Lemma \ref{L:optim} (with $\rho=1$ and $\delta=\delta_0$ that are considered as fixed for the rest of the proof) for the finite family of functions $\{u_{k,j}-u\}_{j=1}^N\subset W^{1,p}(B_i,\R^m)$ to obtain $\eta_{k,i}\in W_0^{1,\infty}(B_i)$ such that $0\leq\eta_{k,i}\leq 1$,
\begin{equation*}
\eta_{k,i}=1\quad\text{ on }{\rm supp}(\varphi_i),\qquad\qquad \|\nabla\eta_{k,i}\|_{L^{\infty}(B_i)}\leq C .
\end{equation*}
and for all $1\leq j\leq N$
\begin{equation}\label{eq:productbound}
\|\nabla\eta_{k,i}\otimes (u_{k,j}-u)\|_{L^{\infty}(B_i)}\leq C\|u_{k,j}-u\|_{W^{1,p}(B_i)}\overset{k\to +\infty}{\longrightarrow}0.
\end{equation}
Similar to \cite[Theorem 2]{Koch22}, we now define the Lipschitz partition of unity as follows: for $k\in\N$ we set
\begin{equation*} \varphi_{k,0}=\frac{\phi_0}{\phi_0+\sum_{j=1}^N\eta_{k,j}},\quad \varphi_{k,i}=\frac{\eta_{k,i}}{\phi_0+\sum_{j=1}^N\eta_{k,j}}.
\end{equation*}
Here the denominator is always $\geq 1$ since $\eta_{k,j}=1$ on ${\rm supp}(\phi_j)$. Therefore the gradient of $\varphi_{k,i}$ ($1\leq i\leq N$) satisfies for any vector $a\in\R^m$ the pointwise estimate
\begin{equation}\label{eq:grad_control}
|\nabla\varphi_{k,i}\otimes a|\leq\frac{\left|\nabla\eta_{k,i}\otimes a\right|}{\phi_0+\sum_{i=1}^N\eta_{k,i}}+\frac{\eta_{k,i}\left|\nabla\phi_0\otimes a+\sum_{j=1}^N\nabla\eta_{k,j} \otimes a\right|}{\left(\phi_0+\sum_{i=1}^N\eta_{k,i}\right)^2}\leq |\nabla\phi_0\otimes a|+2\sum_{j=1}^N|\nabla\eta_{k,j}(x)\otimes a|.
\end{equation}
We then set
\begin{equation*}
u_{k}(x)=\sum_{i=1}^{N}\varphi_{k,i}(x)u_{k,i}(x),
\end{equation*}
for which the gradient on $D$ (here $\varphi_{k,0}$ vanishes) can be expressed as
\begin{equation*}
\nabla u_k=\sum_{i=1}^N\varphi_{k,i}\nabla u_{k,i}+\sum_{i=1}^N\nabla\varphi_{k,i}\otimes(u_{k,i}-u).
\end{equation*}
As a convex combination, we still have that $u_k\to u$ in $L^1(D)^m$, while for gradients the estimate \eqref{eq:grad_control} and the uniform gradient bound for $\eta_{k,i}$ imply that
\begin{align*}
\int_{D}|\nabla u_{k}-\nabla u|\dx&\leq \int_{D}\left|\sum_{i=1}^N(\nabla u_{k,i}-\nabla u)\varphi_{k,i}\right|\dx+\int_{D}\left|\sum_{i=1}^N\nabla \varphi_{k,i}\otimes(u_{k,i}-u)\right|\dx
\\
&\leq C\sum_{i=1}^N\int_{D}|u_{k,i}-u|+|\nabla u_{k,i}-\nabla u|\dx\to 0\quad\text{ as $k\to+\infty$ }. 
\end{align*} 
Finally, due to the convexity of $W$, for any $t\in (0,1)$ it holds that
\begin{equation}\label{eq:convexityused}
\int_{D}W(t\nabla u_{k})\dx\leq\sum_{i=1}^N\int_{D}t\varphi_{k,i}W(\nabla u_{k,i}) \dx+(1-t)\int_{D}W\left(\sum_{i=1}^N\frac{t}{1-t}\nabla\varphi_{k,i}\otimes(u_{k,i}-u)\right)\dx.
\end{equation}
Let us discuss the two right-hand side integrals separately. By a change of variables we have
\begin{equation*}
\sum_{i=1}^N\int_{D} t\varphi_{k,i}(x)W(\nabla u_{k,i})\dx=\rho_k^dt\int_{\R^d}\sum_{i=1}^N(\chi_D\varphi_{k,i})(z_{i}+\rho_k(y-z_{i}))W(\nabla u(y))\,\mathrm{d}y.
\end{equation*}
Due to \eqref{eq:grad_control} and the uniform gradient bound for $\eta_{k,i}$, we know that the functions $\varphi_{k,i}$ are equi-Lipschitz with respect to $k$ and $i$, so that for $z_i+\rho_k(y-z_i)\in D$ we have the estimate
\begin{equation*}
\sum_{i=1}^N\varphi_{k,i}(z_i+\rho_k(y-z_i))\leq 1+C\max_{1\leq i\leq N}|z_i+\rho_k(y-z_i)-y|\leq 1+C(1-\rho_k)|z_i-y|\leq 1+C(1-\rho_k).
\end{equation*}
Inserting the bound into the previous equality, we deduce that
\begin{equation*}
\sum_{i=1}^N\int_{D} t\varphi_{k,i}(x)W(\nabla u_{k,i})\dx\leq\rho_k^d(1+C(1-\rho_k))\int_{\R^d}\max_{1\leq i\leq N}\chi_D(z_{i}+\rho_k(y-z_{i}))W(\nabla u(y))\,\mathrm{d}y.
\end{equation*}
The pre-factor of the right-hand side integral converges to $1$, while for the integrand we can use the dominated convergence theorem to infer that $\max_{1\leq i\leq N}\chi_D(z_i+\rho_k(\cdot-z_i))W(\nabla u)\to\chi_DW(\nabla u)$ in $L^1(\R^d)$ (recall that $D$ is open with $|\partial D|=0$). Hence we deduce that
\begin{equation*}
\limsup_{k\to +\infty}\sum_{i=1}^N\int_{D} t\varphi_{k,i}(x)W(\nabla u_{k,i})\dx \leq \int_{D}W(\nabla u(y))\dy.
\end{equation*}
It remains to show that the second integral in \eqref{eq:convexityused} is negligible. Applying \eqref{eq:grad_control} with $a=u_{k,i}(x)-u(x)$ and inserting \eqref{eq:productbound} in the corresponding right-hand side, due to the fact that $\nabla\phi_0=0$ on $D$ we find that
\begin{equation*}
\left|\sum_{i=1}^N\frac{t}{1-t}\nabla\varphi_{k,i}(u_{k,i}-u)\right|\leq\frac{2}{1-t}\sum_{i,j=1}^N|\nabla\eta_{k,j}\otimes (u_{k,i}-u)|\leq \frac{C}{1-t}\sum_{i,j=1}^N\|u_{k,j}-u\|_{W^{1,p}(B_i)}\quad\text{ on }D.
\end{equation*}
The right-hand side converges to $0$ as $k\to +\infty$, so we find a constant $c=c_{t,D,u}$ such that pointwise on $D$
\begin{equation*}
W\left(\sum_{i=1}^N\frac{t}{1-t}\nabla\varphi_{k,i}(u_{k,i}-u)\right)\leq \sup_{|\eta|\leq c}W(\eta)<+\infty,
\end{equation*}
where we used that $W$ is bounded on bounded sets by continuity. The dominated convergence theorem yields that
\begin{equation*}
\lim_{k\to +\infty}(1-t)\int_{D}W\left(\sum_{i=1}^N\frac{t}{1-t}\nabla\varphi_i(u_{k,i}-u)\right)\dx=(1-t)|D|W(0).
\end{equation*}
In total, starting from \eqref{eq:convexityused} we deduce the estimate
\begin{equation*}
\limsup_{t\uparrow 1}\limsup_{k\to+\infty}\int_DW(t\nabla u_k)\dx\leq \int_D W(\nabla u)\dx.
\end{equation*}
Since $tu_k\to tu$ in $W^{1,1}(D)^m$, we can use a diagonal argument with respect to $k$ and $t$ to find a sequence $\widetilde{u}_k$ such that $\widetilde{u}_k\to u$ in $W^{1,1}(D)^m$ and, combined with Fatou's lemma,
\begin{equation*}
\lim_{k\to +\infty}\int_D W(\nabla \widetilde{u}_k)\dx=\int_D W(\nabla u)\dx.
\end{equation*}
Moreover, it holds that ${\rm supp}(\widetilde{u}_k)={\rm supp}(u_{k'})$ for some $k'\in\N$, so it remains to show that all $u_k$ have compact support in $D$. Let $\delta_k\ll 1$ and consider $x\in D$ such that $\dist(x,\partial D)\leq\delta_k$ and a ball $B_i$ with $x\in B_i$. We show that $\varphi_{k,i}(x)u_{k,i}(x)=0$ for such $i$ and $\delta_k$ small enough, which then implies that $u_k(x)=0$, so that $u_k$ has compact support. Since there are only finitely many sets in the covering and 'interior' balls are compactly contained in $D$, for $\delta_k$ small enough we know from our construction of the covering and \eqref{eq:boundarydescription} that, up to a Euclidean motion, with the corresponding cylinders $C_i=C_{x_i}$ and radii $r_i'=r_{x_i}'$ we can write
\begin{equation*}
D\cap C_i=\{(y',y_d)\in C_i:\,y_d<\psi_i(y')\}
\end{equation*}
for a Lipschitz function $\psi_i:B_{r'_{i}}^{d-1}(0)\to (-h_i,h_i)$. Since ${\rm supp}(\varphi_{k,i})={\rm supp}(\eta_{k,i})\subset B_i\subset\subset C_i$, there exists $0<\eta<\min_i r_i$ such that $\varphi_{k,i}(x)=0$ whenever $\dist(x,\partial C_i)\leq\eta$. Hence we can assume that $\dist(x,\partial C_i)>\eta$. We will show that for $k$ large enough  (independent of $x=(x',x_d)$) we have $z_i+\frac{1}{\rho_k}(x-z_i)\in C_i\setminus \overline{D}$, which then implies that $u_{k,i}(x)=0$. Since $x\in C_i$ and $\dist(x,\partial C_i)>\eta$ and the $z_i$ are fixed, it follows that for $\rho_k$ is sufficiently close to $1$ we have $z_i+\frac{1}{\rho_k}(x-z_i)\in C_i$. Hence, in order to show that this point does not belong to $D$, it suffices to show that in the local coordinates (where $z_i=(0,-h_i)$) we have 
\begin{equation}\label{eq:notinD}
\psi_i\left(\frac{1}{\rho_k}x'\right)< -h_i+\frac{1}{\rho_k}(x_d+h_i).  
\end{equation}
Let $d_x\in\partial D$ be such that $|x-d_x|=\dist(x,\partial D)$. Note that $d_x\in C_i$ (otherwise the line from $x$ to $d_x$ intersects $\partial C_i$ at a distance less than $\delta_k\ll\eta$), so that we can write $d_x=(y',\psi_i(y'))$ for some $y'\in B_{r'_i}^{d-1}(0)$. To show \eqref{eq:notinD}, let $L_i$ be the Lipschitz constant of $\psi_i$. Then we can estimate
\begin{align*}
\psi_i\left(\frac{1}{\rho_k}x'\right)-\left(\frac{1}{\rho_k}-1\right)h_i-\frac{1}{\rho_k}x_d&\leq \underbrace{\psi_i\left(y'\right)-x_d}_{\leq |d_x-x|\leq\delta_k}+L_i\left|y'-\frac{1}{\rho_k}x'\right|-\left(\frac{1}{\rho_k}-1\right)h_i-\left(\frac{1}{\rho_k}-1\right)x_d
\\
&\leq\delta_k +L_i\left(\frac{1}{\rho_k}-1\right)|x'|+L_i\underbrace{|y'-x'|}_{\leq |d_x-x|\leq\delta_k}-\left(\frac{1}{\rho_k}-1\right)(h_i+x_d)
\\
&\leq (L_i+1)\delta_k+\left(\frac{1}{\rho_k}-1\right)(L_i|x'|-h_i-x_d)
\\
&\leq (L_i+2)\delta_k+\left(\frac{1}{\rho_k}-1\right)(L_i|x'|-h_i-\psi(y'))
\\
&\leq (L_i+2)\delta_k+\frac{1}{2}\left(\frac{1}{\rho_k}-1\right)(\underbrace{-h_i-\inf_{|y'|\leq r'_i}\psi(y')}_{=:\kappa_i\overset{\eqref{eq:nottouching}}{<}0}),
\end{align*}
where we used \eqref{eq:rx_small} in the last estimate. Choosing $\delta_k\ll \left(\frac{1}{\rho_k}-1\right)$, it clearly follows that the right-hand side is negative and we conclude the proof for the case \eqref{eq:pgrowthappendix}.

If we assume \eqref{eq:A4appendix} instead, we can show that the analysis reduces to the case that $u\in L^{\infty}(D)^m$. Indeed, consider for $s\gg 1$ the componentwise truncation $T_su$ at level $s$ (cf. \eqref{eq:defTruncation}). Then $T_su\in W^{1,1}_0(D)^m$ and $T_su\to u$ in $W^{1,1}(D)^m$ as $s\to +\infty$ by the dominated convergence theorem (both for the function and the gradient). Moreover, by the non-negativity of $W$ we have
\begin{equation*}
\int_{D}W(\nabla T_s u)\dx\leq \int_{D}W(\nabla u)\dx+\int_{\{|u|_{\infty}\geq s\}}W(\nabla T_s u)\dx,
\end{equation*}
where $|u|_{\infty}=\max_i|u_i|$. We show that the last integral vanishes as $s\to +\infty$, which then yields
\begin{equation*}
\limsup_{s\to +\infty}\int_{D}W(\nabla T_s u)\dx\leq \int_{D}W(\nabla u)\dx.
\end{equation*}
The reverse inequality for the $\liminf$ follows from Fatou's lemma since $W$ is lower semicontinuous and $\nabla T_su\to \nabla u$ pointwise almost everywhere. For said integral, we use \eqref{eq:A4appendix} to estimate
\begin{equation*}
\int_{\{|u|_{\infty}\geq s\}}W(\nabla T_s u)\dx\leq C\int_{\{|u|_{\infty}\geq s\}}W(\nabla u)\dx+C|\{|u|_{\infty}\geq s\}|.
\end{equation*} 
The last term vanishes as $s\to +\infty$ since $|u|_\infty,W(\nabla u)\in L^1(D)$. Hence we have shown that
\begin{equation*}
\lim_{s\to+\infty}\left(\|T_su-u\|_{W^{1,1}(D)}+\left|\int_D W(\nabla T_su)-W(\nabla u)\dx\right|\right)=0
\end{equation*}
and by a diagonal argument it suffices to show the claim for $u\in L^{\infty}(D)^m$. In this case we can argue is in the first part (directly working with the $k$-independent partition of unity $\{\phi_i\}_{i=1}^N$). The only difference occurs when proving that the second integral in \eqref{eq:convexityused} is negligible. Since now $u\in L^{\infty}(\R^d)^m$ and $\|u_{k,i}\|_{L^{\infty}(\R^d)}\leq \|u\|_{L^{\infty}(\R^d)}$, we find a constant $c=c_{t,D,u}$ such that
\begin{equation*}
W\left(\sum_{i=1}^N\frac{t}{1-t}\nabla\phi_{i}(u_{k,i}-u)\right)\leq \sup_{|\eta|\leq c}W(\eta)<+\infty.
\end{equation*}
Then again
\begin{equation*}
\lim_{k\to +\infty}(1-t)\int_{D}W\left(\sum_{i=1}^N\frac{t}{1-t}\nabla\phi_i(u_{k,i}-u)\right)\dx=(1-t)|D|W(0)
\end{equation*}
and we conclude the proof as in the case \eqref{eq:pgrowthappendix}.

\end{proof}

\end{document}